\documentclass[11pt,oneside,english]{amsart}
\usepackage[latin9]{inputenc}
\usepackage[a4paper]{geometry}
\usepackage{xcolor}
\usepackage{babel}
\usepackage{amstext}
\usepackage{amsthm}
\usepackage{amssymb}
\usepackage{setspace}
\usepackage{wasysym}
\usepackage[unicode=true,pdfusetitle,
 bookmarks=true,bookmarksnumbered=false,bookmarksopen=false,
 breaklinks=false,pdfborder={0 0 1},backref=false,colorlinks=false]
 {hyperref}

\makeatletter
\numberwithin{equation}{section}
\numberwithin{figure}{section}
\theoremstyle{remark}
\newtheorem*{acknowledgement*}{\protect\acknowledgementname}
\theoremstyle{plain}
\newtheorem{thm}{\protect\theoremname}[section]
\theoremstyle{definition}
\newtheorem{problem}[thm]{\protect\problemname}
\theoremstyle{definition}
\newtheorem{defn}[thm]{\protect\definitionname}
\theoremstyle{definition}
\newtheorem{xca}[thm]{\protect\exercisename}
\theoremstyle{remark}
\newtheorem{rem}[thm]{\protect\remarkname}
\theoremstyle{plain}
\newtheorem{lem}[thm]{\protect\lemmaname}
\theoremstyle{definition}
\newtheorem{example}[thm]{\protect\examplename}
\theoremstyle{plain}
\newtheorem{cor}[thm]{\protect\corollaryname}
\theoremstyle{plain}
\newtheorem{prop}[thm]{\protect\propositionname}
\theoremstyle{plain}
\newtheorem{fact}[thm]{\protect\factname}
\theoremstyle{definition}
\newtheorem*{example*}{\protect\examplename}
\theoremstyle{remark}
\newtheorem*{claim*}{\protect\claimname}
\theoremstyle{plain}
\newtheorem{conjecture}[thm]{\protect\conjecturename}
\theoremstyle{remark}
\newtheorem*{notation*}{\protect\notationname}

\makeatother

\providecommand{\acknowledgementname}{Acknowledgement}
\providecommand{\claimname}{Claim}
\providecommand{\conjecturename}{Conjecture}
\providecommand{\corollaryname}{Corollary}
\providecommand{\definitionname}{Definition}
\providecommand{\examplename}{Example}
\providecommand{\exercisename}{Exercise}
\providecommand{\factname}{Fact}
\providecommand{\lemmaname}{Lemma}
\providecommand{\notationname}{Notation}
\providecommand{\problemname}{Problem}
\providecommand{\propositionname}{Proposition}
\providecommand{\remarkname}{Remark}
\providecommand{\theoremname}{Theorem}

\begin{document}
\title{Word maps and random words}
\author[E.\ Breuillard]{Emmanuel Breuillard} \address{Emmanuel Breuillard\hfill\break 	Mathematical Institute \hfill\break University of Oxford \hfill\break Oxford OX1 3LB, United Kingdom} \email{breuillard@maths.ox.ac.uk}

\author[I.\ Glazer]{Itay Glazer} \address{Itay Glazer \hfill\break  Department of Mathematics \hfill\break Technion \textendash{} Israel Institute of Technology  \hfill\break  Haifa, Israel} \email{itayglazer@gmail.com}
\begin{abstract}
We discuss some recent results by a number of authors regarding word
maps on algebraic groups and finite simple groups, their mixing properties
and the geometry of their fibers, emphasizing the role played by equidistribution
results in finite fields via recent advances on character bounds and
non-abelian arithmetic combinatorics. In particular, we discuss character varieties of random groups. In the last section, we give
a new proof of a recent theorem of Hrushovski about the geometric
irreducibility of the generic fibers of convolutions of dominant morphisms
to simply connected algebraic groups.

These notes stem out of lectures given by the authors in Oxford, and
by the first author in ICTS Bangalore, in spring 2024. 
\end{abstract}

\maketitle
\pagenumbering{arabic}

\global\long\def\N{\mathbb{N}}%
\global\long\def\R{\mathbb{\mathbb{R}}}%
\global\long\def\Z{\mathbb{Z}}%
\global\long\def\val{\mathbb{\mathrm{val}}}%
\global\long\def\Qp{\mathbb{Q}_{p}}%
\global\long\def\Zp{\mathbb{\mathbb{Z}}_{p}}%
\global\long\def\ac{\mathbb{\mathrm{ac}}}%
\global\long\def\C{\mathbb{\mathbb{C}}}%
\global\long\def\Q{\mathbb{\mathbb{Q}}}%
\global\long\def\supp{\mathbb{\mathrm{supp}}}%
\global\long\def\VF{\mathbb{\mathrm{VF}}}%
\global\long\def\RF{\mathbb{\mathrm{RF}}}%
\global\long\def\VG{\mathbb{\mathrm{VG}}}%
\global\long\def\spec{\operatorname{Spec}}%
\global\long\def\Ldp{\mathbb{\mathcal{L}_{\mathrm{DP}}}}%
\global\long\def\sgn{\mathrm{sgn}}%
\global\long\def\id{\mathrm{Id}}%
\global\long\def\Sym{\mathrm{Sym}}%
\global\long\def\Vol{\mathrm{Vol}}%
\global\long\def\cyc{\mathrm{cyc}}%
\global\long\def\U{\mathrm{U}}%
\global\long\def\SU{\mathrm{SU}}%
\global\long\def\Wg{\mathrm{Wg}}%
\global\long\def\E{\mathbb{E}}%
\global\long\def\Irr{\mathrm{Irr}}%
\global\long\def\P{\mathbb{P}}%
\global\long\def\bh{\mathbf{h}}%
\global\long\def\Span{\operatorname{Span}}%
\global\long\def\pr{\operatorname{pr}}%
\global\long\def\sgn{\operatorname{sgn}}%
\global\long\def\tr{\operatorname{tr}}%
\global\long\def\lct{\operatorname{lct}}%
\global\long\def\sG{\mathsf{G}}%
\global\long\def\sW{\mathsf{W}}%
\global\long\def\sX{\mathsf{X}}%
\global\long\def\sY{\mathsf{Y}}%
\global\long\def\sZ{\mathsf{Z}}%
\global\long\def\sH{\mathsf{H}}%
\global\long\def\sV{\mathsf{V}}%
\global\long\def\sT{\mathsf{T}}%
\global\long\def\v{\mathsf{v}}%
\global\long\def\d{\mathsf{d}}%
\global\long\def\GG{\underline{G}}%
\global\long\def\PP{\underline{P}}%
\global\long\def\g{\mathfrak{g}}%
\global\long\def\l{\mathfrak{l}}%
\global\long\def\p{\mathfrak{p}}%
\global\long\def\n{\mathfrak{n}}%
\global\long\def\m{\mathfrak{m}}%
\global\long\def\t{\mathfrak{t}}%
\global\long\def\q{\mathfrak{q}}%
\global\long\def\QQ{\underline{Q}}%
\global\long\def\LL{\underline{L}}%
\global\long\def\NN{\underline{N}}%
\global\long\def\SL{\mathrm{SL}}%
\global\long\def\Hom{\operatorname{Hom}}%

\raggedbottom

\section{Introduction}

A group $G$ together with a word $w$ in $r$ letters (and their
inverses) give rise to the associated \emph{word map} $w_{G}:G^{r}\to G$
where an $r$-tuple of elements in $G$ is sent to the value of the
word evaluated at the tuple. When $G$ is a field $(K,+)$, this is
nothing else than a linear form with integer coefficients in $r$
variables. For non-abelian $G$, word maps are more subtle objects
and a lot of effort has been devoted in the last decades to unravel
some of their properties (see e.g.~the surveys \cite{Seg09,Sha13,GKP18},
and the references therein). Here is a sample of questions that arise
naturally in this context: is $w_{G}$ surjective? if not, can every
element of $G$ be written as a product of a small number of $w_{G}$
values? If $G$ is finite, how close to uniformly distributed is $w_{G}(g_{1},\ldots,g_{r})$
where the $g_{i}$ are chosen independently at random in $G$? What
is the size or the dimension of a fiber $w_{G}^{-1}(g)$? These questions
can be asked for finite groups, and in particular large finite simple
groups, for compact Lie groups, or also for algebraic groups in arbitrary
characteristic. Surveying the large body of works around these questions
is out of the scope of this article. Rather, we propose here to present
a brief introduction to these topics and the diverse methods they
bring about through the lens of the following three concrete results: 
\begin{enumerate}
\item The proof by Larsen, Shalev and Tiep \cite{LST19} that every word
map $w_{G}$ on a large finite simple group has an $L^{\infty}$-mixing
time which is bounded by a number $t_{\infty}(w)$ depending only
on $w$. 
\item The proof by Becker, Breuillard and Varj\'{u} \cite{BBV} of a dimension
formula for the fibers above the identity element of generic word
maps, 
\item A new analytic proof of a result of Hrushovski \cite{Hru24}, showing
that the convolution of two dominant maps to a simply connected algebraic
group has a geometrically irreducible generic fiber. 
\end{enumerate}
Given a map between two algebraic varieties, the Lang-Weil estimates,
which we recall in Section \ref{sec4}, provide a dictionary between
algebro-geometric notions (dominance, flatness, geometrically irreducible
generic fibers, etc.) and analytic counting estimates in finite fields
(size of the image, size of a fiber, approximate uniformity and boundedness
of the pushforward of the uniform measure, etc.). They will be essential
to the proofs. In fact, on the analytic side, a key role is played
by \emph{equidistribution} in finite fields. The large rank case in
$(1)$ combines recent advances by Larsen and Tiep \cite{LT24} on
sharp character bounds with an argument going back to \cite{LaS12}
proving an upper bound on the size of fibers of the word map. The
small rank case can be proved using the Lang-Weil estimates and using
$(3)$. We will give a proof of $(3)$ that makes use of harmonic
analysis on finite quasi-simple groups, as well as on bounds \cite{Kow07}
on exponential sums associated to arbitrary functions on an algebraic
variety that generalize Deligne's celebrated exponential sums estimates
\cite{Del74}, while Hrushovski's argument was purely model-theoretic.
Regarding $(2)$, essential to the proofs is the fact that random
walks equidistribute very rapidly in the finite simple groups whose
associated Cayley graphs are expander graphs. This expander property
has been established in many instances in the last decade or so, such
as in Bourgain\textendash Gamburd\textendash Sarnak \cite{BGS10},
often relying on methods from arithmetic combinatorics \cite{BG08a,BGT11b,Bre15,PS16}.

This article, which is based on lectures given in Bangalore and in
Oxford in 2024, is mostly expository and we have put the emphasis
on explaining some of the key ideas, sometimes working only on special
illustrative cases, rather than presenting complete proofs. In this
spirit, we have included various ``exercises'' along the way. It is
organized as follows. In Section \ref{sec2} we give further introductory
remarks, recall some landmark results regarding word maps, set up
our notation, and state $(1)$. In Section \ref{sec3} we sketch a
proof of the upper bound for the size of fibers of $w_{G}$ following
an argument of Larsen and Shalev \cite{LaS12}, and a proof of the
high rank case of $(1)$. In Section \ref{sec4} we discuss the Lang-Weil
estimates and make explicit the dictionary mentioned above. This is
then utilized to prove the low rank case of $(1)$. In Section \ref{sec5}
we gather general facts about representation and character varieties
of finitely presented groups and discuss Gromov's random group model.
In Section \ref{sec6} we discuss the expander property for finite
simple groups, and in Section \ref{sec7} we sketch the proof of $(2)$
and discuss the role of Chebotarev's density theorem. In the final
section \ref{sec8}, we prove $(3)$ and state further applications
to algebro-geometric properties of word maps.

\subsection{Conventions}
\begin{itemize}
\item We write $\overline{K}$ for the algebraic closure of a field $K$,
and $\underline{G}$ for an algebraic $K$-group. 
\item Given a field extension $K\leq K'$, and a finite type $K$-scheme
$X$, we denote by $X_{K'}$ the base change of $X$ with respect
to $\spec(K')\rightarrow\spec(K)$. Moreover, if $\varphi:X\rightarrow Y$
is a morphism of $K$-schemes, we denote by $\varphi_{K'}:X_{K'}\rightarrow Y_{K'}$
the corresponding base change to $K'$. 
\item We write $\mathbb{A}_{K}^{m}$ for the $m$-dimensional affine space,
as a $K$-scheme. 
\item Given an algebraic $K$-group $\underline{G}$ and a subset $S$ in
$\underline{G}(K)$, we write $\overline{\langle S\rangle}^{Z}$ for
the Zariski closure of the subgroup generated by $S$, which is an
algebraic subgroup $\underline{H}\leq\underline{G}$. 
\item Let $\mathcal{D}$ be a fixed set (possibly empty) of parameters (i.e.~the
given data). 
\begin{itemize}
\item Given functions $f,g:S\rightarrow\R$, possibly depending on $\mathcal{D}$,
we write $f(s)\gg_{\mathcal{D}}g(s)$ (and also $g=O_{\mathcal{D}}(f)$)
if $f(s)\geq C\cdot g(s)$ for some positive constant $C$ depending
on $\mathcal{D}$.
\item We write $O_{\mathcal{D}}(1)$ to indicate a constant depending only
on the data $\mathcal{D}$. In particular, by $O(1)$ we mean an absolute
constant. 
\end{itemize}
\item We write $\N$ for the set $\left\{ 1,2,\dots\right\} $, and $\Z_{\geq0}$
for $\N\cup\{0\}$. 
\end{itemize}
\begin{acknowledgement*}
We thank Udi Hrushovski for the discussion around Section $\mathsection$\ref{sec8}.
I.G.~was supported by ISF grant 3422/24. 
\end{acknowledgement*}
\setcounter{tocdepth}{1} 

\tableofcontents{}

\section{\label{sec2}Word maps on finite simple groups: probabilistic results }

\subsection{Motivation}

In 1770, Lagrange proved the famous four squares theorem; every natural
number can be represented as the sum of at most four square integers.
Later that year, Waring considered the following generalization, which
was confirmed by Hilbert 139 years later \cite{Hil909}: 
\begin{problem}[Waring's problem, 1770]
\label{prob:Warings problem}Can one find for every $k\in\N$, a
number $t(k)\in\N$, such that every $n\in\Z_{\geq0}$ can be represented
as $n=\sum_{i=1}^{t(k)}x_{i}^{k}$ of $t(k)$ for $x_{1},\dots,x_{t(k)}\in\Z_{\geq0}$?
\end{problem}

We now formulate Waring's problem in a slightly different language,
using the following definition. 
\begin{defn}[\cite{GH19,GH21}]
\label{def:abstract convolution}Let $\varphi:X\rightarrow G$, $\psi:Y\rightarrow G$
be maps from sets $X,Y$ to a (semi-)group $(G,\cdot_{G})$. Define
the \textit{convolution} $\varphi*\psi:X\times Y\rightarrow G$ by
\[
\varphi*\psi(x,y)=\varphi(x)\cdot_{G}\psi(y).
\]
Furthermore, we denote by $\varphi^{*t}:X^{t}\rightarrow G$ the \emph{$t$-th
convolution power}. 
\end{defn}

In this language, Waring's problem can be stated as follows; let $\varphi_{k}:\Z_{\geq0}\rightarrow(\Z_{\geq0},+)$
be the map $\varphi_{k}(x)=x^{k}$. Can one find $t(k)\in\N$ such
that $\varphi_{k}^{*t(k)}:\Z_{\geq0}^{t(k)}\rightarrow(\Z_{\geq0},+)$
is surjective?

More generally, analyzing the surjectivity of $\varphi^{*t}:X^{t}\rightarrow G$,
varying over different semi-groups $G$ and different maps $\varphi:X\rightarrow G$,
gives rise to a family of problems called \emph{Waring-type problems.}
In addition, instead of just asking whether $\varphi^{*t}$ is surjective,
one can ask in how many ways one can write $g\in G$ as $g=\varphi^{*t}(x_{1},\dots,x_{t})$?
or in other words, estimate the size of the fiber $(\varphi^{*t})^{-1}(g)$.
Such problems are called \emph{probabilistic Waring-type problems. }

In the next lectures we will focus on the special case when $\varphi$
is a word map and $G$ is a simple group (finite or algebraic). This
setting was extensively studied by Larsen, Liebeck, Shalev, Tiep,
and many others. We refer to an excellent survey of Shalev \cite{Sha13}.

\subsection{\label{subsec:Waring-type-problems}Waring type problems in the setting
of word maps}
\begin{defn}
Let $w(x_{1},...,x_{r})$ be a word in a free group $F_{r}$ (e.g.~$w=[x,y]$
or $w=x^{\ell}$). For any group $G$, we define a \emph{word map
\[
w_{G}:G^{r}\rightarrow G,\text{\,\,\,\,\,by\,\,\,\,}(g_{1},...,g_{r})\mapsto w(g_{1},...,g_{r}),\,\,\,\text{for }g_{1},\dots,g_{r}\in G.
\]
}
\end{defn}

Note that $w_{G}*w_{G}=(w*w)_{G}$, where $w*w$ denotes concatenation
of words with different letters. For example, $w_{\mathrm{com}}:=[x,y]=xyx^{-1}y^{-1}$
induces the commutator map $w_{\mathrm{com},G}:G^{2}\rightarrow G$,
and $w_{\mathrm{com}}*w_{\mathrm{com}}:=[x,y]\cdot[z,w]$.

Let us mention the state of the art for the Waring problem for word
maps on simple algebraic groups and finite simple groups. 
\begin{thm}[Borel \cite{Bor83}, and also Larsen \cite{Lar04}]
\label{thm:Borel}Let\textup{ $K$}\textup{\emph{ be an algebraically
closed field and let}}\emph{ $1\neq w\in F_{r}$. }hen for every connected
semisimple algebraic $K$-group, the map $w_{\underline{G}}:\underline{G}^{r}\rightarrow\underline{G}$
is dominant.

In particular, $\left(w*w\right)_{\underline{G}(K)}:\underline{G}(K)^{2r}\rightarrow\underline{G}(K)$
is surjective. 
\end{thm}

\begin{thm}[Larsen\textendash Shalev\textendash Tiep, \cite{LST11}]
For every $1\neq w\in F_{r}$, there exists $N(w)\in\N$ such that
for every finite simple group $G$, with $\left|G\right|>N(w)$, $(w*w)_{G}:G^{2r}\rightarrow G$
is surjective. 
\end{thm}

Certain words are in fact surjective over all finite simple groups,
so that no convolutions are needed. A notable example is the following
theorem, which answered a conjecture by Ore from the fifties. 
\begin{thm}[The Ore conjecture 1951, \cite{LOST10}]
If $G$ is a finite non-abelian simple group, then $w_{\mathrm{com},G}:G^{2}\rightarrow G$
is surjective.
\end{thm}

Here are a few more examples in other settings: 
\begin{enumerate}
\item \textbf{Compact simple Lie groups} \cite{HLS15}: Let $1\neq w\in F_{r}$
and let $G$ be a compact connected simple Lie group of high rank
$\mathrm{rk}(G)\gg_{w}1$. Then $(w*w)_{G}:G^{2r}\rightarrow G$ is
surjective. 
\item \textbf{Compact $p$-adic groups} \cite{AGKS13}: Let $1\neq w\in F_{r}$.
Then for every $n\geq2$ and every $p\gg_{n}1$, the map $(w^{*3})_{\mathrm{SL}_{n}(\Zp)}$
is surjective. More generally, one can take $G=\underline{G}(\Zp)$
for $\underline{G}$ simply connected, simple algebraic $\Q$-group.
\item \textbf{Arithmetic groups }\cite{AM19}: Let $1\neq w\in F_{r}$.
Then $w_{\mathrm{SL}_{n}(\Z)}^{*87}$ is surjective for $n\gg_{w}1$.
\item \textbf{Simple Lie algebras:} In \cite{BGKP12}, an analogue of Borel's
theorem (Theorem \ref{thm:Borel}) was shown for Lie algebra word
maps on semisimple Lie algebras $\mathfrak{g}$ under the additional
assumption that the word map is not identically zero on $\mathfrak{sl}_{2}$. 
\end{enumerate}

\subsection{Probabilistic Waring type problems for word maps}

In $\mathsection$\ref{subsec:Waring-type-problems} we saw that word
maps $w_{G}:G^{r}\rightarrow G$ in various settings, become surjective
after taking very few self-convolutions. In other words, any $g\in G$
can be written as $w_{G}^{*t}(g_{1},\dots,g_{rt})=g$ for some $g_{1},\dots,g_{rt}\in G$.
We now consider the case when $G$ is finite, and discuss \emph{in
how many ways} one can write $w_{G}^{*t}(g_{1},\dots,g_{rt})=g$,
or in other words: can one estimate the size of the fiber $(w_{G}^{*t})^{-1}(g)$?
This boils down to analyzing the random walk induced by $w$ on $G$.
We start by introducing basic notions and results from the theory
of random walks on finite groups. 

\subsubsection{\label{subsec:Random-walk-on finite groups}Random walk on finite
groups}

Let $G$ be a finite group, and denote by $\mu_{G}$ the uniform probability
measure, which can be identified with the constant function $\frac{1}{\left|G\right|}$
on $G$. Let $\mu$ be a probability measure on $G$. It will later
be convenient to write $\mu=f_{\mu}\mu_{G}$, where $f_{\mu}$ is
the density of $\mu$ with respect to $\mu_{G}$. If $G$ is finite
this simply means $f_{\mu}(g)=\left|G\right|\mu(g)$. 

The measure $\mu$ induces a random walk on $G$ as follows. In the
first step, we choose a random element $h_{1}\in G$, distributed
according to $\mu$. In the second step, choose a random element $h_{2}\in G$,
distributed according to $\mu$, independently of step 1, and move
to $h_{1}\cdot h_{2}$. Continuing this way, choosing $h_{1},\dots,h_{t}$
independently at random, the probability to reach $g\in G$ after
$t$ steps is given by
\[
\mu^{*t}(g)=\mu*\dots*\mu(g):=\sum_{h_{1},...,h_{t}\in G\text{ s.t. }h_{1}\cdot...\cdot h_{t}=g}\mu(h_{1})\cdot...\cdot\mu(h_{t}),
\]
and moreover, 
\[
f_{\mu^{*t}}(g)=f_{\mu}*...*f_{\mu}(g)=\frac{1}{\left|G\right|^{t-1}}\sum_{h_{1},...,h_{t}\in G\text{ s.t. }h_{1}\cdot...\cdot h_{t}=g}f_{\mu}(h_{1})\cdot...\cdot f_{\mu}(h_{t}).
\]

Denote by $\Irr(G)$ the set of irreducible characters of $G$. Recall
that $\Irr(G)$ is an orthonormal basis for the space of conjugate
invariant functions $\C[G]^{G}$, with respect to the inner product
$\langle f_{1},f_{2}\rangle=\frac{1}{\left|G\right|}\sum_{g\in G}f_{1}(g)\overline{f_{2}}(g)$.
If $\mu$ is a conjugate invariant measure, we can write: 
\[
f_{\mu}(g)=\sum_{\rho\in\mathrm{Irr}(G)}a_{\mu,\rho}\rho(g),
\]
where $a_{\mu,\rho}:=\sum_{g\in G}\overline{\rho(g)}\mu(g)=\langle f_{\mu},\rho\rangle$
is the Fourier coefficient of $\mu$ at $\rho$. 
\begin{xca}
\label{exer:convolution of characters}For every $\rho_{1},\rho_{2}\in\Irr(G),$
we have $\rho_{1}*\rho_{2}=\frac{\delta_{\rho_{1},\rho_{2}}}{\rho_{1}(1)}\rho_{1}$. 
\end{xca}

By Exercise \ref{exer:convolution of characters}, we get: 
\[
f_{\mu}^{*t}(g)=\sum_{\rho\in\mathrm{Irr}(G)}\frac{a_{\mu,\rho}^{t}}{\rho(1)^{t-1}}\rho(g)=1+\sum_{1\neq\rho\in\mathrm{Irr}(G)}\frac{a_{\mu,\rho}^{t}}{\rho(1)^{t-1}}\rho(g).
\]

\begin{defn}
~\label{def:L^q norms}Let $1\leq q\leq\infty$. 
\begin{enumerate}
\item For every $f:G\rightarrow\C$, we set $\left\Vert f\right\Vert _{q}:=\left(\frac{1}{\left|G\right|}\sum_{g\in G}\left|f(g)\right|^{q}\right)^{\frac{1}{q}}$.
\item Given a signed measure $\mu$ on $G$, we set $\left\Vert \mu\right\Vert _{q}:=\left\Vert f_{\mu}\right\Vert _{q}$.
In particular, in this notation we have $\left\Vert \mu_{G}\right\Vert _{q}=1$
for every $1\leq q\leq\infty$. 
\end{enumerate}
\end{defn}

\begin{rem}
\label{rem:Jensen and Young's inequality}The following two inequalities
will be useful.
\begin{enumerate}
\item \textbf{Jensen's inequality}: for every $f:G\rightarrow\C$ and $1\leq q\leq q'\leq\infty$,
we have $\left\Vert f\right\Vert _{q}\leq\left\Vert f\right\Vert _{q'}$.
\item \textbf{Young's convolution inequality}: given $f,h:G\rightarrow\C$,
and given $1\leq q,s,r\leq\infty$ with $\frac{1}{q}+\frac{1}{s}=1+\frac{1}{r}$,
we have $\left\Vert f*h\right\Vert _{r}\leq\left\Vert f\right\Vert _{q}\left\Vert h\right\Vert _{s}$.
\end{enumerate}
\end{rem}

\begin{lem}
\label{lem:Mixing of random walks}Let $G$ be a finite group, and
let $\mu$ be a conjugate invariant probability measure. Suppose that
$1\in\mathrm{supp}\mu\nsubseteq N$ for every proper normal subgroup
$N\vartriangleleft G$. Then there exists $0<\alpha<1$ such that
for every $t\in\mathbb{N}$ and every $q\geq1$, 
\[
\left\Vert \mu^{*t}-\mu_{G}\right\Vert _{q}\leq\left\Vert \mu^{*t}-\mu_{G}\right\Vert _{\infty}\leq\left|G\right|\cdot\alpha^{t}.
\]
\end{lem}

\begin{proof}
First note it is enough to show that $\left|a_{\mu,\rho}\right|<\rho(1)$
for all $1\neq\rho\in\Irr(G)$. Indeed, we then take $\alpha:=\underset{1\neq\rho\in\Irr(G)}{\max}\frac{\left|a_{\mu,\rho}\right|}{\rho(1)}$,
so that
\[
\left\Vert \mu^{*t}-\mu_{G}\right\Vert _{\infty}=\left\Vert \sum_{1\neq\rho\in\Irr(G)}\frac{a_{\mu,\rho}^{t}}{\rho(1)^{t-1}}\rho(g)\right\Vert _{\infty}\leq\sum_{1\neq\rho\in\Irr(G)}\frac{\left|a_{\mu,\rho}^{t}\right|}{\rho(1)^{t}}\rho(1)^{2}\leq\alpha^{t}\sum_{1\neq\rho\in\Irr(G)}\rho(1)^{2}\leq\left|G\right|\cdot\alpha^{t}.
\]
Since $\mathrm{Supp}(\mu^{*t})\subseteq\mathrm{Supp}(\mu^{*t+1})$,
there exists $t_{0}\in\N$ such that $\mathrm{S:=Supp}(\mu^{*t_{0}})=\mathrm{Supp}(\mu^{*(t_{0}+1)})$,
so $S\cdot S=S$ and $S^{-1}\subseteq S^{\left|G\right|-1}=S$ and
$S$ is a normal subgroup, hence $S=G$. In particular,\textbf{ }there
exists $\delta>0$ such that $f_{\mu^{*t_{0}}}(g)>\delta$ for every
$g\in G$. Finally, 
\begin{align*}
\frac{\left|a_{\mu,\rho}^{t_{0}}\right|}{\rho(1)^{t_{0}-1}}=\left|a_{\mu^{*t_{0}},\rho}\right| & =\left|\frac{1}{\left|G\right|}\sum_{g\in G}\overline{\rho(g)}\cdot\delta+\frac{1}{\left|G\right|}\sum_{g\in G}\overline{\rho(g)}(f_{\mu^{*t_{0}}}(g)-\delta)\right|\\
 & \leq0+\rho(1)\frac{1}{\left|G\right|}\sum_{g\in G}(f_{\mu^{*t_{0}}}(g)-\delta)\leq\rho(1)(1-\delta).\qedhere
\end{align*}
\end{proof}
\begin{defn}
\label{def:mixing time}The minimal $t\in\mathbb{N}$ such that $\left\Vert \mu^{*t}-\mu_{G}\right\Vert _{q}<\frac{1}{2}$
is called the $L^{q}$\textit{-mixing time, }and denoted $t_{q}(\mu)$,
or $t_{q}$ if $\mu$ is clear from the context. 
\end{defn}

\begin{rem}
\label{rem:1/2 is arbitrary}~
\begin{enumerate}
\item The choice of $\frac{1}{2}$ in Definition \ref{def:mixing time}
is for definiteness; we could have taken any other numerical value
smaller than $1$. 
\item Lemma \ref{lem:Mixing of random walks} holds more generally for aperiodic
(not necessarily conjugate invariant) measures. This follows e.g.~from
the It\^{o}\textendash Kawada equidistribution theorem \cite{IK40}
(see also \cite[Theorem 4.6.3]{App14}).
\item It follows from Remark \ref{rem:Jensen and Young's inequality}(1)
that $t_{q}(\mu)\leq t_{q'}(\mu)$ if $1\leq q\leq q'$.
\end{enumerate}
\end{rem}

\begin{xca}
\textit{We have $\left\Vert \mu^{*t_{q}l}-\mu_{G}\right\Vert _{q}<2^{-l}$
for any $l\in\mathbb{N}$ }(see e.g.~\cite[Lemma 4.18]{LeP17}). 
\end{xca}

\begin{example}[Bayer\textendash Diaconis, \cite{BD92}]
Shuffling a deck of $52$ playing cards can be seen as applying a
random permutation of $52$ elements, i.e.~a probability measure
on $S_{52}$. The randomness comes from the non-deterministic nature
of shuffling performed by human beings. Repeating the same method
of shuffling several times can be seen as applying a random walk on
the symmetric group $S_{52}$, of the same type presented in $\mathsection$\ref{subsec:Random-walk-on finite groups}.
Bayer and Diaconis studied a common shuffling method called ``riffle
shuffle'', which is used for example in many Casinos, and showed
that it takes $7$ to $8$ riffle shuffles to mix a deck of $52$
cards. More precisely, if $\mu$ is a probability measure on $S_{52}$,
corresponding to a random riffle shuffle of $52$ cards (based on
the Gilbert\textendash Shannon\textendash Reeds model), then the $L^{1}$-mixing
time of $\mu$, according to Definition \ref{def:mixing time}, is
$8$ ( $\left\Vert \mu^{*7}-\mu_{S_{52}}\right\Vert _{1}\sim0.67$
and $\left\Vert \mu^{*8}-\mu_{S_{52}}\right\Vert _{1}\sim0.33$).
Here is a \href{https://www.youtube.com/watch?v=AxJubaijQbI}{Numberphile video}
about this theorem. 
\end{example}

\subsubsection{Probabilistic Waring problem: finite simple groups}
\begin{defn}
Let $w\in F_{r}$ be a word and $G$ be a group. We set $\tau_{w,G}:=(w_{G})_{*}(\mu_{G}^{r})$
to be the corresponding \emph{word measure.} Note that\emph{ 
\[
\tau_{w,G}(g)=\frac{\left|w^{-1}(g)\right|}{\left|G\right|^{r}}.
\]
} 
\end{defn}

\begin{xca}
Show that $\tau_{w_{1},G}*\tau_{w_{2},G}=\tau_{w_{1}*w_{2},G}$ and
that $\tau_{w,G}^{*t}(g)=\frac{\left|(w^{*t})^{-1}(g)\right|}{\left|G\right|^{rt}}$. 
\end{xca}

We are interested in the family of random walks $\left\{ \tau_{w,G}\right\} _{G\text{ f.s.g}}$
on the family of finite simple groups. We say that the family $\left\{ \tau_{w,G}\right\} _{G\text{ f.s.g}}$
has a \emph{uniform $L^{q}$-mixing time of }$t_{q}(w)$, if $t_{q}(w)$
is the minimal $t\in\N$ such that: 
\begin{equation}
\underset{\left|G\right|\rightarrow\infty\,\,G\text{ f.s.g}}{\lim}\left\Vert \tau_{w,G}^{*t}-\mu_{G}\right\Vert _{q}=0.\label{eq:uniform mixing time}
\end{equation}
Since the family of finite simple groups is infinite, it is a priori
not clear that $t_{q}(w)$ exists. A deep result of Larsen, Shalev
and Tiep (Theorem \ref{thm:LST2} below) shows that this is indeed
the case, in the strongest sense of $q=\infty$. Moreover, in the
case that $q=1$, it turns out that $t_{1}(w)\leq2$ for every non-trivial
word $w$. 
\begin{thm}[{\cite[Theorem 1]{LST19}}]
\label{thm:LST1}Let $1\neq w\in F_{r}$ be a word. Then the family
$\left\{ \tau_{w,G}\right\} _{G\text{ f.s.g}}$ has a uniform $L^{1}$-mixing
time of $t_{1}(w)\leq2$. 
\end{thm}

We do not prove Theorem \ref{thm:LST1} in these notes. However, the
case of bounded rank groups of Lie type follows from a geometric statement
(Theorem \ref{thm:-convolution of two word maps is generically absolutely irreducible})
that the convolution of any two non-trivial word maps has geometrically
irreducible generic fiber. A generalization of this theorem is given
in Section \ref{sec8}, where also the connection to $L^{1}$-mixing
time is discussed in details. 

We now state the $L^{\infty}$-result. 
\begin{thm}[\cite{LaS12,LST19}]
\label{thm:LST2}Let $1\neq w\in F_{r}$. 
\begin{enumerate}
\item There exists $\epsilon(w)>0$ such that for every finite simple group
$G$ with$\left|G\right|\gg_{w}1$, and every $g\in G$, one has $\tau_{w,G}(g)<\left|G\right|^{-\epsilon(w)}$
(\cite{LaS12}).
\item There exists $t_{\infty}(w)\in\N$ such that the family $\left\{ \tau_{w,G}\right\} _{G\text{ f.s.g}}$
has a uniform $L^{\infty}$-mixing time of $t_{\infty}(w)$ (\cite{LST19}). 
\end{enumerate}
The exponents $\epsilon(w)^{-1}$ and $t_{\infty}(w)$ are both bounded
from above by $C\cdot\ell(w)^{4}$, for a large absolute constant
$C$, where $\ell(w)$ denotes the length of $w$. 
\end{thm}

\begin{example}
\label{exa:bounds on epsilon of power word}The following example
shows that there is no uniform upper bound for $t_{\infty}(w)$ which
is independent of $w$. Let $w_{(\ell)}=x^{\ell}$ be the power word,
and let $G=\mathrm{SL}_{n}(\mathbb{F}_{p})$. For simplicity, choose
$n$ divisible by $\ell$. Choose a prime $p$ such that $\mathbb{F}_{p}$
contains a primitive $\ell$-th root of unity $\xi_{\ell}$ (this
happens if and only if $\ell|p-1$, and there are infinitely many
such primes for each $\ell$). Note that $(w_{(\ell)})_{G}^{-1}(e)$
contains the diagonal element $g$ consisting of $\ell$ blocks of
size $n/\ell$, each is a scalar matrix $\xi_{\ell}^{j}\cdot I_{n/\ell}$
for $j=0,...,\ell-1$. Since $(w_{\ell})_{G}^{-1}(e)$ is invariant
under conjugation, it contains the conjugacy class $g^{G}$ of $g$,
so: 
\begin{equation}
\left|(w_{(\ell)})_{G}^{-1}(e)\right|\geq\left|g^{G}\right|=\frac{\left|G\right|}{\left|C_{G}(g)\right|}\geq\frac{\left|\mathrm{SL}_{n}(\mathbb{F}_{p})\right|}{\left|\mathrm{SL}_{n}(\mathbb{F}_{p})\cap\mathrm{GL}_{n/\ell}(\mathbb{F}_{p})^{\ell}\right|}.\label{eq:bounds on fiber of power word}
\end{equation}
Arguing using the Lang-Weil estimates (see Theorem \ref{thm:Lang-Weil}
below) and since $n^{2}-1-(\frac{n^{2}}{\ell}-1)>(n^{2}-1)(1-\frac{1}{\ell})$,
the RHS of (\ref{eq:bounds on fiber of power word}) is larger than
$\left|\mathrm{SL}_{n}(\mathbb{F}_{p})\right|^{1-\frac{1}{\ell}}$
for $p\gg_{n,\ell}1$. \textbf{Exercise:} conclude that $\epsilon(w_{(\ell)})\leq\ell^{-1}$
and $t_{\infty}(w_{(\ell)})\geq\ell$ (note that $\mathrm{SL}_{n}(\mathbb{F}_{p})$
is quasi-simple and not simple).
\end{example}

In the next section we discuss some key examples and applications,
in particular the proof of the above theorems in the case when $w$
is the commutator word, and then explain the main ideas of the proofs
of the general case for finite groups of Lie type. The latter splits
into the high rank case which is discussed in $\mathsection$ \ref{subsec:Proof-of-Theorem 2}-\ref{subsec:Proof-of-probabilistic for high rank},
and the low rank case, discussed in Section \ref{sec4}. The two cases
require very different sets of ideas. 

\section{\label{sec3}Commutator word, representation growth, and proof of
the probabilistic results}

\subsection{Representation growth}
\begin{defn}
\label{def:rep zeta function}Let $G$ be a compact group, and $r_{n}(G):=\left|\left\{ \rho\in\mathrm{Irr}(G):\rho(1)=n\right\} \right|$.
The\emph{ representation zeta function }of $G$ is: 
\[
\zeta_{G}(s):=\sum_{n=1}^{\infty}r_{n}(G)n^{-s}=\sum_{\rho\in\mathrm{Irr}(G)}\rho(1)^{-s},\text{ for }s\in\C.
\]
The \emph{abscissa of convergence} of $\zeta_{G}(s)$ is $\alpha(G):=\inf\left\{ s\in\R_{>0}:\zeta_{G}(s)<\infty\right\} $. 
\end{defn}

\begin{thm}[Larsen\textendash Lubotzky, \cite{LL08}]
Let $G$ be a compact, connected, simple Lie group. Then: 
\[
\alpha(G)=\frac{\mathrm{rk}(G_{\C})}{\left|\Sigma^{+}(G_{\C})\right|},
\]
where $\Sigma^{+}(G_{\C})$ is the set of positive roots in the root
system corresponding to $G_{\C}$.
\end{thm}

\begin{example}
If $G=\mathrm{SU}_{2}$, then $r_{n}(G)=1$ for each $n\in\N$ (the
natural action on the space of homogeneous polynomials $f(x,y)$ of
degree $n-1$ with two variables is the unique $n$-dimensional irreducible
representation of $\SU_{2}$, up to isomorphism). Hence $\zeta_{G}(s):=\sum_{n=1}^{\infty}n^{-s}$
is the Riemann zeta function, and $\alpha(G)=1=\frac{1}{1}=\frac{\mathrm{rk}(G_{\C})}{\left|\Sigma^{+}(G_{\C})\right|}$. 
\end{example}

It is meaningless to discuss representation growth in a fixed finite
group. However, one can ask about behavior in families of finite groups: 
\begin{thm}[{Liebeck\textendash Shalev,\textbf{ }\cite[Theorem 1.1]{LiS05a}, \cite[Theorem 1.1]{LiS05b}}]
\label{thm:rep growth of Chevalley groups}~ 
\begin{enumerate}
\item For all $s>1$, we have $\underset{\left|G\right|\rightarrow\infty:\text{ }G\text{ f.s.g}}{\lim}\zeta_{G}(s)=1$. 
\item Let $\underline{G}$ be a Chevalley group scheme (a $\Z$-model of
a connected, simply connected, simple algebraic $\C$-group $\underline{G}_{\C}$).
Then: 
\[
\underset{q\rightarrow\infty}{\lim}\zeta_{\underline{G}(\mathbb{F}_{q})}(s)=1\text{ for every }s>\frac{\mathrm{rk}(\underline{G}_{\C})}{\left|\Sigma^{+}(\underline{G}_{\C})\right|}.
\]
\end{enumerate}
\end{thm}

\subsection{The commutator word and representation growth }

Given a word $w\in F_{r}$ and a finite group $G$, recall that $\tau_{w,G}:=(w_{G})_{*}(\mu_{G^{r}})=\frac{1}{\left|G\right|}\sum_{\rho\in\Irr(G)}a_{w,G,\rho}\cdot\rho$,
where $a_{w,G,\rho}:=a_{\tau_{w,G},\rho}$ is the Fourier coefficient
of $\tau_{w,G}$ at $\rho\in\Irr(G)$. 
\begin{xca}
Since $\tau_{w,G}:=(w_{G})_{*}(\mu_{G}^{r})$, we have: 
\[
a_{w,G,\overline{\rho}}=\sum_{g\in G}\rho(g)\cdot\tau_{w,G}(g)=\E_{(g_{1},...,g_{r})\in G^{r}}(\rho(w(g_{1},...,g_{r}))).
\]
\end{xca}

Let $w_{\mathrm{com}}:=[x,y]$. A classical theorem of Frobenius describes
$a_{w_{\mathrm{com}},G,\rho}$. 
\begin{thm}[Frobenius, 1896]
\label{thm:Frobenius}For every finite group $G$, 
\[
a_{w_{\mathrm{com}},G,\overline{\rho}}=\E_{(x,y)\in G^{2}}(\rho(xyx^{-1}y^{-1}))=\frac{1}{\rho(1)}.
\]
\end{thm}

\begin{proof}
Let $\rho\in\mathrm{Irr}(G)$ be the character of an irreducible representation
$\pi_{\rho}:G\longrightarrow\mathrm{GL}(V_{\rho})$. Consider 
\[
T_{y}:=\frac{1}{\left|G\right|}\sum_{x\in G}\pi_{\rho}(xyx^{-1})\in\mathrm{End}(V_{\rho})
\]
Note that for each $y,z\in G$: 
\[
\pi_{\rho}(z)\circ T_{y}=\frac{1}{\left|G\right|}\sum_{x\in G}\pi_{\rho}(zxyx^{-1})\underset{x\mapsto z^{-1}x}{=}\frac{1}{\left|G\right|}\sum_{x\in G}\pi_{\rho}(xyx^{-1}z)=T_{y}\circ\pi_{\rho}(z).
\]
Thus $T_{y}\in\mathrm{End}(V_{\rho})^{G}$, so by Schur's lemma, it
is a scalar matrix $c_{\rho}\cdot I_{\rho(1)}$, where 
\[
c_{\rho}\rho(1)=\tr\left(\frac{1}{\left|G\right|}\sum_{x\in G}\pi_{\rho}(xyx^{-1})\right)=\frac{1}{\left|G\right|}\sum_{x\in G}\rho(xyx^{-1})=\rho(y).
\]
Finally, 
\begin{align*}
\frac{1}{\left|G\right|^{2}}\sum_{x,y\in G}\rho(xyx^{-1}y^{-1}) & =\frac{1}{\left|G\right|}\sum_{y\in G}\left(\frac{1}{\left|G\right|}\sum_{x\in G}\rho(xyx^{-1}y^{-1})\right)=\frac{1}{\left|G\right|}\sum_{y\in G}\left(\tr\frac{1}{\left|G\right|}\sum_{x\in G}\pi_{\rho}(xyx^{-1}y^{-1})\right)\\
 & =\frac{1}{\left|G\right|}\sum_{y\in G}\tr\left(T_{y}\circ\pi_{\rho}(y^{-1})\right)=\frac{1}{\left|G\right|}\sum_{y\in G}\tr\left(\frac{\rho(y)}{\rho(1)}\cdot\pi_{\rho}(y^{-1})\right)\\
 & =\frac{1}{\rho(1)}\frac{1}{\left|G\right|}\sum_{y\in G}\left|\rho(y)\right|^{2}=\frac{1}{\rho(1)}.\qedhere
\end{align*}
\end{proof}
\begin{cor}
\label{cor:Frobenius cor}Let $G$ be a finite group. Then: 
\[
\frac{\left|(w_{\mathrm{com}}^{*t})^{-1}(g)\right|}{\left|G\right|^{2t}}=\tau_{w_{\mathrm{com}},G}^{*t}(g)=\frac{1}{\left|G\right|}\sum_{\rho\in\mathrm{Irr}(G)}\frac{\left(a_{w_{\mathrm{com}},G,\rho}\right)^{t}}{\rho(1)^{t-1}}\rho(g)=\frac{1}{\left|G\right|}\sum_{\rho\in\mathrm{Irr}(G)}\rho(1)^{1-2t}\rho(g).
\]

Hence, $\zeta_{G}(2t-2)=\left|G\right|\tau_{w_{\mathrm{com}},G}^{*t}(e)$. 
\end{cor}

\begin{thm}[Garion\textendash Shalev, \cite{GS09}]
\label{thm:mixing of comutator}Over the family of finite simple
groups $G$, with $\left|G\right|\gg_{w}1$:
\begin{enumerate}
\item $t_{1}(w_{\mathrm{com}})=t_{2}(w_{\mathrm{com}})=1$ (uniform $L^{1}/L^{2}$-mixing
time of $1$). 
\item $t_{\infty}(w_{\mathrm{com}})=2$. 
\end{enumerate}
\end{thm}

\begin{proof}
1) By Corollary \ref{cor:Frobenius cor} we have: 
\[
\tau_{w_{\mathrm{com}},G}(g)=\frac{1}{\left|G\right|}\sum_{\rho\in\mathrm{Irr}(G)}a_{w_{\mathrm{com}},G,\rho}\rho(g)=\frac{1}{\left|G\right|}\sum_{\rho\in\mathrm{Irr}(G)}\frac{\rho(g)}{\rho(1)}.
\]
Thus, by Jensen's inequality (Remark \ref{rem:Jensen and Young's inequality}(1))
and Theorem \ref{thm:Frobenius},
\[
\left\Vert \tau_{w_{\mathrm{com}},G}-\mu_{G}\right\Vert _{1}^{2}\leq\left\Vert \tau_{w_{\mathrm{com}},G}-\mu_{G}\right\Vert _{2}^{2}=\sum_{1\neq\rho\in\mathrm{Irr}(G)}a_{w_{\mathrm{com}},G,\rho}^{2}=\sum_{1\neq\rho\in\mathrm{Irr}(G)}\rho(1)^{-2}=\zeta_{G}(2)-1\underset{\left|G\right|\rightarrow\infty}{\rightarrow}0.
\]
2) Note that $\left|G\right|\tau_{w_{\mathrm{com}},G}(e)=\left|\Irr(G)\right|=\left|\{\text{conj classes of }G\}\right|$.
Hence $\left|G\right|\tau_{w_{\mathrm{com}},G}(e)-1\rightarrow\infty$
and thus $t_{\infty}(w_{\mathrm{com}})>1$. On the other hand, $t_{\infty}(w_{\mathrm{com}})=2$
since by Young's convolution inequality (Remark \ref{rem:Jensen and Young's inequality}(2)),
\[
\left\Vert \tau_{w_{\mathrm{com}},G}^{*2}-\mu_{G}\right\Vert _{\infty}=\left\Vert \left(\tau_{w_{\mathrm{com}},G}-\mu_{G}\right)^{*2}\right\Vert _{\infty}\leq\left\Vert \tau_{w_{\mathrm{com}},G}-\mu_{G}\right\Vert _{2}^{2}=\zeta_{G}(2)-1\underset{\left|G\right|\rightarrow\infty}{\longrightarrow}0.
\]
\end{proof}

\subsection{\label{subsec:Proof-of-Theorem 2}Proof of Theorem \ref{thm:LST2}}

The key ingredient in the proof of Theorem \ref{thm:LST2}(2) is the
following: 
\begin{thm}[{Larsen\textendash Shalev\textendash Tiep, \cite{LST19}. See also
\cite[Theorem 1.6]{AG}}]
\label{thm:bounds on Fourier coefficients}Let $1\neq w\in F_{r}$.
Then there exists $\epsilon(w)>0$ such that for every finite simple
group $G$, with $\left|G\right|\gg_{w}1$, and every $\rho\in\Irr(G)$:
\[
\left|a_{w,G,\rho}\right|\leq\rho(1)^{1-\epsilon(w)}.
\]
\end{thm}

\begin{proof}[Theorem \ref{thm:bounds on Fourier coefficients} implies Theorem
\ref{thm:LST2}(2)]
Let $t>4\epsilon(w)^{-1}$. Then: 
\begin{align*}
\left\Vert \tau_{w,G}^{*t}-\mu_{G}\right\Vert _{\infty} & =\underset{g\in G}{\max}\left|f_{w^{*t},G}(g)-1\right|=\underset{g\in G}{\max}\left|\sum_{1\neq\rho\in\Irr(G)}\frac{a_{w,G,\rho}^{t}}{\rho(1)^{t-1}}\rho(g)\right|\\
 & \leq\sum_{1\neq\rho\in\Irr(G)}\rho(1)^{1-t\epsilon(w)}\cdot\rho(1)\leq\sum_{1\neq\rho\in\Irr(G)}\rho(1)^{-2}=\zeta_{G}(2)-1\rightarrow0,
\end{align*}
where the last equality follows from Theorem \ref{thm:rep growth of Chevalley groups}.
\end{proof}
We now sketch the main steps in the proof of Theorem \ref{thm:bounds on Fourier coefficients}\textbf{
for groups of Lie type}. We will not discuss the remaining case of
the family of alternating groups. There are two main slogans: 

\textbf{Slogan \#1}: if $g\in G$ has small centralizer $C_{G}(g)$,
then $\left|\rho(g)\right|$ is small.

\textbf{Slogan \#2}: $C_{G}(w(g_{1},...,g_{r}))$ is small with very
high probability. 

Slogan \#1 is an important phenomenon in representation theory, which
can already be seen from Schur's orthogonality. 
\begin{example}
For every $\rho\in\Irr(G)$, we have $\sum_{x\in G}\left|\rho(x)\right|^{2}=\left|G\right|$
by Schur's orthogonality. In particular, $\left|\rho(g)\right|^{2}\left|g^{G}\right|\leq\left|G\right|$,
so: 
\[
\left|\rho(g)\right|\leq\sqrt{\left|C_{G}(g)\right|}.
\]
This is called \emph{the centralizer bound}.
\end{example}

The character estimates required for mixing of word measures, or various
other conjugate invariant measures, are of exponential form $\left|\rho(g)\right|\leq\rho(1)^{1-\epsilon}$.
The current state of the art for (exponential) character estimates
in finite simple groups of Lie type is the recent result of Larsen
and Tiep: 
\begin{thm}[Larsen\textendash Tiep, \cite{LT24}]
\label{thm:character estimates}There exists an absolute constant
$c>0$, such that for every finite simple group of Lie type $G$,
every $1\neq\rho\in\mathrm{Irr}(G)$ and every $g\in G$: 
\begin{equation}
\left|\rho(g)\right|<\rho(1)^{1-c\frac{\log\left|g^{G}\right|}{\log\left|G\right|}}.\label{eq:best character estimates}
\end{equation}
\end{thm}

\begin{rem}
Theorem \ref{thm:LST2} was proved a few years before Theorem \ref{thm:character estimates},
using the characters estimates from \cite{GLT20,GLT23}, which are
less general than (\ref{eq:best character estimates}), but strong
enough for Theorem \ref{thm:LST2}. In \cite{GLT20,GLT23} a bound
of the form (\ref{eq:best character estimates}) was given for all
sufficiently regular elements, in the sense that $\left|C_{G}(g)\right|\leq\left|G\right|^{\delta}$
for some small $\delta>0$. One of the key ideas in \cite{LT24} to
deal with elements of large centralizers, is to reduce to the setting
of \cite{GLT20,GLT23} using a mixing argument showing that sufficiently
many self-convolutions of small conjugacy classes produces elements
with small centralizers with high probability. We will not discuss
the details of the proof of \cite{LT24} or \cite{GLT20,GLT23} in
these notes. 
\end{rem}

For the next theorem, denote 
\begin{equation}
T_{w,\delta,G}:=\left\{ (g_{1},...,g_{r})\in G^{r}:\left|C_{G}(w(g_{1},...,g_{r}))\right|>\left|G\right|^{\delta}\right\} .\label{eq:large centralizer}
\end{equation}

\begin{thm}[\cite{LST19}]
\label{thm:Larsen Shalev}Let $1\neq w\in F_{r}$ and let $\delta>0$.
Then there exists $c'(\delta)>0$ such that for every finite simple
group of Lie type $G$, with $\left|G\right|\gg_{w}1$, one has: 
\[
\P\left(T_{w,\delta,G}\right)<\left|G\right|^{-c'(\delta)/\ell(w)^{2}}.
\]
\end{thm}

\begin{proof}[Theorems \ref{thm:character estimates} and \ref{thm:Larsen Shalev}
imply Theorem \ref{thm:bounds on Fourier coefficients}]
Indeed, fixing (say) $\delta=\frac{1}{4}$, and by combining the
two theorems, we get that for every $1\neq\rho\in\Irr(G)$, and every
finite simple group $G$ with $\left|G\right|\gg_{w}1$, 
\begin{align*}
\left|a_{w,G,\overline{\rho}}\right| & =\left|\E(\rho(w(g_{1},...,g_{r})))\right|\leq\P\left(T_{w,\delta,G}\right)\cdot\rho(1)+\rho(1)^{1-c(1-\delta)}\leq\left|G\right|^{-c'(\delta)/\ell(w)^{2}}\cdot\rho(1)+\rho(1)^{1-c(1-\delta)}\\
 & \leq\rho(1)^{-2c'(\delta)/\ell(w)^{2}}\cdot\rho(1)+\rho(1)^{1-c(1-\delta)}\leq2\rho(1)^{1-2\epsilon(w)}<\rho(1)^{1-\epsilon(w)},
\end{align*}
where the third inequality follows since $\rho(1)<\left|G\right|^{\frac{1}{2}}$,
and the last inequality follows from the fact that $\underset{\left|G\right|\rightarrow\infty}{\lim}\underset{1\neq\rho\in\Irr(G)}{\min}\rho(1)=\infty$
(see \cite{LaSe74}). 
\end{proof}
\begin{rem}
\label{rem:Note-that-Theorem}Note that Theorem \ref{thm:Larsen Shalev},
with $\delta=0.99$ (for example) also implies Item (1) of Theorem
\ref{thm:LST2} for finite groups of Lie type. Indeed, Theorem \ref{thm:Larsen Shalev}
implies that $\tau_{w,G}(g)<\left|G\right|^{-\epsilon'(w)}$ for every
$g\in G$ with $\left|C_{G}(g)\right|>\left|G\right|^{\delta}$. To
bound $\tau_{w,G}(g)$ for the other elements $g$, we can simply
use the bound $\tau_{w,G}(g)=\frac{\tau_{w,G}(g^{G})}{\left|g^{G}\right|}\leq\left|g^{G}\right|^{-1}.$ 
\end{rem}

\subsection{\label{subsec:Proof-of-probabilistic for high rank}Proof of probabilistic
result (Theorem \ref{thm:Larsen Shalev}) in high rank}

The proof of Theorem \ref{thm:Larsen Shalev} in its full generality
is rather complicated. In order to explain it we will make several
simplifying assumptions. We consider the case that $G=\mathrm{GL}_{n}(\mathbb{F}_{p})$,
that $w$ has $2$ letters, and $n>32\ell(w)$, and prove the following
Theorem \ref{Baby case}, which is a slightly weaker version of Theorem
\ref{thm:Larsen Shalev}. Still, its proof (which is based on Larsen\textendash Shalev's
proof of Proposition 3.3 in \cite{LaS12}) contains the main ideas
used in the proof of the stronger Theorem \ref{thm:Larsen Shalev}.
In $\mathsection$\ref{sec4} we will be able to prove the low rank
case as well (i.e.~$n\leq32\ell(w)$). 
\begin{thm}
\label{Baby case}For every $1\neq w\in F_{2}$, every $n\geq32\ell(w)$
and for every prime $p$: 
\[
\tau_{w,\mathrm{GL}_{n}(\mathbb{F}_{p})}(e)\leq\left|\mathrm{GL}_{n}(\mathbb{F}_{p})\right|^{-\frac{1}{11\ell(w)}}.
\]
\end{thm}

Given a reduced word $w(x,y)$, write $w=w_{\ell}\cdots w_{1}$, where
each $w_{i}$ is of the form $x,x^{-1},y$ or $y^{-1}$. We write
$w^{(j)}:=w_{j}\cdots w_{2}w_{1}$, so $w^{(\ell)}=w$. Given $v_{1,0},...,v_{m,0}\in\mathbb{F}_{p}^{n}$,
we denote $v_{i,j}:=w^{(j)}.v_{i,0}$, for $j=1,\dots,\ell$. Each
of the sequences $\{v_{i,0},v_{i,1},...,v_{i,\ell}\}$ is what called
a \emph{trajectory} for $w$.

For example, if $w=w_{\mathrm{com}}=xyx^{-1}y^{-1}$, then 
\[
v_{i,1}=y^{-1}.v_{i,0}\text{ and }v_{i,2}=x^{-1}.v_{i,1}\text{ and }v_{i,3}=y.v_{i,2}\text{ and }v_{i,4}=x.v_{i,3}.
\]

We are going to conduct an ``experiment'' consisting of $m\sim\frac{n}{2\ell}$
small trials. We first choose two elements $x,y$ in $\mathrm{GL}_{n}(\mathbb{F}_{p})$
independently at random. Then we do the following: 
\begin{itemize}
\item In the first trial, we choose a random vector $v_{1,0}\in\mathbb{F}_{p}^{n}$,
and then build its trajectory $\{v_{1,0},v_{1,1},...,v_{1,\ell}\}$
according to $w(x,y)$. At this point we check if $v_{1,\ell}=v_{1,0}$.
If this is not the case, we stop the experiment. If it is the case,
we continue to the next trial. 
\item In the second trial, we choose a random vector $v_{2,0}\in\mathbb{F}_{p}^{n}$
which does not belong to the span of all previous occurrences of the
$v_{i',j'}$'s, i.e $v_{2,0}\notin\mathrm{span}\left\{ v_{1,0},v_{1,1},...,v_{1,\ell-1}\right\} $.
We again build its trajectory $\{v_{2,0},v_{2,1},...,v_{2,\ell}\}$
according to $w(x,y)$ and check if $v_{2,\ell}=v_{2,0}$. Again,
if this is not the case, we stop the experiment, and if it is true,
we continue to the next trial. 
\item We repeat the process $m\leq\frac{n}{2\ell}$ times. 
\end{itemize}
Intuitively, it seems extremely unlikely that we will be able to complete
the full experiment, and that a ``miracle'' is needed in order for
this to happen. Let us define the ``miracle set'' $S_{w,m}$, which
describes all tuples $(x,y,v_{1,0},\ldots,v_{m,0})$ for which the
experiment was successful. Let $\prec$ denote the lexicographic order
on pairs in $\left\{ 1,\dots,m\right\} \times\left\{ 0,\dots,\ell\right\} $,
that is, $(a_{1},a_{2})\prec(b_{1},b_{2})$ if and only if $a_{1}<b_{1}$,
or, $a_{1}=b_{1}$ and $a_{2}<b_{2}$. Denote 
\[
Z_{i,j}:=\Span\left\{ v_{i',j'}\mid(i',j')\prec(i,j)\right\} ,
\]
and
\[
S_{w,m}:=\left\{ (x,y,v_{1,0},\ldots,v_{m,0})\in\mathrm{GL}_{n}(\mathbb{F}_{p})^{2}\times\mathbb{F}_{p}^{nm}:\forall i,v_{i,0}\notin Z_{i,0},v_{i,\ell}=v_{i,0}\right\} .
\]
The next proposition shows that $S_{w,m}$ is indeed a rare event
(a ``miracle''). 
\begin{prop}
\label{prop:main proposition}If $n\geq2m\ell$, then $\left|S_{w,m}\right|<\ell^{m}p^{2n^{2}+m^{2}\ell}$. 
\end{prop}

After we have shown that $S_{w,m}$ is a ``miracle'', we would like
to show that $S_{w,m}$ must be large if $\left|w_{\mathrm{GL}_{n}(\mathbb{F}_{p})}^{-1}(e)\right|$
is large. This will imply that $\left|w_{\mathrm{GL}_{n}(\mathbb{F}_{p})}^{-1}(e)\right|$
must be small. 
\begin{xca}
\label{exer:lower bound on points in GL_n(F_p)}For every prime $p$,
we have $\left|\mathrm{GL}_{n}(\mathbb{F}_{p})\right|>2^{-n}p^{n^{2}}$. 
\end{xca}

\begin{proof}[Proposition \ref{prop:main proposition} implies Theorem \ref{Baby case}]
Indeed, for each $(x,y)\in w_{\mathrm{GL}_{n}(\mathbb{F}_{p})}^{-1}(e)$
we always have $v_{i,\ell}=v_{i,0}$. Hence, there are at least $p^{n}(p^{n}-p^{\ell})...(p^{n}-p^{\ell(m-1)})\geq p^{nm}2^{-m}$
choices for $\{v_{i,0}\}_{i=1}^{m}$ such that $(x,y,v_{1},\ldots,v_{m})\in S_{w,m}$.
Hence, 
\[
\left|w_{\mathrm{GL}_{n}(\mathbb{F}_{p})}^{-1}(e)\right|p^{nm}2^{-m}\leq\left|S_{w,m}\right|<\ell^{m}p^{2n^{2}+m^{2}\ell}.
\]
Let $m=\left\lfloor n/2\ell\right\rfloor $. Since $n\geq32\ell$
we have, $\frac{n}{2\ell}-\frac{n}{32\ell}\leq\frac{n}{2\ell}-1\leq m\leq\frac{n}{2\ell}$.
In particular, 
\begin{align*}
\tau_{w,\mathrm{GL}_{n}(\mathbb{F}_{p})}(e) & =\frac{\left|w_{\mathrm{GL}_{n}(\mathbb{F}_{p})}^{-1}(e)\right|}{\left|\mathrm{GL}_{n}(\mathbb{F}_{p})\right|^{2}}\leq\frac{2^{2n+m}}{p^{2n^{2}}}\cdot\ell^{m}p^{2n^{2}+m^{2}\ell-mn}\leq2^{4n}p^{\frac{n^{2}}{4\ell}-\frac{n^{2}}{2\ell}+\frac{n^{2}}{32\ell}}.\\
 & \leq2^{\frac{n^{2}}{8\ell}}p^{\frac{n^{2}}{4\ell}-\frac{n^{2}}{2\ell}+\frac{n^{2}}{32\ell}}\leq p^{\frac{n^{2}}{4\ell}-\frac{n^{2}}{2\ell}+\frac{n^{2}}{8\ell}+\frac{n^{2}}{32\ell}}\leq\left|\mathrm{GL}_{n}(\mathbb{F}_{p})\right|^{-\frac{1}{11\ell}}.\qedhere
\end{align*}
\end{proof}

\subsubsection{\label{subsec:Proof-of-Proposition}Proof of Proposition \ref{prop:main proposition}}

We first analyze a special case. 

\textbf{Special (model) case}: Let us bound the number the tuples
$(x,y,v_{1,0},\ldots,v_{m,0})$ in $S_{w,m}$ in which all $\{v_{i,j}\}_{i=1,...,m,j=0,...,3}$
are linearly independent. For simplicity of presentation, consider
the commutator word $w=xyx^{-1}y^{-1}$ (the argument generalizes
to any word).

We collect the information on $x,y,v_{i,j}$: 
\begin{align*}
x(v_{i,2}) & =v_{i,1}\text{ and }x(v_{i,3})=v_{i,4}\\
y(v_{i,1}) & =v_{i,0}\text{ and }y(v_{i,2})=v_{i,3}\tag{\ensuremath{\star}}
\end{align*}
Denote $V_{x}=\mathrm{span}\left\{ v_{i,2},v_{i,3}\right\} _{i=1}^{m}$,
$W_{x}=\mathrm{span}\left\{ v_{i,1},v_{i,0}\right\} _{i=1}^{m}$,
$V_{y}=\mathrm{span}\left\{ v_{i,1},v_{i,2}\right\} _{i=1}^{m}$,
$W_{y}=\mathrm{span}\left\{ v_{i,0},v_{i,3}\right\} _{i=1}^{m}$.
Note that each of these subspaces is $2m$-dimensional. For every
choice of linearly independent vectors $\{v_{i,j}\}_{i=1,...,m,j=0,...,3}$,
the following holds:
\[
x|_{V_{x}}=T_{x}\text{ and }y|_{V_{y}}=T_{y},\tag{\ensuremath{\star\star}}
\]
for $T_{x}:V_{x}\rightarrow W_{x}$ and $T_{y}:V_{y}\rightarrow W_{y}$
determined by $(\star)$. 
\begin{xca}
The number of choices for $(x,y)\in\mathrm{GL}_{n}(\mathbb{F}_{p})^{2}$
satisfying $(\star\star)$ is $p^{4m(n-2m)}\left|\mathrm{GL}_{n-2m}(\mathbb{F}_{p})\right|^{2}<p^{2n(n-2m)}$. 
\end{xca}

Since there are at most $p^{4nm}$ options for $\{v_{i,j}\}_{i=1,...,m}$,
we get a total contribution of 
\[
p^{2n(n-2m)+4nm}\leq p^{2n^{2}}.
\]
\textbf{General case}:\textbf{ }This is difficult to analyze directly.
Instead, we relax the conditions in $S_{w,m}$, so that we count a
larger set, but whose computation is similar to the special case above.

For each $i$, let $b_{i}$ be the first index $j\geq1$ such that
$v_{i,j}$ is a linear combination of previous vectors $v_{i',j'}$
with $(i',j')\prec(i,j),j'\leq b_{i'}$. We denote 
\[
R_{i,j}:=\Span\left\{ v_{i',j'}\mid(i',j')\prec(i,j),j'\leq b_{i'}\right\} ,
\]
and 
\[
S_{w,\{b_{i}\}_{i=1}^{m}}:=\left\{ (x,y,v_{1,0},\ldots,v_{m,0})\in G^{2}\times\mathbb{F}_{p}^{nm}:\forall i,v_{i,0}\notin Z_{i,0},v_{i,b_{i}}\in R_{i,b_{i}}\text{ and }v_{i,j'}\notin R_{i,j'}\forall j'<b_{i}\right\} .
\]
Note that $S_{w,m}\subseteq\bigcup_{\{b_{i}\}_{i=1}^{m}}S_{w,\{b_{i}\}_{i=1}^{m}}$.
Hence, instead of bounding $S_{w,m}$ directly, we bound each individual
$S_{w,\{b_{i}\}_{i=1}^{m}}$ and sum over all possible $\{b_{i}\}_{i=1}^{m}$.
The advantage in working with $S_{w,\{b_{i}\}_{i=1}^{m}}$ is that
we can provide estimates for $\left|S_{w,\{b_{i}\}_{i=1}^{m}}\right|$
which are very similar to the special case above. This is because
we artificially stop each trial in the experiment at the moment any
linear dependency occurs.

Hence, by conditioning on $\left\{ v_{i,j}\right\} _{i\in[m],j<b_{i}}$,
and by collecting all information on $x,y$ as was done in $(\star)$
and $(\star\star)$, we obtain the following upper bound: 
\begin{equation}
\left|S_{w,m}\right|\leq\sum_{\{b_{i}\}_{i=1}^{m}}\left|S_{w,\{b_{i}\}_{i=1}^{m}}\right|<\sum_{\{b_{i}\}_{i=1}^{m}}p^{m^{2}\ell}p^{2n^{2}}\leq\ell^{m}p^{2n^{2}+m^{2}\ell},\label{eq:completing the proof}
\end{equation}

where: 
\begin{itemize}
\item $\ell^{m}$ is an upper bound on the number of choices of $\{b_{i}\}_{i=1}^{m}$. 
\item $p^{m^{2}\ell}$ is an upper bound for the number of ways to write
$v_{i,b_{i}}=\sum_{(i',j')\prec(i,b_{i}):j'<b_{i'}}a_{i,(i',j')}v_{i',j'},$
for $a_{i,(i',j')}\in\mathbb{F}_{p}$ and for each $i$. 
\end{itemize}
\begin{xca}
Using a similar computation as in the Special Case above, prove (\ref{eq:completing the proof})
to complete the proof of Proposition \ref{prop:main proposition},
by showing that for each $\left\{ a_{i,(i',j')}\right\} _{i,i',j'}$,
the number of tuples $(x,y,v_{1,0},\ldots,v_{m,0})\in S_{w,\{b_{i}\}_{i=1}^{m}}$
satisfying the condition $v_{i,b_{i}}=\sum_{(i',j')\prec(i,b_{i}):j'<b_{i'}}a_{i,(i',j')}v_{i',j'}$
is bounded by $p^{2n^{2}}$.
\end{xca}

\begin{rem}
To go from Theorem \ref{Baby case} to the full generality of Theorem
\ref{thm:Larsen Shalev}, one has to deal with a few more difficulties: 
\begin{enumerate}
\item Estimating $\tau_{w,\mathrm{GL}_{n}(\mathbb{F}_{p})}(g)$ for an arbitrary
element $g$ with $\left|C_{\mathrm{GL}_{n}(\mathbb{F}_{p})}(g)\right|>\left|\mathrm{GL}_{n}(\mathbb{F}_{p})\right|^{\delta}$. 
\item Taking $w$ to be a word with $r$ letters instead of $2$ letters. 
\item Taking $G$ to be any classical group of Lie type of rank $>C\ell(w)$,
instead of $G=\mathrm{GL}_{n}(\mathbb{F}_{p})$. 
\end{enumerate}
Items (2) and (3) are mainly technical, and do not impose essential
difficulties. To deal with Item (1), Larsen, Shalev and Tiep showed
that if $g\in\mathrm{GL}_{n}(\mathbb{F}_{p})$ satisfies $\left|C_{\mathrm{GL}_{n}(\mathbb{F}_{p})}(g)\right|>p^{n^{2}\delta}$,
then there exists a non-constant polynomial $Q(X)\in\mathbb{F}_{p}[X]$
such that $\dim\ker Q(g)>\frac{1}{2}\delta n\deg Q.$ One can then
show that the condition that $\dim\ker Q(w(g_{1},...,g_{r}))>\frac{1}{2}\delta n\deg Q$
is a rare event, by again analyzing $\Theta(n)$ trajectories of the
word $w^{\deg Q}$. This reduces to a setting which is very similar
to the one analyzed in the proof of Theorem \ref{Baby case}. 
\end{rem}

\section{\label{sec4}Geometry of word maps on simple algebraic groups and
interaction with probability.}

In this section we make a connection between the geometry of word
maps $w_{\underline{G}}:\underline{G}^{r}\rightarrow\underline{G}$
on simple algebraic groups (irreducible components and dimension of
their fibers), and the probabilistic properties of the maps $w_{\underline{G}(\mathbb{F}_{p})}:\underline{G}(\mathbb{F}_{p})^{r}\rightarrow\underline{G}(\mathbb{F}_{p})$. 
\begin{rem}
This is a special case of a more general connection between the singularity
properties of $w_{\underline{G}}$ to probabilistic properties of
the maps $w_{\underline{G}(\Zp)}:\underline{G}(\Zp)^{r}\rightarrow\underline{G}(\Zp)$.
We will not discuss this here. However, this connection is studied
thoroughly in \cite{AA16,AA18,GHb,CGH23} and in \cite[Section 5]{AGL}.
\end{rem}

Our main tool will be the Lang\textendash Weil estimates.

\subsection{The Lang\textendash Weil estimates and a geometric interpretation
of $L^{\infty}$-mixing time}
\begin{thm}[The Lang\textendash Weil estimates, \cite{LW54}]
\label{thm:Lang-Weil}Let $X$ be a finite type $\mathbb{F}_{q}$-scheme
(e.g.~defined by $f_{1}=...=f_{r}=0$, for $f_{1},...,f_{r}\in\mathbb{F}_{q}[x_{1},...,x_{n}]$).
Then: 
\[
\left|\frac{\left|X(\mathbb{F}_{q})\right|}{q^{\dim X}}-C_{X}\right|\leq Cq^{-1/2},
\]
where $C_{X}$ is the number of top-dimensional irreducible components
of $X_{\overline{\mathbb{F}_{q}}}$, which are defined over $\mathbb{F}_{q}$,
and $C$ depends only on the complexity of $X$\footnote{For the precise notion of complexity, see e.g.~\cite[Definition 7.7]{GH19}.}
(i.e.~on $r,n$ and the degrees of $f_{i}$).
\end{thm}

\begin{rem}
\label{thm:effective LW}In fact, uniform bounds can be given for
the constant $C$ above, which are polynomial in the degree $d:=\underset{i=1,...,r}{\max}\deg f_{i}$.
Namely (see \cite[Theorem 7.5]{CM06}) one can take: 
\[
C=O_{n}(d^{O_{n}(1)}).
\]
 This will be essential in $\mathsection$\ref{sec7} when we exploit
uniform mixing bounds for word varieties whose complexity does not
remain bounded.
\end{rem}

\begin{example}
\label{exa:two examples}~ 
\begin{enumerate}
\item Consider $X=\mathrm{SL}_{2}$. Then $\mathrm{dim}X_{\mathbb{F}_{p}}=3$
and indeed 
\[
\left|\mathrm{SL}_{2}(\mathbb{F}_{p})\right|=p(p^{2}-1)=p^{3}(1-p^{-2}).
\]
\item Let $X=\{x^{2}=-1\}$. Let $p>2$. Then $\dim X_{\mathbb{F}_{p}}=0$
and $X_{\overline{\mathbb{F}_{p}}}$ has $2$ irreducible components
$x=\pm i$. If $p=1\mod4$, then $\left|X(\mathbb{F}_{p})\right|=C_{X_{\mathbb{F}_{p}}}=2$
and if $p\neq1\mod4$ then $\left|X(\mathbb{F}_{p})\right|=C_{X_{\mathbb{F}_{p}}}=0$. 
\end{enumerate}
\end{example}

\begin{defn}
Let $K$ be a field, and let $\varphi:X\to Y$ be a morphism of $K$-schemes.
For every field extension $K\leq K'$ and $y\in Y(K')$, denote: 
\begin{itemize}
\item $X_{y,\varphi}$ the $K'$-scheme $\left\{ x\in X:\varphi(x)=y\right\} $,
which is the fiber of $\varphi$ above $y$. 
\item $\varphi^{-1}(y):=\left\{ x\in X(K'):\varphi(x)=y\right\} $, the
$K'$ points of $X_{y,\varphi}$.
\end{itemize}
\end{defn}

\begin{defn}
\label{def:flatness}Let $K$ be a field. A morphism $\varphi:X\longrightarrow Y$
between smooth, irreducible $K$-varieties is called: 
\begin{enumerate}
\item \textit{Flat} if for every $x\in X(\overline{K})$ we have $\dim X_{\varphi(x),\varphi}=\dim X-\dim Y$. 
\item \emph{Flat with geometrically irreducible fibers, or (FGI),} if it
is flat, and $X_{\varphi(x),\varphi}$ is irreducible (as a $\overline{K}$-scheme)
for each $x\in X(\overline{K})$. 
\end{enumerate}
\end{defn}

\begin{rem}
\label{rem:Precise notion of flatness}Definition \ref{def:flatness}(1)
is not the standard definition of flatness, and it is a consequence
of the Miracle Flatness Theorem \cite[III, Exercise 10.9]{Har77}.
For the more general definition of flatness see e.g.~\cite[p.253-254]{Har77}. 
\end{rem}

\begin{example}
\label{two examples}Let $K$ be a field. 
\begin{itemize}
\item The map $\varphi:\mathbb{A}_{K}^{2}\rightarrow\mathbb{A}_{K}^{2}$
$(x,y)\longmapsto(x,xy)$ is not flat since the fiber over $(0,0)\in\overline{K}^{2}$
is one dimensional although the ``expected'' dimension is $0$. 
\item The map $\psi:\mathbb{A}_{K}^{2}\rightarrow\mathbb{A}_{K}$ $(x,y)\mapsto x^{2}+y^{2}$
is flat but not (FGI). 
\end{itemize}
\end{example}

The following proposition is a direct consequence of the above definitions,
the Lang\textendash Weil estimates (Theorem \ref{thm:Lang-Weil}),
and the fact that the complexity of the fibers of $\varphi$ is uniformly
bounded. 
\begin{prop}
\label{prop: flatness and counting points}Fix a prime $p$. Let $\varphi:X\to Y$
be a morphism of smooth, geometrically irreducible $\mathbb{F}_{p}$-varieties.
\begin{enumerate}
\item $\varphi:X\to Y$ is flat if and only if there exists $C>0$ such
that for every $q=p^{r}$ and every $y\in Y(\mathbb{F}_{q})$: 
\[
\left|\frac{|\varphi^{-1}(y)|}{q^{(\dim X-\dim Y)}}-C_{X_{y,\varphi}}\right|<Cq^{-\frac{1}{2}},
\]
or equivalently, by the Lang\textendash Weil estimates, 
\[
\left|\frac{|\varphi^{-1}(y)|}{\left|X(\mathbb{F}_{q})\right|}\left|Y(\mathbb{F}_{q})\right|-C_{X_{y,\varphi}}\right|<Cq^{-\frac{1}{2}}.
\]
\item $\varphi:X\to Y$ is (FGI) if and only if there exists $C>0$ such
that for every $q=p^{r}$:
\[
\left\Vert \varphi_{*}\mu_{X(\mathbb{F}_{q})}-\mu_{Y(\mathbb{F}_{q})}\right\Vert _{\infty}=\underset{y\in Y(\mathbb{F}_{q})}{\max}\left|\frac{|\varphi^{-1}(y)|}{\left|X(\mathbb{F}_{q})\right|}\left|Y(\mathbb{F}_{q})\right|-1\right|<Cq^{-\frac{1}{2}}.
\]
\end{enumerate}
Moreover, $C$ depends only on the complexity of the map $\varphi$. 
\end{prop}

\begin{example}
Back to Example \ref{two examples}, if $K=\mathbb{F}_{q}$:
\begin{enumerate}
\item For $(0,0)\in\mathbb{F}_{q}^{2}$, $\left|\varphi^{-1}(0,0)\right|=q$,
where the expected size is $q^{0}=1$. 
\item The map $\psi_{\mathbb{F}_{q}}$ is flat, so all non-empty fibers
are of dimension $1$. However, for $0\in\mathbb{F}_{q}$, the scheme
$X_{0,\psi}=\left\{ x^{2}+y^{2}=0\right\} $ has $2$ irreducible
components over $\overline{\mathbb{F}_{q}}$: $\left\{ x-iy=0\right\} $
and $\left\{ x+iy=0\right\} $. Also:
\begin{enumerate}
\item If $q=1\mod4$, then $C_{X_{0,\psi}}=2$ and $\left|\psi^{-1}(0)\right|=2q-1$. 
\item If $q=3\mod4$, then $C_{X_{0,\psi}}=0$ and $\left|\psi^{-1}(0)\right|=1$. 
\end{enumerate}
\end{enumerate}
In either case, the equivalent conditions of Item (2) of Proposition
\ref{prop: flatness and counting points} are not satisfied. 
\end{example}

\subsection{Completing the proof of Theorem \ref{thm:LST2} for groups of Lie
type}

We have seen so far a (sketch of) proof of Theorem \ref{thm:LST2}
(and Theorem \ref{thm:bounds on Fourier coefficients}) for groups
of Lie type in the case that $\mathrm{rk}(\underline{G})>C\ell(w)$,
where the bottleneck was the probabilistic argument (Theorem \ref{thm:Larsen Shalev}).
The low rank case is based on the following geometric statement, whose
proof we discuss in the next subsection. 
\begin{thm}
\label{thm:Low rank geometric statement}Let $w\in F_{r}$ and let
$\underline{H}$ be a simple, simply connected algebraic $\mathbb{F}_{q}$-group.
Then for every $t\geq\dim\underline{H}+1$, the map $(w^{*t})_{\underline{H}}:\underline{H}^{rt}\rightarrow\underline{H}$
is (FGI). 
\end{thm}

We can now complete the proof of Theorem \ref{thm:LST2} for groups
of Lie type. 
\begin{proof}[Sketch of proof of Theorem \ref{thm:LST2} for groups of Lie type]
For simplicity, we consider the case of untwisted Chevalley groups
(the other case can be handled similarly, with a few more complications).
So, suppose $G\simeq\underline{G}(\mathbb{F}_{q})/Z(\underline{G}(\mathbb{F}_{q}))$,
where $\underline{G}$ is a Chevalley group scheme. 
\begin{enumerate}
\item If $\mathrm{rk}(\underline{G})>C'\ell(w)$ for some $C'\gg1$, this
case follows from the argument at the end of $\mathsection$\ref{subsec:Proof-of-Theorem 2},
using Theorem \ref{thm:Larsen Shalev} which was proved in high rank
in $\mathsection$\ref{subsec:Proof-of-probabilistic for high rank},
and using Theorem \ref{thm:character estimates}. 
\item If $\mathrm{rk}(\underline{G})\leq C'\ell(w)$, we use Theorem \ref{thm:Low rank geometric statement}
and Proposition \ref{prop: flatness and counting points} to deduce
\begin{equation}
\underset{q\rightarrow\infty}{\lim}\left\Vert \tau_{w,\underline{G}(\mathbb{F}_{q})}^{*t}-\mu_{\underline{G}(\mathbb{F}_{q})}\right\Vert _{\infty}=0,\label{eq:L^infty mixing for low rank}
\end{equation}
 for $t\geq C''\ell(w)^{2}\geq\dim\underline{G}+1$.\textbf{ Exercise:}
(\ref{eq:L^infty mixing for low rank}) holds if we replace each $\underline{G}(\mathbb{F}_{q})$
with $G=\underline{G}(\mathbb{F}_{q})/Z(\underline{G}(\mathbb{F}_{q}))$. 
\end{enumerate}
Combining Items (1) and (2) proves Theorem \ref{thm:LST2}(2). 
\end{proof}
\begin{xca}
Show, using the Lang\textendash Weil estimates, that Theorem \ref{thm:Low rank geometric statement}
further implies Theorem \ref{thm:Larsen Shalev} for untwisted Chevalley
groups $G$ of low rank (i.e.~$\mathrm{rk}(\underline{G})\leq C'\ell(w)$).
\end{xca}

By Theorem \ref{thm:LST2} and Proposition \ref{prop: flatness and counting points},
we deduce: 
\begin{cor}[{\cite[Theorem 5]{LST19}}]
For every $1\neq w\in F_{r}$, there exist $\epsilon(w)>0$ and $t(w)\in\N$,
such that for every simple, simply connected algebraic $\mathbb{F}_{q}$-group
$\underline{G}$: 
\begin{enumerate}
\item The fibers $(\underline{G}^{r})_{g,w_{\underline{G}}}$ of $w_{\underline{G}}:\underline{G}^{r}\rightarrow\underline{G}$
are of codimension $\geq\epsilon(w)\dim\underline{G}$, for every
$g\in\underline{G}(\overline{\mathbb{F}_{q}})$. 
\item For every $t>t(w)$, the map $w_{\underline{G}}^{*t}:\underline{G}^{rt}\rightarrow\underline{G}$
is (FGI).
\end{enumerate}
\end{cor}

\subsection{\label{subsec:Proof-of-low rank}Proof of the low rank geometric
statement (Theorem \ref{thm:Low rank geometric statement})}

We use the notion of convolution in algebraic geometry, defined by
the second author and Hendel in \cite{GH21}. Let $K$ be a field. 
\begin{defn}[{\cite[Definition 1.1]{GH21}}]
\label{Convolution in algebraic geometry}Let $\varphi:X\rightarrow\underline{G}$,
$\psi:Y\rightarrow\underline{G}$ be maps from $K$-varieties $X,Y$
to an algebraic $K$-group $\underline{G}$. Define the \textit{convolution}
$\varphi*\psi:X\times Y\rightarrow\underline{G}$ by 
\[
\varphi*\psi(x,y)=\varphi(x)\cdot\psi(y).
\]
We denote by $\varphi^{*t}:X^{t}\rightarrow\underline{G}$ the $t$-th
convolution power. 

As is well known in classical analysis, convolution of functions improves
regularity. For example, if one convolves a $k$-differentiable function
$f:\R^{n}\rightarrow\R$, with an $l$-differentiable function $h:\R^{n}\rightarrow\R$,
the result $f*h$ is $k+l$-differentiable. Another example is Young's
convolution inequality (as in Remark \ref{rem:Jensen and Young's inequality}). 

A similar phenomenon takes place here; the convolution operation in
Definition \ref{Convolution in algebraic geometry} improves singularity
properties of morphisms. We shall see manifestations of this phenomenon
below, and also in $\mathsection$\ref{sec8}.
\end{defn}

\begin{fact}[{cf.~\cite[Proposition 3.1]{GH21}}]
\label{fact:convolution improves singularities}Let $X,Y$ be $K$-varieties.
Let $\varphi:X\to\underline{G}$ be a flat morphism and $\psi:Y\to\underline{G}$
be any morphism. Then $\varphi*\psi:X\times Y\to\underline{G}$ is
flat. 
\end{fact}

\begin{proof}
Consider the following diagram: 
\[
\begin{array}{ccc}
X\times Y & \overset{\pi_{X}}{\longrightarrow} & X\\
\downarrow\alpha & \, & \downarrow\varphi\\
\underline{G}\times Y & \overset{\beta}{\longrightarrow} & \underline{G}\\
\downarrow\pi_{\underline{G}} & \, & \,\\
\underline{G}, & \, & \,
\end{array}
\]
where $\pi_{X}(x,y)=x$, $\pi_{\underline{G}}(g,y)=g$ are the natural
projections, $\alpha(x,y)=(\varphi\ast\psi(x,y),y)$ and $\beta(g,y)=g\cdot\psi(y)^{-1}$.
Note that the morphism $\alpha:X\times Y\rightarrow\underline{G}\times Y$
is a base change of $\varphi:X\to\underline{G}$, and that the morphism
$\pi_{\underline{G}}:\underline{G}\times Y\rightarrow\underline{G}$
is flat, since $Y\rightarrow\spec K$ is flat. Since $\varphi\ast\psi=\pi_{\underline{G}}\circ\alpha$
and since flatness is preserved under base change and compositions,
the proposition follows.
\end{proof}
\begin{thm}[{Glazer\textendash Hendel, \cite[Thm B]{GH21}}]
\label{thm:convolutions makes flat}Let $\varphi:X\rightarrow\underline{G}$
be a dominant $K$-morphism from a smooth, geometrically irreducible
$K$-variety $X$ to an algebraic $K$-group $\underline{G}$. Then
$\varphi^{*t}$ is flat for all $t\geq\dim\underline{G}$. 
\end{thm}

\begin{proof}
Let $Z_{1}\subseteq X$ be the non-flat locus of $\varphi$. Since
$\varphi$ is dominant, it is generically flat, hence $\mathrm{dim}Z_{1}\leq\dim X-1$.
By Fact \ref{fact:convolution improves singularities}, the non-flat
locus $Z_{t}$ of $\varphi^{*t}$ in $X^{t}$ must be contained in
$Z_{1}\times...\times Z_{1}$, and hence $\mathrm{dim}Z_{t}\leq t\dim X-t$.

Suppose toward contradiction that $\varphi^{*t}$ is not flat, so
that it has a fiber of dimension $>t\dim X-\dim\underline{G}$. Any
top-dimensional irreducible component of this fiber is contained in
$Z_{t}$, so we must have $t<\dim\underline{G}$, and a contradiction. 
\end{proof}
\begin{rem}
The smoothness hypothesis in Theorem \ref{thm:convolutions makes flat}
can be omitted, when one works with the more general definition of
flatness (see Remark \ref{rem:Precise notion of flatness}). The proof
remains the same. 
\end{rem}

\begin{thm}[{\cite[Lemma 2.4]{LST19}}]
\label{thm:-convolution of two word maps is generically absolutely irreducible}Let
$\underline{G}$ be a simply connected semisimple $K$-algebraic group
and let $w_{1}\in F_{r_{1}}$ and $w_{2}\in F_{r_{2}}$. Then there
exists a Zariski open subset $U\subseteq\underline{G}$ such that
the fiber $(\underline{G}^{r_{1}+r_{2}})_{g,(w_{1}*w_{2})_{\underline{G}}}$
is geometrically irreducible for every $g\in U(\overline{K})$ (\textbf{``geometrically
irreducible generic fiber}''). 
\end{thm}

\begin{example}[\textbf{Warning(!)}]
\textbf{\label{exa:Warning(!)}}The assumption that $\underline{G}$
is simply connected is crucial. The map $(w_{\mathrm{com}}^{*t})_{\mathrm{PGL}_{n}}$
does not have geometrically irreducible generic fibers, for every
$t\in\N$. 
\end{example}

In $\mathsection$\ref{sec8}, we discuss a vast generalization of
Theorem \ref{thm:-convolution of two word maps is generically absolutely irreducible}
(see Theorem \ref{thm:convolutions of morphisms has geometrically irreducible generic fiber}). 
\begin{lem}
\label{topological exercise}Let $f:Z\rightarrow Y$ be a continuous
map of topological spaces, such that $f$ is open and $Y$ is irreducible.
Suppose there exists a dense collection $U$ of points $y\in Y$ s.t.~$f^{-1}(y)$
is irreducible. Then $Z$ is irreducible. 
\end{lem}

\begin{proof}
If $Z$ is not irreducible, then $Z=Z_{1}\cup Z_{2}$, with $Z_{1},Z_{2}$
closed, proper subsets. For each $y\in U$, since $f^{-1}(y)$ is
irreducible, either $f^{-1}(y)\subseteq Z_{1}$ or $f^{-1}(y)\subseteq Z_{2}$.
Let $W_{i}:=\left\{ y\in Y:f^{-1}(y)\subseteq Z_{i}\right\} $. Then
$W_{i}=Y\backslash f(Z_{i}^{c})$ is closed, as $f$ is open. Moreover,
$W_{i}\neq Y$, since $Z_{i}\neq X$. Finally, $W_{1}\cup W_{2}\supseteq U$
and hence $W_{1}\cup W_{2}\supseteq\overline{U}=Y$, and we are done. 
\end{proof}
\begin{prop}[{\cite[Corollary 3.20]{GHb}}]
\label{prop:gen irred convoluted with flat}Let $X,Y$ be geometrically
irreducible $K$-varieties, let $\varphi:X\rightarrow\underline{G}$
be a flat morphism with geometrically irreducible generic fiber, and
let $\psi:Y\rightarrow\underline{G}$ be a dominant morphism. Then
$\varphi*\psi:X\times Y\rightarrow\underline{G}$ is (FGI). 
\end{prop}

\begin{proof}
Let $Z:=\left(X\times Y\right)_{g,\varphi*\psi}$ be a fiber of $\varphi*\psi$
over $g\in\underline{G}(\overline{K})$, and consider the map $f:Z(\overline{K})\rightarrow Y(\overline{K})$,
induced from the standard projection $X\times Y\rightarrow Y$. Then
$Y(\overline{K})$ is irreducible and there exists $U\subseteq Y$
such that $\forall y\in U(\overline{K})$, 
\[
f^{-1}(y)=\left\{ (x,y)\in Z(\overline{K}):\varphi*\psi(x,y)=g\right\} \simeq\left\{ x\in X(\overline{K}):\varphi(x)=g\psi(y)^{-1}\right\} ,
\]
is irreducible. By Lemma \ref{topological exercise}, $Z$ is irreducible.
By Fact \ref{fact:convolution improves singularities}, $\varphi*\psi$
is flat, hence it is (FGI). 
\end{proof}
We are now ready to prove Theorem \ref{thm:Low rank geometric statement}. 
\begin{proof}[Proof of Theorem \ref{thm:Low rank geometric statement}]
Let $w\in F_{r}$ and let $\underline{H}$ be a simple, simply connected
algebraic $\mathbb{F}_{q}$-group. Since $\dim\underline{H}>1$, by
Theorems \ref{thm:convolutions makes flat} and \ref{thm:-convolution of two word maps is generically absolutely irreducible},
the map $(w^{*t})_{\underline{H}}:\underline{H}^{rt}\rightarrow\underline{H}$
is flat, with geometrically irreducible generic fiber, for every $t\geq\dim\underline{H}$.
By Proposition \ref{prop:gen irred convoluted with flat} and Theorem
\ref{thm:Borel}, $(w^{*(t+1)})_{\underline{H}}:\underline{H}^{r(t+1)}\rightarrow\underline{H}$
is (FGI) for every $t\geq\dim\underline{H}$, as required.
\end{proof}

\section{\label{sec5}Representation varieties of random groups}

In Sections \ref{sec5} to \ref{sec7}, we change the setting, so
that now we fix a connected semisimple algebraic group $\underline{G}$,
and consider the word maps $w_{\underline{G}}:\underline{G}^{r}\rightarrow\underline{G}$,
varying over all words $w\in F_{r}$ of a given length $\ell$, as
$\ell\rightarrow\infty$. 

However, unlike Sections \ref{sec2} to \ref{sec4}, instead of considering
self-convolutions of word measures $\tau_{w,G}$, we consider self-convolutions
of measures of the form $\mu_{\underline{A}}=\frac{1}{2r}\sum_{i=1}^{r}\left(\delta_{A_{i}}+\delta_{A_{i}^{-1}}\right)$,
where $\underline{A}:=(A_{1},...,A_{r})$ is a generating tuple in
$\underline{G}(\mathbb{F}_{p})$. Applying $\ell$ self-convolutions
of $\mu$ amounts to applying a \textbf{random} word of length $\ell$
on the tuple $(A_{1},...,A_{r})$. Hence, good equidistribution of
$\mu_{\underline{A}}^{*\ell}$ on $\underline{G}(\mathbb{F}_{p})$
for all generating tuples $\underline{A}$, and for all primes, should
say something about the geometry of \textbf{random} words.

In this section we consider all algebraic varieties and algebraic
groups over the complex numbers. This is merely for simplicity of
presentation. In sections \ref{sec6} and \ref{sec7}, we discuss
the same objects over other fields as well. 

\subsection{Representation varieties and universal theories}
\begin{defn}
We consider a presentation for an arbitrary finitely presented group
as follows: 
\[
\Gamma_{\underline{w}}:=\left\langle x_{1},...,x_{r}:w_{1}(\underline{x})=...=w_{k}(\underline{x})=1\right\rangle 
\]
where $r$ is the number of generators and $\underline{w}=\left(w_{1},...,w_{k}\right)$,
with $w_{i}\in F_{r}$, is the tuple of relators. We denote by 
\[
X_{\underline{w},\underline{G}}=\mathrm{Hom}(\Gamma_{\underline{w}},\underline{G})=\left\{ (g_{1},...,g_{r})\in\underline{G}^{r}:w_{1}(\underline{g})=...=w_{k}(\underline{g})=1\right\} ,
\]
the \textbf{representation variety} of $\Gamma_{\underline{w}}$ in
$\underline{G}$. This is a closed subscheme of $\underline{G}^{r}$.

If $\underline{w}$ consists of a single word $w\in F_{r}$, then
we write $X_{w,\underline{G}}$. For example, if $w=1$, then $X_{w,\underline{G}}=\underline{G}^{r}$. 
\end{defn}

\begin{rem}
By Hilbert's Basis Theorem, for every $r$-generated group $\Gamma$
(even infinitely presented), there exist finitely many $w_{1},...,w_{k}\in F_{r}$,
such that 
\[
\mathrm{Hom}(\Gamma,\underline{G})\simeq\mathrm{Hom}(\Gamma_{\underline{w}},\underline{G})=X_{\underline{w},\underline{G}}.
\]
\end{rem}

Clearly, there is no uniqueness in the tuple $\underline{w}$ in the
above remark as any other tuple that generates (as a normal subgroup)
the same subgroup of $F_{r}$ as $\underline{w}$ would also work.
Yet the relations $\underline{w}=1$ in $\underline{G}$ can imply
other relations of the form $v=1$, where $v$ may not belong to the
normal subgroup generated by $\underline{w}$. For instance, it is
easily seen that the relation $ba^{3}b^{-1}=a^{2}$ in $\SL_{2}(\C)$
implies that the subgroup $\langle a,b\rangle\leq\SL_{2}(\C)$ is
metabelian. However the one-relator group $\langle a,b|ba^{3}b^{-1}=a^{2}\rangle$
is a non-solvable Baumslag-Solitar group. The main results of this
section will give ample further evidence of this phenomenon. However,
determining the exact set of relations entailed by a given one in
$\underline{G}$ is an intriguing problem that is completely open
as far as we know even for $\underline{G}=\SL_{2}$.
\begin{problem}
\label{prob:universal theory}Let $G=\mathrm{SL}_{2}(\C)$. Given
$\underline{w}\in F_{r}^{k}$, find all $\nu\in F_{r}$, such that
\[
\forall g_{1},\ldots,g_{r}\in G:\,\,\,\,\,\,\underline{w}(g_{1},...,g_{r})=1\Longrightarrow\nu(g_{1},...,g_{r})=1.
\]
\end{problem}

This problem can be paraphrased (and extended) in terms of first order
logic by asking to determine the \emph{universal theory of $\mathrm{SL}_{2}(\C)$}.
In logic, a universal sentence in the language of groups is a formula
with a single (universal) quantifier and no free variables. It is
easy to see that it must be (the conjunction of) formulas of the following
form, where $\underline{w}\in F_{r}^{k}$ and $\underline{\nu}\in F_{r}^{m}$:
\[
\forall g_{1},\ldots,g_{r}\in G:\,\,\,\,\,\,\bigwedge_{1}^{k}w_{i}(g_{1},\ldots,g_{r})=1\Longrightarrow\bigvee_{1}^{m}\nu_{i}(g_{1},\ldots,g_{r})=1,
\]
and the set of all universal sentences that are true in $G$ is called
the \emph{universal theory of $G$}. For example the universal theory
of the free group can be described using the Makanin\textendash Razborov
algorithm and groups with the same universal theory as the free group
have been given a geometric description as the so-called \emph{limit
groups} by Z. Sela \cite{sela}. Hence, Problem \ref{prob:universal theory}
can be rephrased as: 
\begin{problem}
Describe the universal theory of $\SL_{2}(\C)$. 
\end{problem}

The universal theory of $\SL_{2}(\C)$ is decidable, because of Hilbert's
Nullstellensatz and the decidability of ideal membership (elimination
theory, Groebner bases, etc.).

\subsection{Character varieties}

Note that $\underline{G}$ acts on $\underline{G}^{r}$ by conjugation
$g.(g_{1},...,g_{r})=(gg_{1}g^{-1},...,gg_{r}g^{-1})$, and that $X_{\underline{w},\underline{G}}$
is $\underline{G}$-invariant with respect to this action. Assume
from now on that $\underline{G}$ is reductive, we then define the
\textbf{character variety} as the categorical quotient 
\[
\mathcal{X}_{\underline{w},\underline{G}}:=X_{\underline{w},\underline{G}}/\!\!/\underline{G}
\]
in the sense of geometric invariant theory (GIT) (see e.g.~\cite[Section 4]{PV89}).
If $\C[X_{\underline{w},\underline{G}}]$ is the coordinate ring of
$X_{\underline{w},\underline{G}}$, then $\mathcal{X}_{\underline{w},\underline{G}}$
is the affine scheme whose coordinate ring\footnote{A priori it is not clear that $\C[X_{\underline{w},\underline{G}}]^{\underline{G}}$
is a finitely generated $\C$-algebra. However, this is true whenever
$\underline{G}$ is reductive (Hilbert, \cite{PV89}).} is $\C[X_{\underline{w},\underline{G}}]^{\underline{G}}$.
\begin{example}[{see e.g.~\cite{CS83,ABL18}, \cite[Section 4]{Bow98}}]
Let $r=2$ and $\underline{G}=\mathrm{SL}_{2}$. Suppose $w=1$.
Then $\Gamma_{w}=F_{2}$, $X_{w,\mathrm{SL}_{2}}=\mathrm{SL}_{2}\times\mathrm{SL}_{2}$
and 
\[
\mathcal{X}_{w,\mathrm{SL}_{2}}=\left(\mathrm{SL}_{2}\times\mathrm{SL}_{2}\right)/\!\!/\mathrm{SL}_{2}\simeq\mathbb{A}_{\C}^{3}.
\]
The isomorphism is given by $(A,B)\longmapsto(\tr(A),\tr(B),\tr(AB))$.
Indeed: 
\begin{itemize}
\item For every $(x,y,z)\in\C^{3}$, we can find $(A,B)\in\mathrm{SL}_{2}(\C)$
with $\tr(A)=x$, $\tr(B)=y$ and $\tr(AB)=z$. 
\item There is a Zariski open set $U\subseteq\mathbb{A}_{\C}^{3}$ (defined
by $\Delta\neq0$, see Example \ref{2sl2} below) where 
\[
(\tr(A),\tr(B),\tr(AB))=(\tr(A'),\tr(B'),\tr(A'B'))\Longleftrightarrow(A',B')\in(A,B)^{\mathrm{SL}_{2}}
\]
\end{itemize}
These are called the \emph{Fricke\textendash Klein} coordinates on
$\mathcal{X}_{1,\mathrm{SL}_{2}}$. 
\end{example}

\begin{example}
When $r=2$ and $\underline{G}=\SL_{3}$, then $X_{1,\mathrm{SL}_{3}}$
is a (branched) $2$-cover of $\mathbb{A}_{\C}^{8}$ via the map $(A,B)\longmapsto(\tr(A),\tr(B),\tr(AB),\tr(A^{-1}),\tr(B^{-1}),\tr(AB^{-1}),\tr(A^{-1}B),\tr(A^{-1}B^{-1}))$,
see \cite{lawton}. 
\end{example}

Note that in general, the map $\underline{G}(\C)^{r}\to\mathrm{Hom}(\C[X_{\underline{w},\underline{G}}]^{\underline{G}},\C)$
that sends $(g_{1},\ldots,g_{r})$ to the evaluation map at this point
descends to an onto map between $\underline{G}$-orbits closures of
$r$-tuples and (closed) points of $X_{\underline{w},\underline{G}}/\!\!/\underline{G}$.
Moreover, two orbit closures are mapped to the same point if and only
if they intersect. We refer to \cite{sikora} for general background
on character varieties.

Recall our notation: $\langle g_{1},\ldots,g_{r}\rangle$ stands for
the (abstract) subgroup generated by $g_{1},\ldots,g_{r}\in\underline{G}(\C)$
and $\overline{\langle g_{1},\ldots,g_{r}\rangle}^{Z}$ for the Zariski-closure
of this subgroup, which is a closed algebraic subgroup of $\underline{G}$.
The following fact is essential: 
\begin{fact}[{\cite[Theorem 2.50]{PRR23}}]
\label{fact:closedness of conjugacy orbits}Let $\underline{G}$
be a reductive group and suppose that $g_{1},...,g_{r}\in\underline{G}(\C)$
Zariski-generates a reductive subgroup $\underline{H}\leq\underline{G}$
(i.e. $\overline{\langle g_{1},...,g_{r}\rangle}^{Z}=\underline{H}$).
Then the $\underline{G}$-orbit $(g_{1},...,g_{r})^{\underline{G}}$
of $(g_{1},...,g_{r})$ is closed. 
\end{fact}

\begin{cor}
If $\overline{\langle g_{1},...,g_{r}\rangle}^{Z}=\underline{G}$
then $(g_{1},...,g_{r})^{\underline{G}}$ is closed. 
\end{cor}

\begin{example}
The element 
\[
u_{t}:=\left(\begin{array}{cc}
1 & t\\
0 & 1
\end{array}\right)\in\mathrm{SL}_{2}(\C)
\]
Zariski-generates (when $t\neq0$) the maximal unipotent subgroup
$\{u_{t},t\in\C\}$, and indeed, $d_{\lambda}u_{t}d_{\lambda}^{-1}=u_{\lambda^{2}t}$,
where $d_{\lambda}$ is the diagonal matrix $diag(\lambda,\lambda^{-1})$.
So the identity is contained in the closure of the $\SL_{2}(\C)$-orbit
of $u_{t}$, but not in the orbit itself. For a similar reason it
can be shown that the converse to Fact \ref{fact:closedness of conjugacy orbits}
holds (e.g. see \cite[Theorem 29]{sikora}). 
\end{example}

\subsection{Random groups}

Fix $r\geq2$. We can define a random group on $r$ generators, 
\[
\Gamma_{\underline{w}}:=\left\langle x_{1},...,x_{r}:w_{1}(\underline{x})=...=w_{k}(\underline{x})=1\right\rangle ,
\]
by choosing the relators $\underline{w}$ at random among words of
length $\ell$. There are several probabilistic models for random
words, among them: 
\begin{enumerate}
\item \textbf{Choosing reduced words of length $\ell$}: there are $2r$
options to choose the first letter in $\left\{ x_{i}^{\pm1}\right\} _{i=1}^{r}$
and then $2r-1$ options for choosing each of the next $\ell-1$ letters,
so a total of $2r\cdot(2r-1)^{\ell-1}$ possible reduced words of
length $\ell$. 
\item \textbf{Choosing reduced words of length in an interval} $[c_{1}\ell,c_{2}\ell]$
for constants $c_{1}<c_{2}$. 
\item \textbf{Choosing non-reduced words of length $\ell$}: there are $(2r)^{\ell}$
options. 
\end{enumerate}
It turns out that the analysis of Models (1) and (2) can be reduced
to that of Model (3), so we will focus on the latter, which is more
convenient as it allows to see random words as random walks on the
free group. All three models usually behave similarly and give the
same qualitative answers. Nevertheless, the Gromov density model is
usually phrased using reduced words as in Model (1). It is as follows: 
\begin{defn}[Gromov density model, \cite{Gro93}, \cite{Oll05}]
A \emph{random group on $r$ generators, at density $\delta\in[0,1]$
and length} $\ell$ is the group $\Gamma_{\underline{w}}$, where
$\underline{w}=(w_{1},...,w_{k})$, for $k=\left\lfloor (2r-1)^{\delta\ell}\right\rfloor $,
and where $w_{1},...,w_{k}$ are chosen independently uniformly at
random among all reduced words of length $\ell$. 
\end{defn}

We say that an event (or a sequence of events parameterized by $\ell$)
occurs \textbf{asymptotically almost surely} (or a.a.s.) if it occurs
in probability $\geq1-o(1)$ as $\ell\rightarrow\infty$. 
\begin{thm}
\label{thm:random group properties}Let $\delta\in[0,1]$ and $k=\left\lfloor (2r-1)^{\delta\ell}\right\rfloor $.
The following hold a.a.s.: 
\begin{enumerate}
\item \cite[9.B]{Gro93}; If $\delta\in[0,\frac{1}{12})$. Then $\Gamma_{\underline{w}}$
has small cancellation $C'(\frac{1}{6})$. This means that no two
relators (or their cyclic permutations) overlap on a segment of length
$\geq\frac{\ell}{6}$. 
\item \cite[9.B]{Gro93},\cite{Oll05}; If $\delta\in[0,\frac{1}{2})$,
then $\Gamma_{\underline{w}}$ is infinite and Gromov-hyperbolic. 
\item \cite{DGP11}; If $\delta>0$ then $\Gamma_{\underline{w}}$ does
not act on a tree without a global fixed point. 
\item \cite{Gro93}, \cite{Zuk03}, \cite{KK}; If $\delta>\frac{1}{3}$
then $\Gamma_{\underline{w}}$ has Kazhdan's property $T$. 
\item \cite{Gro93}; If $\delta>\frac{1}{2}$ then $\Gamma_{\underline{w}}$
is trivial if $\ell$ is odd and isomorphic to $\Z/2\Z$ if $\ell$
is even. 
\end{enumerate}
\end{thm}

A good general reference for the Gromov model is Ollivier's monograph
\cite{Oll05}. In addition to the properties above, random groups
at positive density do not have low dimensional representations: 
\begin{thm}[Kozma\textendash Lubotzky, \cite{KL19}]
\label{thm:Kozma-Lubotzky}If $\delta>0$ and $d\in\N$ is fixed,
then a.a.s., every $\varphi\in\mathrm{Hom}(\Gamma_{\underline{w}},\mathrm{GL}_{d}(\C))$
satisfies: 
\[
\begin{cases}
\varphi\text{ is trivial} & \text{if }\ell\text{ is odd}\\
\left|\varphi(\Gamma_{\underline{w}})\right|\leq2 & \text{if }\ell\text{ is even}.
\end{cases}
\]
\end{thm}

\begin{rem}
In contrast to Kozma\textendash Lubotzky's result, Ollivier\textendash Wise
\cite{OW11} and Agol \cite{Ago13} showed that every $C'(1/6)$ group
embeds in a right-angled Artin group (RAAG) and hence embeds in $\mathrm{GL}_{N}(\Z)$
for some $N$ (see also \cite[Theorem 5.9]{Ber15}). Hence for a random
group, the minimal such $N$ must grow with the length $\ell$ of
the relators.
\end{rem}

\subsubsection{Sketch of proof of Theorem \ref{thm:Kozma-Lubotzky}}

Note that if $w\neq1$, then $X_{w,\mathrm{GL}_{d}(\C)}\subseteq\mathrm{GL}_{d}(\C)^{r}$
is a proper subvariety.

The idea of the proof is to take a random relator $w_{1}$, and consider
$X_{w_{1},\mathrm{GL}_{d}(\C)}$, and then consider another random
relator $w_{2}$ and show that $X_{w_{1},\mathrm{GL}_{d}(\C)}\cap X_{w_{2},\mathrm{GL}_{d}(\C)}$
is of smaller dimension with high probability. Then we iterate until
we are only left with points $\underline{a}\in\mathrm{GL}_{d}(\C)^{r}$
with $\left|\left\langle a_{1},...,a_{r}\right\rangle \right|\leq2$.
The main tool to make this work is Bezout's theorem. 
\begin{defn}
\label{The-degree-of}The \textbf{degree} of an irreducible subvariety
$X\subseteq\mathbb{A}_{\C}^{n}$ of dimension $m$, denoted $\deg(X)$,
is the number of intersection points of $X$ with an affine subspace
$H$ of codimension $m$, in general position. If $X$ has irreducible
components $X_{1},...,X_{M}$, then we define 
\[
\deg(X):=\sum_{i=1}^{M}\deg(X_{i}).
\]

Note that in particular, the number of irreducible components of $X$
is at most its degree. 
\end{defn}

From this definition, we see immediately that if $X$ is a hypersurface
defined by the vanishing of a polynomial $f\in\C[x_{1}\ldots,x_{n}]$
of degree $d$, then $X$ has degree $d$.
\begin{thm}[Bezout's theorem]
\label{bez}Let $X,Y$ be closed subvarieties of $\mathbb{A}_{\C}^{n}$.
Then 
\[
\deg(X\cap Y)\leq\deg(X)\cdot\deg(Y).
\]
\end{thm}

For the proof, we refer the reader to \cite{fulton-intersection})
and \cite[Theorem II.3.2.2]{Danilov}, which deal mostly with projective
varieties and to \cite{schnorr} for the affine case. See also \cite{KL19}
where the following is deduced: 
\begin{xca}
If $X=\{f_{1}=...=f_{s}=0\}\subseteq\mathbb{A}_{\C}^{n}$, $f_{i}\in\C[x_{1},...,x_{n}]$,
and denote $d:=\underset{i}{\max}\deg(f_{i})$. Then $X$ has degree
at most $d^{\min(s,n)}$. 
\end{xca}

For a similar reason, given a word $w$ of length $\ell$, the degree
of the word variety $X_{w,\mathrm{GL}_{d}(\C)}$ is at most $\ell^{O_{d,r}(1)}$.

\textbf{A sketch of proof of Theorem \ref{thm:Kozma-Lubotzky}:} 
\begin{enumerate}
\item Let $U\subseteq\mathrm{GL}_{d}(\C)^{r}$ be the open subset $\mathrm{GL}_{d}(\C)^{r}\backslash\left\{ (a,...,a):a^{2}=I_{n}\right\} $.
Pick $w_{1}\in F_{r}$ of length $\ell$ at random, and decompose
$X_{w_{1},\mathrm{GL}_{d}(\C)}\cap U$ into its irreducible components
$X_{w_{1},\mathrm{GL}_{d}(\C)}\cap U=X_{w_{1}}^{(1)}\cup...\cup X_{w_{1}}^{(k_{1})}$. 
\begin{claim*}
For every $\underline{a}\in\mathrm{GL}_{d}(\C)^{r}$, if $\left|\left\langle a_{1},...,a_{r}\right\rangle \right|>2$,
then 
\begin{equation}
\Pr_{\left|w\right|=\ell}\left[w(\underline{a})=1\right]\leq1-\frac{1}{2r}.\label{eq:KL upper bound}
\end{equation}
\end{claim*}
\begin{proof}
Let $\mu:=\frac{1}{2r}\sum_{i=1}^{r}\left(\delta_{a_{i}}+\delta_{a_{i}^{-1}}\right)$.
Since we assumed $\underline{a}\in U$, we must have $\underset{x\in\mathrm{GL}_{d}(\C)}{\max}\mu(x)\leq1-\frac{1}{2r}$.
In particular, 
\[
\Pr_{\left|w\right|=\ell}\left[w(\underline{a})=1\mid\text{given first }\ell-1\text{ letters}\right]\leq1-\frac{1}{2r},
\]
which implies the claim. 
\end{proof}
\item For each $i=1,...,k_{1}$, we choose a tuple $\underline{a}^{(i)}$
belonging to $X_{w_{1}}^{(i)}$, and further choose $c\log(\ell)$
independent words $\underline{w_{2}}=w_{2}^{(1)},...,w_{2}^{(c\log(\ell))}$.
By Item (1), for each fixed $i$, the probability that $\underline{w_{2}}$
will not break the irreducible component of $X_{w_{1}}^{(i)}$ into
smaller components is bounded by: 
\[
\Pr\left[\underline{w_{2}}(\underline{a}^{(i)})=1\right]\leq(1-\frac{1}{2r})^{c\log(\ell)},
\]
and hence, 
\begin{align*}
 & \Pr\left[\dim X_{w_{1},\mathrm{GL}_{d}(\C)}\cap U>\dim X_{(w_{1},\underline{w_{2}}),\mathrm{GL}_{d}(\C)}\cap U\right]\geq\Pr\left[\forall i:\underline{w_{2}}(\underline{a}^{(i)})\neq1\right]\\
= & 1-\Pr\left[\exists i\in[k_{1}]:\underline{w_{2}}(\underline{a}^{(i)})=1\right]\underset{\text{union bound}}{\geq}1-k_{1}(1-\frac{1}{2r})^{c\log(\ell)}.
\end{align*}
By Bezout's theorem: $k_{1}\leq\ell^{O_{r,d}(1)}$, and hence by taking
$c$ large enough, 
\[
\Pr\left[\dim X_{w_{1},\mathrm{GL}_{d}(\C)}\cap U>\dim X_{(w_{1},\underline{w_{2}}),\mathrm{GL}_{d}(\C)}\cap U\right]\geq1-\frac{1}{\ell^{O_{r,d}(1)}}.
\]
We repeat this process $O_{r,d}(1)$ times, until we obtain an empty
set. 
\item We have proved that if $\Gamma_{\underline{w}}$ is a random group
with $\gg_{d,r}\log(\ell)$ random relators (and in particular, if
$\Gamma_{\underline{w}}$ is a random group at density $\delta>0$),
then a.a.s., the variety $X_{\underline{w},\mathrm{GL}_{d}(\C)}$
contains only elements $\underline{a}$ of the form $(a,...,a)$ with
$a^{2}=I$. I.e., $X_{\underline{w},\mathrm{GL}_{d}(\C)}$ consists
of homomorphisms $\varphi:\Gamma_{\underline{w}}\rightarrow\mathrm{GL}_{d}(\C)$
sending $(x_{1},...,x_{r})$ to $(a,...,a)$, so $|\varphi(\Gamma_{\underline{w}})|\leq2$. 
\end{enumerate}
\begin{rem}
Note that if we take only $O(\log(\ell))$ random relations, then
the error term in ``a.a.s.'' obtained in this proof is only $>1-1/\ell^{O(1)}$.
To get an exponential error term, the argument requires at least a
linear in $\ell$ number of independent relators (which is the case
if we assume $\delta>0$ of course). 
\end{rem}

\subsection{Principal part of character varieties}

We will be mostly interested in understanding the structure of the
so-called \emph{principal part} of the character variety (defined
below), which is a Zariski-open set accounting for `most' of the character
variety.

We first prove the following fact: 
\begin{fact}
\label{fact:principal part is open} Assume that $\underline{G}$
is semisimple. The variety

\[
(\underline{G}^{r})^{Z}:=\left\{ (g_{1},...,g_{r})\in\underline{G}^{r}:\overline{\left\langle g_{1},...,g_{r}\right\rangle }^{Z}=\underline{G}\right\} ,
\]
is an open subvariety of $\underline{G}^{r}$. 
\end{fact}

The assumption that $\underline{G}$ is semisimple is essential. Note
that it fails for a torus, since tuples made of finite order elements
form a dense subset in this case.

Fact \ref{fact:principal part is open} for $\underline{G}$ simple,
follows from the following lemma, by observing that the condition
that $\left\langle g_{1},...,g_{r}\right\rangle $ act irreducibly
on $\C^{d}$ is a Zariski open condition on $\underline{G}(\C)^{r}$.
(\textbf{exercise}: prove this).
\begin{lem}
Given a simple algebraic group $\underline{G}$, there exist finite
dimensional $\rho_{1},\rho_{2}\in\Irr(\underline{G}(\C))$ such that
for every $g_{1},...,g_{r}\in\underline{G}(\C)$: 
\[
\overline{\left\langle g_{1},...,g_{r}\right\rangle }^{Z}=\underline{G}(\C)\Longleftrightarrow\rho_{1},\rho_{2}|_{\langle g_{1},...,g_{r}\rangle}\text{ are both irreducible}.
\]
\end{lem}

\begin{proof}
We take $\rho_{1}=\mathrm{Ad}:\underline{G}(\C)\rightarrow\mathrm{GL}(\mathfrak{g})$,
for $\mathfrak{g}=\mathrm{Lie}(\underline{G}(\C))$ and take $\rho_{2}$
to be any irreducible representation of large enough dimension.

The direction $\Longrightarrow$ is easy, so it is left to prove $\Longleftarrow$.
Suppose that $\overline{\left\langle g_{1},...,g_{r}\right\rangle }^{Z}=\underline{H}(\C)\lneq\underline{G}(\C)$
for a proper algebraic subgroup $\underline{H}$. Then $\mathfrak{h}=\mathrm{Lie}(\underline{H}(\C))$
is a subrepresentation of $\mathfrak{g}$, which is contradiction
to the irreducibility of $\rho_{1}$, unless $\underline{H}(\C)$
is finite. To deal with the case that $\underline{H}(\C)$ is finite,
we apply Jordan's theorem: 
\begin{thm}[Jordan, 1878]
\label{thm:Jordan's theorem}There exists a function $f:\N\rightarrow\N$,
such that if $H$ is a finite subgroup of $\mathrm{GL}_{d}(\C)$,
then it must contain a normal, abelian subgroup $H_{0}\vartriangleleft H$
of index $\left|H:H_{0}\right|\leq f(d)$. In fact, Collins \cite{Col07}
showed that $f(d)$ can be taken to be $(d+1)!$ whenever $d\geq71$. 
\end{thm}

See \cite{breuillard-jordan} for a discussion of Jordan's theorem.
Note that $\underline{G}(\C)\leq\mathrm{GL}_{N}(\C)$ for some $N=N(\underline{G})$,
and write $\rho_{2}:\underline{G}(\C)\rightarrow\mathrm{GL}(V_{\rho_{2}})$.
Applying Jordan's theorem to $H:=\underline{H}(\C)$, we can find
$H_{0}\vartriangleleft H$ abelian of index $\left|H:H_{0}\right|\leq f(N)$.
Since $\rho_{2}(H)$ is finite it consists of semisimple elements,
so that the elements in $\rho_{2}(H_{0})$ are simultaneously diagonalizable.
Hence $V_{\rho_{2}}|_{\rho_{2}(H_{0})}$ is a direct sum of one dimensional
representations. Since $\left|H/H_{0}\right|\leq f(N)$, taking $\dim\rho_{2}>f(N)$
guarantees that $\rho_{2}$ cannot be irreducible. 
\end{proof}
\begin{xca}
Adapt the proof of Fact \ref{fact:principal part is open} to handle
the case that $\underline{G}$ is semisimple. 
\end{xca}

By Fact \ref{fact:principal part is open}, we deduce that 
\begin{align*}
X_{\underline{w},\underline{G}}^{Z} & :=\mathrm{Hom}(\Gamma_{\underline{w}},\underline{G})\cap(\underline{G}^{r})^{Z}\text{ is open in }X_{\underline{w},\underline{G}},\\
\mathcal{X}_{\underline{w},\underline{G}}^{Z} & :=\mathrm{Hom}^{Z}(\Gamma_{\underline{w}},\underline{G})/\!\!/\underline{G}\text{ is open in }\mathcal{X}_{\underline{w},\underline{G}}.
\end{align*}
Moreover, by Fact \ref{fact:closedness of conjugacy orbits}, distinct
$\underline{G}$-orbits in $X_{\underline{w},\underline{G}}^{Z}$
correspond to distinct points in $\mathcal{X}_{\underline{w},\underline{G}}^{Z}$. 
\begin{defn}
We call $\mathcal{X}_{\underline{w},\underline{G}}^{Z}$ the \textbf{principal
part} of $\mathcal{X}_{\underline{w},\underline{G}}$. 
\end{defn}

In the case of the free group on two generators and $\underline{G}=\SL_{2}$,
the principal part of the character variety is entirely described
as follows:
\begin{example}[exercise!]
\label{2sl2} Let $\underline{G}=\mathrm{SL}_{2}$, $r=2$. Recall
that $\mathcal{X}_{1,\underline{G}}(\C)=\C^{3}$ with the parametrization
$(a,b)\longmapsto(x,y,z):=(\tr(a),\tr(b),\tr(ab))$. The principal
part $\mathcal{X}_{1,\underline{G}}^{Z}(\C)$ of $\mathcal{X}_{1,\underline{G}}(\C)$
is $\mathcal{X}_{1,\underline{G}}^{Z}(\C)=\C^{3}\setminus\{F\cup\{\triangle=0\}\}$,
where 
\[
\triangle=x^{2}+y^{2}+z^{2}-4xyz-4=\tr(aba^{-1}b^{-1})-2,
\]
and where $F$ is a certain finite set $F\subseteq\left\{ (x,y,z);x,y,z\in\{0,\pm1,\pm\sqrt{2},\frac{1\pm\sqrt{5}}{2}\}\right\} $.
Moreover, 
\begin{align*}
\triangle & =0\Longleftrightarrow\langle a,b\rangle\text{ is not irreducible in }\C^{2}.\\
(x,y,z) & \in F\Longleftrightarrow\langle a,b\rangle\text{ is finite and irreducible on }\C^{2}.
\end{align*}
\end{example}

\begin{example}
\label{sl2ex} When $w$ is arbitrary, the trace $\tr(w(A,B))$ can
easily be computed in these coordinates. Using the Cayley\textendash Hamilton
theorem to express $A^{2}$ as a linear combination of $1$ and $A$
and similarly for $B^{2}$ and expanding one obtains: 
\[
\tr(w(A,B))=P_{w}(x,y,z)
\]
where $P_{w}\in\Z[x,y,z]$ is called the \emph{word polynomial} associated
to $w$ in $\SL_{2}$. Little is known in general about these polynomials,
see \cite{horowitz}. But they provide explicit equations (exercise!)
for the principal part of the character variety $\mathcal{X}_{w,\underline{\SL_{2}}}(\C)$,
namely: 
\begin{align}
\mathcal{X}_{w,\underline{\SL_{2}}}^{Z}(\C)=\left\{ (x,y,z)\in\C^{3}\setminus F,\Delta\neq0,P_{w}=2,P_{aw}=x,P_{bw}=y\right\} .\label{eqforX}
\end{align}
\end{example}


\section{\label{sec6}The main theorem and uniform gap results in finite simple
groups }

Kesten \cite{Kes59} gave a characterization of non-amenability of
a group in terms of the probability of return to the identity of the
simple random walk on one (or any) of its Cayley graphs: the group
is non-amenable if and only if the probability of return decays exponential
fast. In \cite{Bre11}, the first author proved a uniform version
of Kesten's theorem for linear groups, which tells us the following:
\begin{thm}[{{\cite[Corollary 1.6]{Bre11}}}]
\label{thm:improved bounds on prbability of hiting id}Let $\underline{G}$
be a semisimple algebraic group. There exists $c=c(r,\underline{G})>0$
such that for every $\ell\in\N$ and every\footnote{More generally, it is enough to demand that $\langle x_{1},...,x_{r}\rangle$
is not virtually solvable.} $\underline{x}=(x_{1},...,x_{r})\in(\underline{G}(\C)^{r})^{Z}$,
\[
\underset{w:\left|w\right|=\ell}{\Pr}\left[w(\underline{x})=1\right]<e^{-c\ell}.
\]
\end{thm}

We stress that the exponential rate $c>0$ is independent of the generating
set $\underline{x}$. The proof of uniformity relies of global estimates
for Diophantine heights over $\overline{\Q}$. This is to be contrasted
with the following other basic result of Kesten regarding the free
group:
\begin{thm}[Kesten, \cite{Kes59}]
If $\langle x_{1},...,x_{r}\rangle=F_{r}$, then 
\[
\underset{w:\left|w\right|=\ell}{\Pr}\left[w(\underline{x})=1\right]\approx\left(\frac{\sqrt{2r-1}}{r}\right)^{\ell}=e^{-c_{r}\ell},
\]
which is equal to the probability that a random (in the non-reduced
model) word of length $\ell$ is trivial. 
\end{thm}

Theorem \ref{thm:improved bounds on prbability of hiting id} and
the proof of Theorem \ref{thm:Kozma-Lubotzky} imply the following: 
\begin{cor}
\label{KLimproved} Let $\underline{G}$ be a semisimple algebraic
group, let $r,k\in\N$. There exists $c'=c'(r,k,\underline{G})>0$
such that if $\underline{w}=(w_{1},...,w_{k})$, for random words
$w_{i}\in F_{r}$ of length $\ell$, and if $k>(r-1)\dim\underline{G}$,
then $\mathcal{X}_{\underline{w},\underline{G}}^{Z}=\varnothing$
with probability $\ge1-e^{c'\ell}$. 
\end{cor}

In other words, it is enough to ask for $(r-1)\dim\underline{G}+1$
random generators to be able to conclude that the random group $\Gamma_{\underline{w}}$
has no group homomorphism with a Zariski-dense image in $\underline{G}$.
And this is achieved with an exponentially small probability of exception.
In a similar vein:

\begin{cor}
\label{vsol}If $d\in\N$ and $k>(r-1)d^{2}$, then a.a.s., every
$\varphi\in\mathrm{Hom}(\Gamma_{\underline{w}},\mathrm{GL}_{d}(\C))$
has virtually solvable image. 
\end{cor}

\begin{proof}[Sketch of proof of Corollary \ref{KLimproved}]
The proof is similar to the proof of Theorem \ref{thm:Kozma-Lubotzky},
where one just replaces the upper bound (\ref{eq:KL upper bound})
with the upper bound of Theorem \ref{thm:improved bounds on prbability of hiting id},
using the extra assumption that we only consider generating tuples
$\underline{g}\in(\underline{G}(\C)^{r})^{Z}$, respectively (for
Corollary \ref{vsol}) tuples that generate a non-virtually solvable
subgroup.
\end{proof}

\subsection{Statement of the main result}

Let $K$ be a field of characteristic $0$. We note that if $\underline{G}$
is defined over $K$ (e.g. if $K=\Q$),\textcolor{purple}{{} }then so
are the character variety $\mathcal{X}_{\underline{w},\underline{G}}$
and its principal part $\mathcal{X}_{\underline{w},\underline{G}}^{Z}$.
Consequently the Galois group $\mathrm{Gal}(\overline{K}|K)$ permutes
its geometric irreducible components. 

We are now ready to state the main theorem of \cite{BBV}: 
\begin{thm}[Becker\textendash Breuillard\textendash  Varj\'{u} '24+, \cite{BBV}]
\label{thm:main theorem random words}Let $\underline{G}$ be a semisimple
algebraic $K$-group, and let $k,r\in\N$. Let $w_{1},...,w_{k}\in F_{r}$
be random, independent, non-reduced words of length $\ell$, with
$\underline{w}=(w_{1},...,w_{k})$. Denote by $\delta:=r-k$ the \textbf{defect}
of $\Gamma_{\underline{w}}$. Then there exists $c=c(r,k,\underline{G})>0$
such that with probability $>1-e^{-c\ell}$: 
\begin{enumerate}
\item[(i)] If $\delta\geq1$, then $\mathcal{X}_{\underline{w},\underline{G}}^{Z}\neq\varnothing$. 
\item[(ii)] If $\delta\leq0$, then $\mathcal{X}_{\underline{w},\underline{G}}^{Z}=\varnothing$. 
\item[(iii)] If $\delta\geq1$, then $\dim\mathcal{X}_{\underline{w},\underline{G}}^{Z}=(\delta-1)\dim\underline{G}$. 
\item[(iv)] If $\delta=1$ and $\underline{G}$ is simply connected, then $\mathcal{X}_{\underline{w},\underline{G}}^{Z}$
is finite and a single Galois orbit. 
\item[(v)] If $\delta\geq2$ and $\underline{G}$ is simply connected, then
$\mathcal{X}_{\underline{w},\underline{G}}^{Z}$ is geometrically
irreducible. 
\end{enumerate}
\end{thm}

\begin{rem}
\label{some remarks}~
\begin{enumerate}
\item All items but (i) are under the assumption of the generalized Riemann
hypothesis (GRH). This assumption is used when applying an effective
version of Chebotarev's theorem \ref{thm:improved effective Chebotarev}
with error term. At the expense of weakening the bound on the probability
of exceptional words, this assumption can be relaxed by assuming only
that Dedekind zeta functions of number fields have no zeroes in a
certain small disc around $s=1$. 
\item Items (iv) and (v) imply that the Galois action on $\mathcal{X}_{\underline{w},\underline{G}}^{Z}$
is transitive, when $\underline{G}(\C)$ is simply-connected. 
\item Note that item (iii) says that $\mathcal{X}_{\underline{w},\underline{G}}^{Z}$
has the expected dimension (counting the number of degrees of freedom)
with probability $>1-e^{-c\ell}$. But this is \textbf{not} the case
for $\mathcal{X}_{\underline{w},\underline{G}}$ as the following
example shows.
\end{enumerate}
\end{rem}

\begin{example}
Let $\underline{G}=\mathrm{SL}_{2}$, and let $\underline{B}=\left\{ \left(\begin{array}{cc}
* & *\\
0 & *
\end{array}\right)\right\} $ be a Borel subgroup of $\underline{G}$. Since $\dim\underline{B}=2$,
by Item (iii) of Theorem \ref{thm:main theorem random words}, for
each $r<3k$, with probability $>1-e^{-c\ell}$, a random tuple $\underline{w}=(w_{1},...,w_{k})$
of words of length $\ell$ satisfies: 
\[
\dim\mathcal{X}_{\underline{w},\underline{G}}^{Z}=3(r-1-k)<2r-3=\dim\underline{B}^{r}-3.
\]
On the other hand, if $w_{1},...,w_{k}\in F_{r}'':=[[F_{r},F_{r}],[F_{r},F_{r}]]$
then $X_{\underline{w},\underline{G}}=\underline{B}^{r}$, and hence
\[
\dim\mathcal{X}_{\underline{w},\underline{G}}\geq\dim X_{\underline{w},\underline{G}}-3=\dim\underline{B}^{r}-3>\dim\mathcal{X}_{\underline{w},\underline{G}}^{Z}.
\]
However $F_{r}/F_{r}''$ is solvable, hence amenable, and Kesten's
theorem tells us that: 
\[
\Pr_{w\in F_{r}:\left|w\right|=\ell}\left[w\in F_{r}''\right]\apprge e^{-o(\ell)}.
\]
So it is \textbf{not true} that $\dim\mathcal{X}_{\underline{w},\underline{G}}=\dim\mathcal{X}_{\underline{w},\underline{G}}^{Z}$,
with probability $>1-e^{-c'\ell}$. 
\end{example}

\begin{rem}
Nonetheless, one can further show that for $r\geq3$ and $k=1$ (i.e.~$\Gamma_{w}$
is a one-relator group), we have $\dim\mathcal{X}_{\underline{w},\underline{G}}=\dim\mathcal{X}_{\underline{w},\underline{G}}^{Z}=(r-2)\dim\underline{G}$.
In other words, a word map $w_{\underline{G}}:\underline{G}^{r}\rightarrow\underline{G}$
is flat over $1\in\underline{G}$ with probability $>1-e^{-c\ell}$.
\end{rem}

\begin{example}
Here are several other examples for Theorem \ref{thm:main theorem random words}
in the case that $\underline{G}=\mathrm{SL}_{2}$, $r=2$ and $k=1$,
so $\Gamma_{w}=\langle a,b\mid w\rangle$: 
\begin{enumerate}
\item $\mathcal{X}_{w,\underline{G}}^{Z}=\varnothing$ whenever $w=ba^{n}b^{-1}a^{-m}$
and $gcd(m,n)=1$ ($\Gamma_{w}=\mathrm{BS}(m,n)$ is the \textbf{Baumslag\textendash Solitar}
group), and also whenever $w=baba^{3}b^{2}a^{2}$ or $w=bab^{-3}ab^{2}a$,
by direct computation (using the formula (\ref{eqforX}) in Example
\ref{sl2ex}). 
\item If $w=a^{2}ba^{-2}b^{-2}$, then $\Gamma_{w}$ is non-linear and residually
finite \cite{DS05}, and $\mathcal{X}_{w,\underline{G}}^{Z}=\left\{ \left(\mp\sqrt{2},-1,\pm\frac{1}{\sqrt{2}}\right)\right\} $
is zero-dimensional. 
\item If $w=uau^{-1}b^{-1}$ for $u=ab^{-1}a^{-1}b$, then $\Gamma_{w}$
is the fundamental group of the figure eight knot. Then 
\[
\mathcal{X}_{w,\underline{G}}^{Z}=\left\{ \begin{array}{c}
z^{2}+2x^{2}=1+z(1+x^{2})\\
y=x
\end{array}\right\} \text{ is of dimension }1.
\]
\item If $w=[a,v]$ for $v=[b,a]b^{-1}ab$, then $\Gamma_{w}$ is the fundamental
group of the complement of the Whitehead link, and $\mathcal{X}_{w,\underline{G}}^{Z}$
is a hypersurface of dimension $2$ given by 
\[
\mathcal{X}_{w,\underline{G}}^{Z}=\left\{ x^{2}z+y^{2}z+z^{3}=xy+2z+xyz^{2}\right\} ,
\]
For Items (3) and (4) see \cite{MR03,ABL18}. 
\item If $w=b^{5}ab^{-1}a^{-1}ba^{3}$, one can check (using (\ref{eqforX})
and Mathematica), that $\mathcal{X}_{w,\underline{G}}^{Z}$ consists
of $10$ points with algebraic coordinates permuted transitively by
the Galois action. 
\end{enumerate}
\end{example}

\subsection{Representations of surface groups}

Let $\Sigma_{g}$ be a closed orientable surface of genus $g$, and
let $\Gamma_{g}=\pi_{1}(\Sigma_{g})$ be its fundamental group. Then
\[
\Gamma_{g}=\langle(a_{1},b_{1},...,a_{g},b_{g}):\prod_{i=1}^{g}[a_{i},b_{i}]=1\rangle=\Gamma_{w_{\mathrm{com}}^{*g}}
\]
Let $X_{g}:=\mathrm{Hom}(\Gamma_{g},\underline{G})$ and $\mathcal{X}_{g}=X_{g}/\!\!/\underline{G}$
be the character variety. 
\begin{thm}[{{\cite{Li93}, \cite[Section 11]{Sim94}, \cite{BCR96}, \cite[Corollary 1.11]{LiS05b},
\cite[Theorem 2.8]{Sha09}}}]
If $\underline{G}$ is a simply connected complex semisimple group.
Then for $g\geq2$: 
\begin{enumerate}
\item $\mathcal{X}_{g}$ and $\mathcal{X}_{g}^{Z}$ have dimension $(2g-2)\dim\underline{G}$. 
\item $\mathcal{X}_{g}$ and $\mathcal{X}_{g}^{Z}$ are irreducible. 
\end{enumerate}
\end{thm}

\begin{proof}
Every semisimple complex Lie group admits a model defined over the
rationals. So without loss of generality, we may assume that $\underline{G}$
is defined over $\Q$ and it then makes sense to consider its reductions
modulo large primes. By the proof of Item (2) of Theorem \ref{thm:mixing of comutator}
restricting to the family $\left\{ \underline{G}(\mathbb{F}_{p})\right\} _{p}$,
and by Theorem \ref{thm:rep growth of Chevalley groups}(2), we have:
\[
\left\Vert \tau_{w_{\mathrm{com}},\underline{G}(\mathbb{F}_{p})}^{*g}-\mu_{\underline{G}(\mathbb{F}_{p})}\right\Vert _{\infty}\leq\left\Vert \tau_{w_{\mathrm{com}},\underline{G}(\mathbb{F}_{p})}^{*2}-\mu_{\underline{G}(\mathbb{F}_{p})}\right\Vert _{\infty}\leq\left\Vert \tau_{w_{\mathrm{com}},\underline{G}(\mathbb{F}_{p})}-\mu_{\underline{G}(\mathbb{F}_{p})}\right\Vert _{2}^{2}\leq\zeta_{\underline{G}(\mathbb{F}_{p})}(2)-1\underset{p\rightarrow\infty}{\rightarrow}0.
\]
We therefore get: 
\[
\tau_{w_{\mathrm{com}},\underline{G}(\mathbb{F}_{p})}^{*g}(e)\left|\underline{G}(\mathbb{F}_{p})\right|=\frac{\left|X_{g}(\mathbb{F}_{p})\right|}{\left|\underline{G}(\mathbb{F}_{p})\right|^{2g-1}}\underset{p\rightarrow\infty}{\rightarrow}1.
\]
Since $\underline{G}_{\mathbb{F}_{p}}$ is geometrically irreducible
for $p\gg1$, by the Lang\textendash Weil estimates (Theorem \ref{thm:Lang-Weil}),
$\left|\underline{G}(\mathbb{F}_{p})\right|=p^{\dim\underline{G}}(1+O(p^{-1/2}))$.
This implies that $\left|X_{g}(\mathbb{F}_{p})\right|p^{-(2g-1)\dim\underline{G}}\underset{p\rightarrow\infty}{\rightarrow}1$,
and again by Theorem \ref{thm:Lang-Weil} and by Chebotarev's density
theorem (Corollary \ref{cor:number of points as we vary primes }
bellow), this means that $X_{g}$ is geometrically irreducible of
dimension $(2g-1)\cdot\dim\underline{G}$. 
\end{proof}

\subsection{Mixing results in finite groups}

Let $G$ be a finite group, with generating set $\underline{x}=(x_{1},...,x_{r})$.
We consider the random walk induced by the measure $\mu:=\frac{1}{2r}\sum_{i=1}^{r}\left(\delta_{x_{i}}+\delta_{x_{i}^{-1}}\right)$,
and would like to find the $L^{\infty}$-mixing time (recall Definition
\ref{def:mixing time}), i.e.~the minimal $t_{\infty}$ such that
\[
\underset{g\in G}{\max}\left|\Pr_{w:\left|w\right|=t_{\infty}}\left[w(\underline{x})=g\right]-\frac{1}{\left|G\right|}\right|=\underset{g\in G}{\max}\left|\mu^{*t_{\infty}}(g)-\frac{1}{\left|G\right|}\right|<\frac{1}{2\left|G\right|}.
\]
We denote by $\triangle:L^{2}(G)\rightarrow L^{2}(G)$ the operator
$\triangle(f)(g):=2rf(g)-\sum_{i=1}^{r}\left(f(x_{i}g)+f(x_{i}^{-1}g)\right)$,
also called the \textbf{combinatorial Laplacian}. It is a self-adjoint
operator, with real eigenvalues $0=\lambda_{0}<\lambda_{1}\leq\lambda_{2}...\leq\lambda_{\left|G\right|-1}$,
where the constant functions are eigenvectors with eigenvalue $\lambda_{0}=0$.
Write $f_{0},f_{1},...,f_{\left|G\right|-1}$ for an orthonormal basis
of eigenvectors, where $f_{i}$ has eigenvalue $\lambda_{i}$. It
is classical that a lower bound on the first eigenvalue $\lambda_{1}$
yields an upper bound on mixing time. Concretely: 
\begin{fact}
\label{fact:mixing result}We have 
\[
\underset{g\in G}{\max}\left|\Pr_{w:\left|w\right|=\ell}\left[w(\underline{x})=g\right]-\frac{1}{\left|G\right|}\right|<e^{-\frac{\lambda_{1}\ell}{2r}}.
\]
\end{fact}

\begin{proof}
Note that $\mu^{*\ell}(g)=(2r)^{-\ell}(2r-\triangle)\circ...\circ(2r-\triangle)(\delta_{e})(g)$.
Writing $\delta_{e}=\frac{1}{\left|G\right|}+\sum_{i=1}^{\left|G\right|-1}\overline{f_{i}(e)}f_{i}$,
and since $(1+x)\leq e^{x}$ for all $x\in\R$: 
\[
\left|\mu^{*\ell}(g)-\frac{1}{\left|G\right|}\right|^{2}=\left|\sum_{i=1}^{\left|G\right|-1}(1-\frac{\lambda_{i}}{2r})^{\ell}\overline{f_{i}(e)}f_{i}(g)\right|^{2}\leq\sum_{i=1}^{\left|G\right|-1}(1-\frac{\lambda_{i}}{2r})^{2\ell}\left|f_{i}(e)\right|^{2}\cdot\sum_{i=1}^{\left|G\right|-1}\left|f_{i}(g)\right|^{2}\leq(1-\frac{\lambda_{1}}{2r})^{2\ell}\leq e^{-\frac{\lambda_{1}\ell}{r}}.
\]
In particular, the mixing time $t_{\infty}(\mu)\leq\frac{2r}{\lambda_{1}}\ln(2\left|G\right|)$. 
\end{proof}
The following fact is also classical, see e.g. \cite[Section 5]{BT16}
or \cite{LP16}. 
\begin{fact}
\label{fact:Mixing result 2}Let $\gamma=\mathrm{diam}_{\underline{x}}(G)=\inf\left\{ \ell:\forall g\in G:\exists w:\left|w\right|\leq\ell,\,w(\underline{x})=g\right\} $.
Then $\lambda_{1}\geq\frac{1}{8\gamma^{2}}$. 
\end{fact}

Combining Facts \ref{fact:mixing result} and \ref{fact:Mixing result 2}
yields: 
\begin{equation}
t_{\infty}(\mu)\leq16r\gamma^{2}\ln(2\left|G\right|).\label{eq:mixing time in terms of diameter}
\end{equation}

We now recall the following key result, which is an essential step
in establishing lower bounds on $\lambda_{1}$ for Cayley graphs following
the so-called Bourgain\textendash Gamburd machine \cite{BG08a}.
\begin{thm}[Product Theorem \cite{BGT11b}, \cite{PS16}, \cite{Hru12}]
Let $\underline{G}$ be a simple algebraic group over $\mathbb{F}_{q}$.
Let $A\subseteq\underline{G}(\mathbb{F}_{q})$ be a generating set.
Then there exists $\epsilon=\epsilon(\underline{G})>0$ s.t.: 
\[
\left|AAA\right|\geq\min\left\{ \left|\underline{G}(\mathbb{F}_{q})\right|,\left|A\right|^{1+\epsilon}\right\} .
\]
\end{thm}

\begin{cor}
\label{cor:good diameter hence mixing time}Let $\underline{G}$ be
a simply connected, simple algebraic group over $\mathbb{F}_{p}$.
For every generating set $\underline{x}=(x_{1},...,x_{r})$ of $\underline{G}(\mathbb{F}_{q})$,
writing $\mu:=\frac{1}{2r}\sum_{i=1}^{r}\left(\delta_{x_{i}}+\delta_{x_{i}^{-1}}\right)$
we have: 
\begin{enumerate}
\item $\gamma\leq O_{\underline{G}}(1)\left(\log\left|\underline{G}(\mathbb{F}_{q})\right|\right)^{O_{\underline{G}}(1)}$. 
\item $t_{\infty}(\mu)\leq O_{r,\underline{G}}(1)\cdot\left(\log\left|\underline{G}(\mathbb{F}_{q})\right|\right)^{O_{\underline{G}}(1)}$. 
\end{enumerate}
\end{cor}

We note that a celebrated conjecture of Babai asserts that the constants
$O_{\underline{G}}(1)$ in part $(1)$ (diameter bound) of the above
corollary ought to be a numerical (absolute) constant. In fact a stronger
conjecture can be made in the case of groups of bounded rank:
\begin{conjecture}[Conjecture 4.5 in \cite{breuillard-icm} for (1)]
\label{unifconj}In Corollary \ref{cor:good diameter hence mixing time}
(1)+(2), the constant $O_{\underline{G}}(1)$ in the exponent of $\log$
can be taken to be equal to $1$. 
\end{conjecture}

Conjecture \ref{unifconj} has now been proven for a large family
of finite fields. Namely, every generating set $\underline{x}=(x_{1},...,x_{r})$
of $\underline{G}(\mathbb{F}_{p})$, for any prime $p$ outside a
certain ``small'' (and possibly empty) bad set $\mathrm{Bad}_{\epsilon}$
of primes, will give rise to a Cayley graph with logarithmic diameter
and mixing time. Explicitly: 
\begin{thm}[Breuillard\textendash Gamburd \cite{BG10}, Becker\textendash Breuillard
\cite{BBa}]
\textcolor{red}{\label{thm:uniform gap for many primes}}Given $\epsilon>0$,
there exist $\delta_{\epsilon}>0$ and a possibly empty exceptional
set $\mathrm{Bad}_{\epsilon}$ of primes such that: 
\begin{enumerate}
\item $\left|\mathrm{Bad}_{\epsilon}\cap[1,T)\right|\leq T^{\epsilon}$
for any $T>1$. 
\item Every Cayley graph of $\underline{G}(\mathbb{F}_{p})$, for $p\notin\mathrm{Bad}_{\epsilon}$,
has $\lambda_{1}>\delta_{\epsilon}$. 
\end{enumerate}
\end{thm}

In the proof of this theorem, the uniformity is deduced from a uniform
spectral gap for infinite groups established in characteristic zero,
that refines Theorem \ref{thm:improved bounds on prbability of hiting id}
above, and whose proof uses Diophantine heights over $\overline{\mathbb{Q}}$.

\section{\label{sec7}Effective Lang\textendash Weil, effective Chebotarev
and proof of the main theorem}

The proof of the main theorem of Section 6, Theorem \ref{thm:main theorem random words},
will proceed by reduction modulo large primes. To be able to lift
information mod $p$ to characteristic zero it is essential to have
a good upper bound on the size of the primes of bad reduction. In
the next subsection, we provide such bounds for general $\Z$-schemes;
for detailed proofs, see \cite{BBa}.

\subsection{Good reduction}

Let $X:=\left\{ f_{1}(x)=...=f_{s}(x)=0\right\} $, with $f_{i}\in\Z[x_{1},...,x_{m}]$,
and $d=\underset{i}{\max}\deg(f_{i})$. Set $\mathrm{height}(f_{i})=\log\max\left|\mathrm{coeff}(f_{i})\right|$
and 
\[
h_{X}:=\underset{i}{\max}\text{ \ensuremath{\mathrm{height}}(\ensuremath{f_{i}})}.
\]

\begin{lem}
\label{lem:Lemma 1 good reduction}Let $Z$ be a geometrically irreducible
component of $X_{\Q}$, and let $K$ be the field of definition of
$Z$ (i.e.~$K:=\overline{\Q}^{H}$ where $H$ is the stabilizer of
$Z$ under the action of $\mathrm{Gal}(\overline{\Q}/\Q)$ on the
geometrically irreducible components of $X_{\Q}$). Then: 
\begin{enumerate}
\item $[K:\Q]\leq d^{m}$. 
\item $\log(\triangle_{K})\leq O_{m}(d^{O_{m}(1)}h_{X})$, where $\triangle_{K}$
is the discriminant of the extension $K/\Q$. 
\end{enumerate}
\end{lem}

Part (1) follows from Bezout's theorem \ref{bez}, since the number
of components of $X_{\Q}$ is bounded by its degree and part (2) follows
from the Mahler bound for discriminants and from standard bounds on
the discriminant of a compositum of number fields, see \cite{BBa}.

We now examine the Galois action. Suppose $L/\Q$ be a finite Galois
extension, and let $\mathcal{O}_{L}$ be the ring of integers of $L$.
Let $p$ be a prime number, and let $\mathfrak{p}$ be a prime ideal
of $\mathcal{O}_{L}$ dividing $p\mathcal{O}_{L}$. Denote by $L_{\mathfrak{p}}:=\mathcal{O}_{L}/\mathfrak{p}\mathcal{O}_{L}$
the residue field at $\mathfrak{p}$. The \textbf{decomposition group}
of $\mathfrak{p}$ over $p$ is defined as: 
\[
D_{\mathfrak{p},p}:=\left\{ \sigma\in\mathrm{Gal}(L/\Q):\sigma(\mathfrak{p})=\mathfrak{p}\right\} .
\]
Denote by $a\mapsto\overline{a}$ the reduction map mod $\mathfrak{p}$
(and similarly mod $p$). There is a unique map $\Psi_{\mathfrak{p},p}:D_{\mathfrak{p},p}\longrightarrow\mathrm{Gal}(L_{\mathfrak{p}}/\mathbb{F}_{p})$
$\sigma\mapsto\overline{\sigma}$ satisfying $\overline{\sigma(a)}=\overline{\sigma}(\overline{a})$,
which is a surjective homomorphism. The kernel of $\Psi_{\mathfrak{p},p}$
is the \textbf{inertia subgroup} 
\[
I_{\mathfrak{p},p}:=\left\{ \sigma\in\mathrm{Gal}(L/\Q):\sigma(\alpha)=\alpha\mod\mathfrak{p}\text{ for all }\alpha\in\mathcal{O}_{L}\right\} .
\]
Hence, $\Psi_{\mathfrak{p},p}$ induces an isomorphism $\Psi_{\mathfrak{p},p}:D_{\mathfrak{p},p}/I_{\mathfrak{p},p}\rightarrow\mathrm{Gal}(L_{\mathfrak{p}}/\mathbb{F}_{p})$.

Now suppose that $p$ is unramified in $\mathcal{O}_{L}$ (or equivalently
that $p\nmid\triangle_{L/\Q}$). In this case, $I_{\mathfrak{p},p}$
is trivial, and we have an isomorphism 
\[
\Psi_{\mathfrak{p},p}:D_{\mathfrak{p},p}\rightarrow\mathrm{Gal}(L_{\mathfrak{p}}/\mathbb{F}_{p}).
\]
Let $\mathrm{Frob}\in\mathrm{Gal}(L_{\mathfrak{p}}/\mathbb{F}_{p})$
be the Frobenius automorphism $x\mapsto x^{p}$. Then we call the
automorphism $\mathrm{Frob}_{\mathfrak{p},p}:=\Psi_{\mathfrak{p},p}^{-1}(\mathrm{Frob})\in\mathrm{Gal}(L/\Q)$
\textbf{the Frobenius automorphism} of $D_{\mathfrak{p},p}$. Note
that if $\mathfrak{p}'\in\mathcal{O}_{L}$ is another prime lying
over $p$, there is an element $\sigma\in\mathrm{Gal}(L/\Q)$ such
that $\mathfrak{p}'=\sigma(\mathfrak{p})$ and then we get that $\mathrm{Frob}_{\mathfrak{p}',p}=\sigma\circ\mathrm{Frob}_{\mathfrak{p},p}\circ\sigma^{-1}$.
We therefore define the \textbf{Frobenius symbol $\mathrm{Frob}_{p}$}
of $p$ in $L/\Q$, to be the conjugacy class $\left\{ \mathrm{Frob}_{\mathfrak{p},p}\right\} _{\mathfrak{p}\text{ lying over }p}$.
\begin{lem}[\cite{BBa}]
\textcolor{red}{\label{lem:good reduction of irred comp}}In the
setting of Lemma \ref{lem:Lemma 1 good reduction}, there exists $\triangle\in\N$
with $\log\triangle=O_{m}(d^{O_{m}(1)}h_{X})$ such that if $p\nmid\triangle$
is a prime, then the reduction mod $p$ yields a 1-to-1 correspondence
between 
\[
\left\{ \text{irreducible components of }X_{\overline{\Q}}\right\} \Longleftrightarrow\left\{ \text{irreducible components of }X_{\overline{\mathbb{F}_{p}}}\right\} ,
\]
which preserves the dimension and the action of the Frobenius automorphism
of $\mathrm{Gal}(L/\Q)$ and $\mathrm{Gal}(L_{\mathfrak{p}}/\mathbb{F}_{p})$
respectively, where $L$ is the compositum of the fields of definition
of the irreducible components of $X_{\overline{\Q}}$. 
\end{lem}

The proof involves effective elimination theory and bounds on the
degree and height of Groebner bases for ideals over $\Z$. We refer
to \cite{BBa} for details. Below we will apply Lemma \ref{lem:good reduction of irred comp}
to $X_{w,\underline{G}}$, with $d=O_{\underline{G}}(\ell)$, $h_{X}=O_{\underline{G}}(\ell)$,
and $\triangle=O_{\underline{G}}(\ell^{O_{\underline{G}}(1)})$.

\subsection{Chebotarev's density theorem}



Chebotarev's density theorem states that as one vary over primes in
large enough intervals $[\frac{1}{2}T,T]$, the map $p\longmapsto\mathrm{Frob}_{p}$
produces a uniformly random element of $\mathrm{Gal}(L/\Q)$. More
precisely: 
\begin{thm}[Chebotarev's Density Theorem 1926, \cite{Tsc26}]
\label{thm:Chebotarve density theorem}Let $L/\Q$ be a Galois extension,
with $G=\mathrm{Gal}(L/\Q)$. Let $f:G\rightarrow\C$ be a class function
on $G$. Then if $p$ is a random prime in the interval $[\frac{1}{2}T,T]$,
we have: 
\[
\E_{p}(f(\mathrm{Frob}_{p}))\underset{T\rightarrow\infty}{\longrightarrow}\E_{g\in G}(f(g)).
\]
\end{thm}

Here the notation $\E_{p}$ stands for the average over all primes
$p$ in the interval $[\frac{1}{2}T,T]$.

We will apply Theorem \ref{thm:Chebotarve density theorem} in the
setting where $f(g):=\mathrm{fix}(g)$ is the number of fixed points
of the action of $\mathrm{Gal}(L/\Q)$ on the finite set $\Omega_{X}$
of top-dimensional geometrically irreducible components of $X_{\overline{\Q}}$.
Recall that for $p>\triangle_{L/\Q}$, we have $\Omega_{X}=\Omega_{X,p}$
by Lemma \ref{lem:good reduction of irred comp} and thus $f(\mathrm{Frob}_{p})=C_{X_{\mathbb{F}_{p}}}$
is the number of geometrically irreducible components of $X_{\mathbb{F}_{p}}$
that are defined over $\mathbb{F}_{p}$, so that 
\[
\E_{p}(C_{X_{\mathbb{F}_{p}}})\underset{T\rightarrow\infty}{\longrightarrow}\E_{g\in G}(\mathrm{fix}(g)).
\]
Combining this with the Lang-Weil estimates (Theorem \ref{thm:Lang-Weil}),
we therefore deduce: 
\begin{cor}
\label{cor:number of points as we vary primes }Let $X$ be a $\Q$-variety.
Then: 
\begin{enumerate}
\item $\dim X\leq N$ $\Longleftrightarrow$ $\left|X(\mathbb{F}_{p})\right|\leq O(p^{N})$
as $p\rightarrow\infty$. 
\item $\underset{p\rightarrow\infty}{\limsup}\frac{\left|X(\mathbb{F}_{p})\right|}{p^{\dim X}}=C_{X_{\overline{\Q}}}=|\Omega_{X}|=\#$top-dim.
geometrically irreducible components of $X$. 
\item $\E_{p}\left(\frac{\left|X(\mathbb{F}_{p})\right|}{p^{\dim X}}\right)\underset{T\rightarrow\infty}{\longrightarrow}\mathrm{TDIC}(X_{\Q})$
$=\#$top-dim. $\Q$-irreducible components of $X$, or equivalently,
the number of $\mathrm{Gal}(L/\Q)$-orbits on $\Omega_{X}$. 
\end{enumerate}
\end{cor}

We now state an effective version of Chebotarev's density theorem. 
\begin{thm}[{Effective Chebotarev, \cite{LO77,Ser12}, \cite[Eq. 34 on page 137]{Ser81}}]
Let $L/\Q$ be a Galois extension. Then, under (GRH)\footnote{i.e.~under the assumption that the non-trivial zeros of the Dedekind
zeta function $\zeta_{L}$ are on the critical line.}, for every class function $f$ on $\mathrm{Gal}(L/\Q)$, 
\[
\left|\E_{p}(f(\mathrm{Frob}_{p})-\E_{g\in G}(f(g))\right|\leq40\E_{g\in G}(|f(g)|)\left(\log\triangle_{L/\Q}+[L:\Q]\right)\frac{\log(T)^{2}}{T^{1/2}}.
\]
\end{thm}

\begin{rem}
\label{rem:not good enough}By Minkowski's bound \cite[V.4]{LangANT},
$\triangle_{L/\Q}\geq\pi^{[L:\Q]}$ if $[L:\Q]\ge3$. But $[L:\Q]$
could be very large! For example, in the case of $X_{w,\underline{G}}$,
$[L:\Q]$ could be as large as $\ell!$ (since $[L:\Q]=\left|\mathrm{Gal}(L/\Q)\right|$,
and $\mathrm{Gal}(L/\Q)$ could be as large as the symmetric group
on $\ell^{O(1)}$ elements), where on the other hand, $T$ must be
taken smaller than $e^{c'\ell}$ (as $\ensuremath{\ell\gtrsim c\log p}$
is required for mixing time). This is\textbf{ not good enough}! 
\end{rem}

By restriction to a special case, namely that $f=\mathrm{fix}(g)$
for the action of $G$ on $\Omega_{X}$, and relying on \cite{GM19},
we obtain, an effective version of Chebotarev, which is suitable for
our purposes. 
\begin{thm}
\label{thm:improved effective Chebotarev}Let $L/\Q$ be a Galois
field extension, with $G=\mathrm{Gal}(L/\Q)$. Let $\Omega$ be a
transitive $G$-set, let $\omega_{0}\in\Omega$ and set $K:=L^{\mathrm{stab}_{G}(w_{0})}$.
Then under (GRH), 
\[
\left|\E_{p}(\mathrm{fix}(\mathrm{Frob}_{p})-\E_{g\in G}(\mathrm{fix}(g))\right|\leq40\left(\log\triangle_{K/\Q}+[K:\Q]\right)\frac{\log(T)^{2}}{T^{1/2}}.
\]
\end{thm}

\begin{rem}
\label{rem:good enough}In contrast with Remark \ref{rem:not good enough},
now by Lemma \ref{lem:Lemma 1 good reduction} $[K:\Q]\leq O(\ell^{O(1)})$
and $\log(\triangle_{K/\Q})\leq O(d^{O(1)}h)\leq O(\ell^{O(1)})$,
while $T$ is again around $e^{c'\ell}$. 
\end{rem}

We refer the reader to \cite{breuillard-ghosh} where Theorem \ref{thm:improved effective Chebotarev}
is discussed and generalized. It implies the following effective version
of Corollary \ref{cor:number of points as we vary primes }. 
\begin{cor}
\label{cor:applying Chebotarev}Let $X=\left\{ f_{1}(x)=...=f_{s}(x)=0\right\} $,
with $f_{i}\in\Z[x_{1},...,x_{m}]$, $d=\underset{i}{\max}\deg(f_{i})$,
and $h_{X}:=\underset{i}{\max}\text{ \ensuremath{\mathrm{height}}(\ensuremath{f_{i}})}$.
Then 
\[
\left|\E_{p\in[\frac{1}{2}T,T]}(C_{X_{\mathbb{F}_{p}}})-\mathrm{TDIC}(X_{\Q})\right|\leq O_{m}(d^{O_{m}(1)}h_{X})\frac{\log(T)^{2}}{T^{1/2}}.
\]
\end{cor}

\subsection{Sketch of the proof of the main theorem}

We now sketch the proof of Theorem \ref{thm:main theorem random words},
focusing only on the dimension formula, Item (iii). We will use the
effective Lang-Weil estimates to find the right dimension of the character
varieties together with the combination of the above effective Chebotarev
theorem (which needs sufficiently large primes to be useful) and the
uniform mixing on Cayley graphs of $\underline{G}(\mathbb{F}_{p})$
(which requires the prime to be not too large compare to the length
of the words). Fortunately the two requirements leave enough room
to find plenty of primes to complete the argument. We now pass to
the details.

Let $\delta=r-k\geq1$. We would like to show that for every simply
connected, simple algebraic group $\underline{G}$, we have $\dim X_{\underline{w},\underline{G}}^{Z}=\delta\dim\underline{G}$
for all $\underline{w}$ but an exponentially small proportion of
words of length $\ell$. The main idea is to estimate 
\[
\left|X_{\underline{w},\underline{G}}^{Z}(\mathbb{F}_{p})\right|=\left|\left\{ (x_{1},...,x_{r})\in(\underline{G}^{r})^{Z}(\mathbb{F}_{p}):w_{1}(\underline{x})=...=w_{k}(\underline{x})=1\right\} \right|,
\]
using a double counting argument, in which we average over all $\underline{w}$
of length $\ell$, and average over ($\epsilon$-good) primes $p$
in $[\frac{1}{2}T,T]$, where $T\sim e^{c\ell}$ for small $c>0$.

\textbf{Step 1} - \textit{ reduction to a mixing statement}:

Note that 
\begin{align}
 & \E_{\underline{w}\in F_{r}^{k}:\left|w_{i}\right|=\ell}\left|X_{\underline{w},\underline{G}}^{Z}(\mathbb{F}_{p})\right|=\E_{\underline{w}}\left(\sum_{\underline{x}\in(\underline{G}^{r})^{Z}(\mathbb{F}_{p})}1_{\underline{w}(\underline{x})=1}\right)\label{eq:reduction to mixing}\\
= & \sum_{\underline{x}\in(\underline{G}^{r})^{Z}(\mathbb{F}_{p})}\Pr_{\underline{w}}\left[\underline{w}(\underline{x})=1\right]=\sum_{\underline{x}\in(\underline{G}^{r})^{Z}(\mathbb{F}_{p})}\left(\Pr_{w\in F_{r},|w|=\ell}\left[w(\underline{x})=1\right]\right)^{k},\nonumber 
\end{align}
where the last equality follows since $w_{1},...,w_{k}$ are independent
(non-reduced) words.

\textbf{Step 2 -} \textit{generation mod $p$}:

By Fact \ref{fact:mixing result}, if $\underline{x}$ generates $\underline{G}(\mathbb{F}_{p})$,
then the probability distribution of $w(\underline{x})$ converges
to the Haar measure on $\underline{G}(\mathbb{F}_{p})$ as $\ell\rightarrow\infty$.
However, a priori $\underline{x}$ only lies in $(\underline{G}^{r})^{Z}(\mathbb{F}_{p})$.
In his celebrated work, Nori \cite{Nor87} classified arbitrary subgroups
of $\mathrm{GL}_{n}(\mathbb{F}_{p})$ for $p$ large. His work was
extended by Larsen\textendash Pink to all finite fields \cite{LP11}.
These results imply the following: 
\begin{thm}[Nori, Larsen\textendash Pink, \cite{Nor87,LP11}]
\label{thm:Larsen--Pink}Let $\underline{G}$ be a simply connected,
semisimple algebraic group over $\mathbb{F}_{p}$. Let $H\lneq\underline{G}(\mathbb{F}_{p})$
be a proper subgroup. Then $H$ is a subgroup of $\underline{H}(\mathbb{F}_{p})$,
for a proper algebraic subgroup $\underline{H}\leq\underline{G}$
of bounded complexity (hence degree). 
\end{thm}

The simply connectedness assumption is necessary only because otherwise
$\underline{G}(\mathbb{F}_{p})$ may have proper subgroups of bounded
index (see also Lemma \ref{lem:subgroups are of unbounded index}
below). Theorem \ref{thm:Larsen--Pink} implies: 
\begin{cor}
\label{cor:Fp points generates}If $p\gg_{\underline{G},r}1$, then
every $\underline{x}\in(\underline{G}^{r})^{Z}(\mathbb{F}_{p})$ generates
$\underline{G}(\mathbb{F}_{p})$. 
\end{cor}

Note that since we consider the asymptotics in $\ell\rightarrow\infty$,
and take $p$ in a window $[\frac{1}{2}T,T]$, where $T\sim e^{c\ell}$,
we may safely assume that the condition $p\gg_{\underline{G},r}1$
in Corollary \ref{cor:Fp points generates} is satisfied.

\textbf{Step 3 -} \textit{uniform spectral gap and Lang\textendash Weil
estimates:}

Thanks to Corollary \ref{cor:Fp points generates}, we may apply the
uniform gap result (Theorem \ref{thm:uniform gap for many primes})
to all $\underline{x}\in(\underline{G}^{r})^{Z}(\mathbb{F}_{p})$.
Hence, for every $\epsilon>0$, there exists $\delta_{\epsilon}>0$,
so that for every $p\notin\mathrm{Bad}_{\epsilon}$ and every $\underline{x}\in(\underline{G}^{r})^{Z}(\mathbb{F}_{p})$
we get from Fact \ref{fact:mixing result}, 
\begin{equation}
\left|\Pr_{w:\left|w\right|=\ell}\left[w(\underline{x})=1\right]-\frac{1}{\left|\underline{G}(\mathbb{F}_{p})\right|}\right|\leq e^{-\frac{\delta_{\epsilon}\ell}{2r}}.\label{eq:uniform gap}
\end{equation}
By (\ref{eq:reduction to mixing}) and (\ref{eq:uniform gap}), we
deduce for $\ell\gg_{\epsilon,r,\underline{G}}\log(p)$ and $p\notin\mathrm{Bad}_{\epsilon}$:
\begin{align}
 & \E_{\underline{w}\in F_{r}^{k}:\left|w_{i}\right|=\ell}\left|X_{\underline{w},\underline{G}}^{Z}(\mathbb{F}_{p})\right|\leq\sum_{\underline{x}\in(\underline{G}^{r})^{Z}(\mathbb{F}_{p})}\left(\left|\underline{G}(\mathbb{F}_{p})\right|^{-1}+e^{-\frac{\delta_{\epsilon}\ell}{2r}}\right)^{k}\nonumber \\
= & \sum_{\underline{x}\in(\underline{G}^{r})^{Z}(\mathbb{F}_{p})}\left|\underline{G}(\mathbb{F}_{p})\right|^{-k}(1+e^{-\frac{\delta_{\epsilon}\ell}{2r}}\left|\underline{G}(\mathbb{F}_{p})\right|)^{k}\nonumber \\
= & \frac{\left|(\underline{G}^{r})^{Z}(\mathbb{F}_{p})\right|}{\left|\underline{G}(\mathbb{F}_{p})\right|^{k}}+O(ke^{-\frac{\delta_{\epsilon}\ell}{2r}}\left|\underline{G}^{r+1}(\mathbb{F}_{p})\right|))=p^{\dim\underline{G}(r-k)}(1+O_{\underline{G},r}(p^{-1/2}))+o(1),\label{eq:small ball estimate for random word}
\end{align}
where $o(1)$ goes to zero as $\ell\sim\frac{1}{c}\log(T)$ for $\frac{1}{c}\gg_{\epsilon,r,\underline{G}}1$.

\textbf{Step 4 -} \textit{use of effective Chebotarev}:

By Corollary \ref{cor:applying Chebotarev}, since $h_{X_{\underline{w},\underline{G}}}<\ell^{O_{\underline{G},r}(1)}$
and since $\mathrm{Bad}_{\epsilon}\cap[\frac{1}{2}T,T]\leq T^{\epsilon}\leq T^{\frac{1}{2}}$
for $\epsilon<\frac{1}{2}$, one has: 
\begin{align*}
 & \Pr_{\underline{w}:\left|w_{i}\right|=\ell}\left[\dim X_{\underline{w},\underline{G}}^{Z}>\delta\dim\underline{G}\right]\leq\E_{\underline{w}}\left(1_{\dim X_{\underline{w},\underline{G}}^{Z}>\delta\dim\underline{G}}\cdot\mathrm{TDIC}(X_{\underline{w},\underline{G}}^{Z})\right)\\
\leq & \E_{\underline{w}}\left(1_{\dim X_{\underline{w},\underline{G}}^{Z}>\delta\dim\underline{G}}\E_{p\in[\frac{1}{2}T,T]\backslash\mathrm{Bad}_{\epsilon}}C_{\mathbb{F}_{p}}(X_{\underline{w},\underline{G}}^{Z})\right)+O_{\underline{G},r}(\ell^{O_{\underline{G},r}(1)})\frac{(\log T)^{2}}{T^{1/2}}.
\end{align*}
By the effective Lang-Weil estimates (Theorem \ref{thm:Lang-Weil}
and Remark \ref{thm:effective LW}), we have:
\[
C_{\mathbb{F}_{p}}(X_{\underline{w},\underline{G}}^{Z})=\frac{\left|X_{\underline{w},\underline{G}}^{Z}(\mathbb{F}_{p})\right|}{p^{\dim(X_{\underline{w},\underline{G}}^{Z})_{\mathbb{F}_{p}}}}+O_{\underline{G},r}(\frac{\ell^{O_{\underline{G},r}(1)}}{\sqrt{p}}),
\]
and therefore,
\begin{align*}
 & \Pr_{\underline{w}\in F_{r}^{k}:\left|w_{i}\right|=\ell}\left[\dim X_{\underline{w},\underline{G}}^{Z}>\delta\dim\underline{G}\right]\\
 & \leq\E_{\underline{w}}\left(1_{\dim X_{\underline{w},\underline{G}}^{Z}>\delta\dim\underline{G}}\E_{p\in[\frac{1}{2}T,T]\backslash\mathrm{Bad}_{\epsilon}}\frac{\left|X_{\underline{w},\underline{G}}^{Z}(\mathbb{F}_{p})\right|}{p^{\dim(X_{\underline{w},\underline{G}}^{Z})_{\mathbb{F}_{p}}}}\right)+O_{\underline{G},r}(\frac{\ell^{O_{\underline{G},r}(1)}}{T^{1/2}})+O_{\underline{G},r}(\ell^{O_{\underline{G},r}(1)})\frac{(\log T)^{2}}{T^{1/2}}.
\end{align*}
Note that if $\dim X_{\underline{w},\underline{G}}^{Z}>\delta\dim\underline{G}$
then $\dim(X_{\underline{w},\underline{G}}^{Z})_{\mathbb{F}_{p}}\geq\dim(X_{\underline{w},\underline{G}}^{Z})\geq\delta\dim\underline{G}+1$.
Combining with (\ref{eq:small ball estimate for random word}), and
swapping the expectations $\E_{\underline{w}}$ and $\E_{p\in[\frac{1}{2}T,T]\backslash\mathrm{Bad}_{\epsilon}}$,
we have: 
\begin{align*}
 & \Pr_{\underline{w}\in F_{r}^{k}:\left|w_{i}\right|=\ell}\left[\dim X_{\underline{w},\underline{G}}^{Z}>\delta\dim\underline{G}\right]\\
 & \leq\E_{p\in[\frac{1}{2}T,T]\backslash\mathrm{Bad}_{\epsilon}}\left(p^{-\delta\dim\underline{G}-1}\E_{\underline{w}}\left(1_{\dim X_{\underline{w},\underline{G}}^{Z}>\delta\dim\underline{G}}\left|X_{\underline{w},\underline{G}}^{Z}(\mathbb{F}_{p})\right|\right)\right)+O_{\underline{G},r}(\ell^{O_{\underline{G},r}(1)})\frac{(\log T)^{2}}{T^{1/2}}\\
 & \leq\E_{p\in[\frac{1}{2}T,T]\backslash\mathrm{Bad}_{\epsilon}}\left(p^{-\delta\dim\underline{G}-1}\E_{\underline{w}}\left|X_{\underline{w},\underline{G}}^{Z}(\mathbb{F}_{p})\right|\right)+O_{\underline{G},r}(\ell^{O_{\underline{G},r}(1)})\frac{(\log T)^{2}}{T^{1/2}}\\
 & \leq\E_{p\in[\frac{1}{2}T,T]\backslash\mathrm{Bad}_{\epsilon}}\left(\frac{1+o(1)}{p}\right)+O_{\underline{G},r}(\ell^{O_{\underline{G},r}(1)})\frac{(\log T)^{2}}{T^{1/2}}\\
 & =O_{\underline{G},r}(\ell^{O_{\underline{G},r}(1)})\left(\frac{(\log T)^{2}}{T^{1/2}}+\frac{1}{T}\right).
\end{align*}
This finishes the proof of Item (iii) of Theorem \ref{thm:main theorem random words}.

\section{\label{sec8}Generic fibers of convolutions of morphisms}

In $\mathsection$\ref{sec4}, we discussed a connection between the
geometry of polynomial maps, and the probabilistic properties of pushforward
measures by such maps (Proposition \ref{prop: flatness and counting points}).
By a variant of Proposition \ref{prop: flatness and counting points}
(more precisely, by \cite[Theorem 8.4]{GHb}), if $\varphi:X\to Y$
is a morphism of smooth, geometrically irreducible $\Q$-varieties,
then 
\begin{equation}
\underset{q=p^{r}\rightarrow\infty,p\gg_{\varphi}1}{\lim}\left\Vert \varphi_{*}\mu_{X(\mathbb{F}_{q})}-\mu_{Y(\mathbb{F}_{q})}\right\Vert _{\infty}=0\,\,\,\Longleftrightarrow\,\,\,\varphi\text{\, is \,}(FGI).\label{eq:algebraic characterization of L^infty mixing}
\end{equation}
In the special case that $Y=\underline{G}$ is an algebraic $\Q$-group,
finding a uniform $L^{\infty}$-mixing time $t_{\infty}$ for the
family of measures $\left\{ \varphi_{*}\mu_{X(\mathbb{F}_{q})}\right\} _{q}$
as in (\ref{eq:algebraic characterization of L^infty mixing}) is
equivalent to showing that $\varphi^{*t}:X^{t}\to\underline{G}$ is
(FGI), for every $t\geq t_{\infty}$. 

A weaker geometric condition than the (FGI) property is when $\varphi$
is (FGI) over a Zariski open set $U\subseteq Y$, or equivalently,
if the generic fiber of $\varphi$ is geometrically irreducible. This
property has an interesting probabilistic interpretation (see \cite[Theorem 2]{LST19}
and also \cite[Theorem 9.4]{GHb})
\[
\underset{q=p^{r}\rightarrow\infty,p\gg_{\varphi}1}{\lim}\left\Vert \varphi_{*}\mu_{X(\mathbb{F}_{q})}-\mu_{Y(\mathbb{F}_{q})}\right\Vert _{1}=0\,\,\,\Longleftrightarrow\,\,\,\varphi\text{ has geometrically irreducible generic fiber}.
\]
As mentioned at the end of $\mathsection$\ref{sec4}, proving that
$\varphi^{*t}$ has geometrically irreducible generic fiber is a key
step in showing that $\varphi^{*t}$ becomes (FGI) after sufficiently
many self convolutions (see Proposition \ref{prop:gen irred convoluted with flat}). 

In Theorem \ref{thm:-convolution of two word maps is generically absolutely irreducible},
we saw that the convolution of two word maps on a simply connected
simple algebraic group has geometrically irreducible generic fiber,
and we saw that simply connectedness is a crucial assumption (Example
\ref{exa:Warning(!)}). In this section, we give an analytic proof
of a generalization of Theorem \ref{thm:-convolution of two word maps is generically absolutely irreducible},
which was recently proved by Hrushovski in \cite[Appendix C]{Hru24}
using model theory. 

The following definition plays a key role: 
\begin{defn}
\label{def:simply cnnected algebraic group}Let $K$ be a field. A
connected algebraic $K$-group $\underline{G}$ is called \emph{simply
connected}, if it does not admit a non-trivial isogeny\footnote{An isogeny is a surjective homomorphism with finite kernel.}
$\psi:\underline{H}\rightarrow\underline{G}$ from a connected algebraic
$K$-group. We say that $\underline{G}$ is \emph{absolutely simply
connected} if $\underline{G}_{\overline{K}}$ is simply connected. 
\end{defn}

\begin{thm}[{\cite[Corollary C.8]{Hru24}}]
\label{thm:convolutions of morphisms has geometrically irreducible generic fiber}Let
$X$ and $Y$ be geometrically irreducible $\Q$-varieties, let $\underline{G}$
be an absolutely simply connected algebraic $\Q$-group and let $\varphi:X\rightarrow\underline{G}$
and $\psi:Y\rightarrow\underline{G}$ be dominant morphisms. Then
\begin{equation}
\underset{q=p^{r}\rightarrow\infty,p\gg_{\varphi,\psi}1}{\lim}\left\Vert (\varphi*\psi)_{*}\mu_{X(\mathbb{F}_{q})\times Y(\mathbb{F}_{q})}-\mu_{\underline{G}(\mathbb{F}_{q})}\right\Vert _{1}=0.\label{eq:want to prove}
\end{equation}
In particular, the generic fiber of $\varphi*\psi:X\times Y\rightarrow\underline{G}$
is geometrically irreducible. 
\end{thm}

Note that Theorem \ref{thm:convolutions of morphisms has geometrically irreducible generic fiber}
is an asymptotic statement about families of measures. In many cases,
an asymptotic statement can be phrased as a non-asymptotic statement
on a limit object. In this case, the limit object is the group $\underline{G}(\mathbb{F})$
where $\mathbb{F}$ is a non-principal ultraproduct of finite fields.
The group $\underline{G}(\mathbb{F})$ carries a natural probability
measure $\mu_{\underline{G}(\mathbb{F})}$, which generalizes the
family of counting measures $\left\{ \mu_{\underline{G}(\mathbb{F}_{q})}\right\} _{q}$.
Similarly, one can define $\mu_{X(\mathbb{F})}$ and $\mu_{Y(\mathbb{F})}$.
In this language, (\ref{eq:want to prove}) is now equivalent to the
non-asymptotic statement that the pushforward measure $(\varphi*\psi)_{*}\mu_{X(\mathbb{F})\times Y(\mathbb{F})}$
\textbf{is equal to} $\mu_{\underline{G}(\mathbb{F})}.$ Hrushovski
uses a very general result on convolutions of definable measures in
a suitable model theoretic framework, and shows that such measures
must be invariant with respect to a definable subgroup of $\underline{G}(\mathbb{F})$
of bounded index. To finish the proof he shows that if $\underline{G}$
is absolutely simply connected, that $\underline{G}(\mathbb{F})$
has no definable subgroups of bounded index \cite[Corollary C.5]{Hru24},
or equivalently, that for every $M\in\N$ there exists $p_{0}(M)$
such that for every $p>p_{0}(M)$ and $q=p^{r}$, the group $\underline{G}(\mathbb{F}_{q})$
has no proper subgroups of index smaller than $M$.

We now give a purely analytic proof to Theorem \ref{thm:convolutions of morphisms has geometrically irreducible generic fiber}.
Some aspects of the proof share similar components to Hrushovski's
proof, but the proof itself is different. We thank Udi for sharing
his ideas and insights on this problem.

\subsection{Analytic proof of Theorem \ref{thm:convolutions of morphisms has geometrically irreducible generic fiber}}

Let $G$ be a compact group with Haar probability measure $\mu_{G}$.
Write $\Irr(G)$ for the set of irreducible characters of $G$. For
each $\rho\in\Irr(G)$ we write $\pi_{\rho}:G\rightarrow\U(V_{\rho})$
for the irreducible representation with character $\rho$. Given a
measure $\mu$ on $G$, we denote by 
\[
A_{\mu,\rho}:=\widehat{\mu}(\pi_{\rho})=\int_{G}\pi_{\rho}(g^{-1})\mu(g),
\]
the (non-commutative) Fourier coefficient of $\mu$ at $\rho$. Note
that by the Plancherel theorem for compact groups (see e.g.~\cite[Theorem 2.3.1(2)]{App14}),
if $\mu\in L^{2}(\mu_{G})$ (i.e.~$\mu=f\mu_{G}$, with $f\in L^{2}(G,\mu_{G})$),
then we have 
\begin{equation}
\left\Vert \mu\right\Vert _{2}^{2}=\sum_{\rho\in\Irr(G)}\rho(1)\left\Vert A_{\mu,\rho}\right\Vert _{\mathrm{HS}}^{2},\label{eq:Plancherel}
\end{equation}
where $\left\Vert C\right\Vert _{\mathrm{HS}}:\sqrt{\tr(CC^{*})}$
denotes the Hilbert-Schmidt norm. Further denote by $\left\Vert C\right\Vert _{\mathrm{op}}$
the operator norm on matrices. Similarly to Definition \ref{def:L^q norms},
if $\mu=f\mu_{G}$ for a measurable function $f$, for each $q\geq1$,
we define $\left\Vert \mu\right\Vert _{q}:=\left(\int_{G}\left|f(g)\right|^{q}\mu_{G}(g)\right)^{\frac{1}{q}}$. 
\begin{defn}
\label{def:spectrally nice}Let $\mathcal{G}=\{G_{n}\}_{n\in\N}$
be a family of compact groups. A family $\{\mu_{n}\}_{n\in\N}$ of
probability measures on $\mathcal{G}$: 
\begin{enumerate}
\item Has \emph{uniform spectral decay}, if 
\[
\underset{n\rightarrow\infty}{\lim}\underset{1\neq\rho\in\Irr(G_{n})}{\sup}\left\Vert A_{\mu_{n},\rho}\right\Vert _{\mathrm{op}}=0.
\]
\item Is\emph{ almost uniform in $L^{s}$} for $1\leq s\leq\infty$ if 
\[
\underset{n\rightarrow\infty}{\lim}\left\Vert \mu_{n}-\mu_{G_{n}}\right\Vert _{s}=0.
\]
\end{enumerate}
\end{defn}

\begin{rem}
\label{rem:Jensen's inequality uniform in L^s}By Remark \ref{rem:Jensen and Young's inequality}(1),
if $\{\mu_{n}\}_{n\in\N}$ is almost uniform in $L^{s}$, then it
is almost uniform in $L^{s'}$ for every $1\leq s'\leq s$. 
\end{rem}

\begin{lem}
\label{lem:Conolution of nice, with L2 bounded L2 mixed}Let $\mathcal{G}=\{G_{n}\}_{n\in\N}$
be a family of compact groups. Let $\{\mu_{n}\}_{n\in\N}$ and $\{\gamma_{n}\}_{n\in\N}$
be families of probability measures on $\mathcal{G}$, where $\underset{n\rightarrow\infty}{\limsup}\left\Vert \mu_{n}\right\Vert _{2}<C$
and $\{\gamma_{n}\}_{n\in\N}$ has uniform spectral decay. Then the
family $\left\{ \mu_{n}*\gamma_{n}\right\} _{n}$ is almost uniform
in $L^{2}$ in $\mathcal{G}$. 
\end{lem}

\begin{proof}
For $n\gg1$, we have: 
\begin{equation}
\sum_{\rho\in\Irr(G_{n})}\rho(1)\left\Vert A_{\mu_{n},\rho}\right\Vert _{\mathrm{HS}}^{2}=\left\Vert \mu_{n}\right\Vert _{2}^{2}<C^{2}.\label{eq:0.2}
\end{equation}
Since $A_{\mu_{n}*\gamma_{n},\rho}=A_{\mu_{n},\rho}\circ A_{\gamma_{n},\rho}$,
and since $\left\Vert BD\right\Vert _{\mathrm{HS}}\leq\left\Vert B\right\Vert _{\mathrm{HS}}\left\Vert D\right\Vert _{\mathrm{op}}$
for every $B,D\in\mathrm{Mat}_{m}(\C)$, we have: 
\begin{equation}
\left\Vert A_{\mu_{n}*\gamma_{n},\rho}\right\Vert _{\mathrm{HS}}^{2}=\left\Vert A_{\mu_{n},\rho}\circ A_{\gamma_{n},\rho}\right\Vert _{\mathrm{HS}}^{2}\leq\left\Vert A_{\mu_{n},\rho}\right\Vert _{\mathrm{HS}}^{2}\left\Vert A_{\gamma_{n},\rho}\right\Vert _{\mathrm{op}}^{2}.\label{eq:0.3}
\end{equation}
By (\ref{eq:0.2}), (\ref{eq:0.3}) and since $\{\gamma_{n}\}_{n\in\N}$
has uniform spectral decay, we have: 
\[
\left\Vert \mu_{n}*\gamma_{n}-\mu_{G_{n}}\right\Vert _{2}=\sum_{1\neq\rho\in\Irr(G_{n})}\rho(1)\left\Vert A_{\mu_{n}*\gamma_{n},\rho}\right\Vert _{\mathrm{HS}}^{2}\leq\sum_{1\neq\rho\in\Irr(G_{n})}\rho(1)\left\Vert A_{\mu_{n},\rho}\right\Vert _{\mathrm{HS}}^{2}\left\Vert A_{\gamma_{n},\rho}\right\Vert _{\mathrm{op}}^{2}\rightarrow0,
\]

as $n\rightarrow\infty$, as required. 
\end{proof}
We now turn to the proof of Theorem \ref{thm:convolutions of morphisms has geometrically irreducible generic fiber}.
We start with a reduction to the case of flat morphisms from smooth
varieties. 
\begin{prop}
\label{prop:reduction to flat morphisms}It is enough to prove Theorem
\ref{thm:convolutions of morphisms has geometrically irreducible generic fiber}
in the case that $X,Y$ are smooth $\Q$-varieties and $\varphi$
and $\psi$ are flat $\Q$-morphisms. 
\end{prop}

\begin{proof}
To prove (\ref{eq:want to prove}), we can throw away any constructible
set $Z(\mathbb{F}_{q})\subseteq X(\mathbb{F}_{q})$ with $\mathrm{codim}Z\geq1$,
since by the Lang-Weil estimates (Theorem \ref{thm:Lang-Weil}), the
total contribution to (\ref{eq:want to prove}) will go to $0$ as
$q\rightarrow\infty$. We may therefore assume $X,Y$ are smooth $\Q$-varieties.
By generic flatness, there exist Zariski open subsets $U\subseteq X$
and $V\subseteq Y$, such that $\varphi|_{U}:U\rightarrow\underline{G}$
and $\psi|_{V}:V\rightarrow\underline{G}$ are flat. Since $X$ and
$Y$ are geometrically irreducible, $U\times V$ is Zariski dense
in $X\times Y$, with complement of smaller codimension, so we may
restrict to $U(\mathbb{F}_{q})\times V(\mathbb{F}_{q})$ for $q=p^{r}$
and $p\gg1$. We may therefore assume that $\varphi$ and $\psi$
are flat morphisms, as required. 
\end{proof}
\begin{notation*}
For each $M\in\N$, denote by $\mathcal{P}_{M}$ the set of all prime
powers $q=p^{r}$ for $p>M$. 
\end{notation*}
With the help of Proposition \ref{prop:reduction to flat morphisms},
Theorem \ref{thm:convolutions of morphisms has geometrically irreducible generic fiber}
follows from the following key proposition. 
\begin{prop}
\label{Prop:polynomial measures are nice}Let $X$ be a smooth, geometrically
irreducible $\Q$-variety and $\varphi:X\rightarrow\underline{G}$
be a flat $\Q$-morphism to a connected, absolutely simply connected
algebraic $\Q$-group. Then there exists $M\in\N$ such that the family
$\left\{ \varphi_{*}\mu_{X(\mathbb{F}_{q})}\right\} _{q\in\mathcal{P}_{M}}$
of probability measures on $\left\{ \underline{G}(\mathbb{F}_{q})\right\} _{q\in\mathcal{P}_{M}}$
has uniform spectral decay. 
\end{prop}

\begin{proof}[Proof that Proposition \ref{Prop:polynomial measures are nice} implies
Theorem \ref{thm:convolutions of morphisms has geometrically irreducible generic fiber}]
By Proposition \ref{prop:reduction to flat morphisms}, we may assume
that $\varphi$ and $\psi$ are flat morphisms. By Proposition \ref{Prop:polynomial measures are nice},
the family $\left\{ \varphi_{*}\mu_{X(\mathbb{F}_{q})}\right\} _{q\in\mathcal{P}_{M}}$
has uniform spectral decay. Since $\psi$ is flat, by the Lang\textendash Weil
estimates (Proposition \ref{prop: flatness and counting points},
Theorem \ref{thm:Lang-Weil}, and more precisely \cite[Theorem 8.4]{GHb}),
the density of $\psi_{*}\mu_{Y(\mathbb{F}_{q})}$ with respect to
$\mu_{\underline{G}(\mathbb{F}_{q})}$ is bounded by a constant $C>0$
independent of $q$, for $q\in\mathcal{P}_{M}$, with $M$ possibly
larger. In particular, we have: 
\begin{equation}
\sum_{\rho\in\Irr(\underline{G}(\mathbb{F}_{q}))}\rho(1)\left\Vert A_{\psi_{*}\mu_{Y(\mathbb{F}_{q})},\rho}\right\Vert _{\mathrm{HS}}^{2}=\left\Vert \psi_{*}\mu_{Y(\mathbb{F}_{q})}\right\Vert _{2}^{2}\leq\left\Vert \psi_{*}\mu_{Y(\mathbb{F}_{q})}\right\Vert _{\infty}^{2}\leq C^{2}.\label{eq:bound on L^2 norm}
\end{equation}
Theorem \ref{thm:convolutions of morphisms has geometrically irreducible generic fiber}
now follows from (\ref{eq:bound on L^2 norm}) and Lemma \ref{lem:Conolution of nice, with L2 bounded L2 mixed}. 
\end{proof}
Hence, it is left to prove Proposition \ref{Prop:polynomial measures are nice}.
Let us first characterize simply connected algebraic groups. 
\begin{lem}
\label{lem:nice Levi decomposition}Let $\underline{G}$ be a connected,
absolutely simply connected algebraic $\Q$-group. Then $\underline{G}\simeq\underline{U}\rtimes\underline{L}$,
where $\underline{U}$ is the unipotent radical of $\underline{G}$
and $\underline{L}$ is an absolutely simply connected, semisimple
algebraic $\Q$-group.
\end{lem}

\begin{proof}
By Levi's theorem in characteristic $0$ \cite{Mos56}, one has an
isomorphism $\underline{G}\simeq\underline{U}\rtimes\underline{L}$,
where $\underline{L}$ is a connected reductive subgroup of $\underline{G}$.
Since $\underline{G}$ is absolutely simply connected, so is $\underline{L}$.
Let $\underline{L}^{\mathrm{rad}}$ be the radical (also the center)
of $\underline{L}$. Then the multiplication map $\underline{L}^{\mathrm{rad}}\times[\underline{L},\underline{L}]\rightarrow\underline{L}$
is a central isogeny, and thus an isomorphism, by our assumption.
Finally, note that $(\underline{L}^{\mathrm{rad}})_{\overline{\Q}}$
is a torus, which is never simply connected unless trivial, hence
$\underline{L}$ is perfect, and thus semisimple.
\end{proof}
\begin{lem}
\label{lem:subgroups are of unbounded index}Let $\underline{G}$
be a connected, absolutely simply connected algebraic $\Q$-group.
Then there exist $M\in\N$ and an absolute constant $0<c<1$, such
that for every $q\in\mathcal{P}_{M}$, any proper subgroup $H$ of
$\underline{G}(\mathbb{F}_{q})$ is of index $\geq cq$. 
\end{lem}

\begin{proof}
Suppose that $H<\underline{G}(\mathbb{F}_{q})\simeq\underline{U}(\mathbb{F}_{q})\rtimes\underline{L}(\mathbb{F}_{q})$
is a subgroup of index $<q$. The unipotent group $\underline{U}$
is split over $\mathbb{F}_{q}$ (\cite[Corollary 15.5(ii)]{Bor91}),
and hence $H\cap\underline{U}(\mathbb{F}_{q})=\underline{U}(\mathbb{F}_{q})$.
Moreover, $\underline{L}(\mathbb{F}_{q})$ is a finite quasi-simple
group of Lie type, whenever $q\in\mathcal{P}_{M}$ for $M\in\N$ (see
e.g.~\cite[Theorem 24.17]{MT11}). Hence, by \cite{LaSe74} any $1\neq\rho\in\Irr(\underline{L}(\mathbb{F}_{q}))$
satisfies $\rho(1)\geq cq$ for some absolute constant $0<c<1$ (i.e.~$\underline{L}(\mathbb{F}_{q})$
is \emph{$cq$-quasi random}). In particular, any proper subgroup
$H'<\underline{L}(\mathbb{F}_{q})$ is of index at least $cq$ (as
otherwise the space $\C[\underline{L}(\mathbb{F}_{q})/H']$ is a representation
of dimension $<cq$). We get that $H=\underline{G}(\mathbb{F}_{q})$,
unless it is of index $\geq cq$
\end{proof}
\begin{cor}
\label{cor:Simply connected groups are almost quasi-random}Let $\underline{G}$
be a connected, absolutely simply connected algebraic $\Q$-group.
Then there exist $c>0$, and $M\in\N$ such that for every $q\in\mathcal{P}_{M}$,
$\underset{\rho\in\Irr(\underline{G}(\mathbb{F}_{q})):\rho(1)>1}{\min}\rho(1)\geq cq$. 
\end{cor}

\begin{proof}
Fix $1<N\in\N$ and let $\rho\in\Irr(\underline{G}(\mathbb{F}_{q}))$
be a character of degree $N$. Let $\underline{L},\underline{U}$
subgroups such that $\underline{G}\simeq\underline{U}\rtimes\underline{L}$
as in Levi's theorem. By the proof of Lemma \ref{lem:subgroups are of unbounded index},
whenever $q\in\mathcal{P}_{M}$ for $M\in\N$, the restriction $\rho|_{\underline{L}(\mathbb{F}_{q})}$
is either trivial or of dimension at least $cq$. Hence, it is left
to deal with the case that $\rho|_{\underline{L}(\mathbb{F}_{q})}$
is trivial. In this case, $\rho|_{\underline{U}(\mathbb{F}_{q})}$
is irreducible and non-trivial. Since $\underline{U}(\mathbb{F}_{q})$
is nilpotent, the character $\rho|_{\underline{U}(\mathbb{F}_{q})}$
is monomial, and thus of degree at least $q$.
\end{proof}
The final component we need for the proof of Proposition \ref{Prop:polynomial measures are nice}
are exponential character estimates for polynomials on varieties over
finite fields. 
\begin{lem}[{{\cite[Theorem 2]{Kow07}}}]
\label{lem:exponential sums on varieties}Let $X$ be a geometrically
irreducible $\Q$-variety. Let $f:X\rightarrow\mathbb{A}_{\Q}^{1}$
be a non-constant morphism. Then there exist constants $C>0$ and
$M\in\N$ such that for every $q\in\mathcal{P}_{M}$, and every non-trivial
additive character $\psi:(\mathbb{F}_{q},+)\rightarrow\C^{\times}$,
we have: 
\[
\left|\frac{1}{\left|X(\mathbb{F}_{q})\right|}\sum_{x\in X(\mathbb{F}_{q})}\psi(f(x))\right|\leq Cq^{-1/2}.
\]
\end{lem}

Lemma \ref{lem:exponential sums on varieties} implies the following: 
\begin{cor}
\label{cor:character estimates on varieties multivariate}Let $X$
be a $\Q$-variety and $\varphi:X\rightarrow\mathbb{A}_{\Q}^{m}$
be a dominant morphism. Then there exist constants $C>0$ and $M\in\N$
such that for every $q\in\mathcal{P}_{M}$ and every $1\neq\Psi\in\Irr(\mathbb{F}_{q}^{m})$,
one has 
\[
\left|\frac{1}{\left|X(\mathbb{F}_{q})\right|}\sum_{x\in X(\mathbb{F}_{q})}\Psi(\varphi(x))\right|\leq Cq^{-1/2}.
\]
\end{cor}

\begin{proof}
Write $\varphi=(\varphi_{1},...,\varphi_{m})$ for $\varphi_{i}:X\rightarrow\mathbb{A}_{\Q}^{1}$.
Since $\varphi$ is dominant, there are no complex numbers $c_{1},...,c_{m}$
such that $\sum_{i=1}^{m}c_{i}\varphi_{i}$ is equal to a constant.
Hence, the same statement is true over $\overline{\mathbb{F}_{p}}$
for $p\gg1$. Since any $\Psi(y)\in\Irr(\mathbb{F}_{q}^{m})$ is of
the form $\psi(a\cdot y)$ for $a\in\mathbb{F}_{q}^{m}$ and $\psi$
a fixed non-trivial additive character of $\mathbb{F}_{q}$, we get
$\Psi(\varphi(x))=\psi(a\cdot\varphi(x))$, where $a\cdot\varphi(x)=\left(\sum_{i=1}^{m}a_{i}\varphi_{i}\right)(x)$
is a non-constant polynomial. The previous Lemma \ref{lem:exponential sums on varieties}
implies the corollary. 
\end{proof}
We can now prove Proposition \ref{Prop:polynomial measures are nice}. 
\begin{proof}[Proof of Proposition \ref{Prop:polynomial measures are nice}]
Since $\varphi:X\rightarrow\underline{G}$ is flat, we can use (\ref{eq:bound on L^2 norm})
to deduce there exist $C>0$ and $M\in\N$ such that for $q\in\mathcal{P}_{M}$:
\[
\rho(1)\left\Vert A_{\varphi_{*}\mu_{X(\mathbb{F}_{q})},\rho}\right\Vert _{\mathrm{HS}}^{2}\leq C^{2},\text{ for all }\rho\in\Irr(\underline{G}(\mathbb{F}_{q})).
\]
If $\rho(1)>1$ then by Corollary \ref{cor:Simply connected groups are almost quasi-random}
with possibly enlarging $M$,
\begin{equation}
\left\Vert A_{\varphi_{*}\mu_{X(\mathbb{F}_{q})},\rho}\right\Vert _{\mathrm{op}}^{2}\leq\left\Vert A_{\varphi_{*}\mu_{X(\mathbb{F}_{q})},\rho}\right\Vert _{\mathrm{HS}}^{2}\leq\frac{C^{2}}{\rho(1)}\underset{q\rightarrow\infty}{\longrightarrow}0.\label{eq:niceness for high dimensional}
\end{equation}
It is left to consider the case when $\rho$ is a one-dimensional
character. In this case, $\rho$ factors through $\underline{G}(\mathbb{F}_{q})/[\underline{G}(\mathbb{F}_{q}),\underline{G}(\mathbb{F}_{q})]$.
Since $\underline{L}$ is perfect, we get $[\text{\ensuremath{\underline{G}},\ensuremath{\underline{G}}}]\simeq\underline{L}\ltimes\underline{U}'$
for $\underline{U}':=\underline{U}\cap[\text{\ensuremath{\underline{G}},\ensuremath{\underline{G}}}]$
so $\underline{G}/[\underline{G},\underline{G}]$ is unipotent and
$[\underline{G},\underline{G}]$ is absolutely simply connected. Note
that for sufficiently large $t\in\N$, the map $w_{\mathrm{com},\underline{G}}^{*t}:\underline{G}^{2t}\rightarrow[\underline{G},\underline{G}]$
is flat (by Theorem \ref{thm:convolutions makes flat}) and the image
of $w_{\mathrm{com},\underline{G}(\mathbb{F}_{q})}^{*t}$ is contained
in $[\underline{G}(\mathbb{F}_{q}),\underline{G}(\mathbb{F}_{q})]$.
By Proposition \ref{prop: flatness and counting points}, the index
$[\underline{G},\underline{G}](\mathbb{F}_{q})/[\underline{G}(\mathbb{F}_{q}),\underline{G}(\mathbb{F}_{q})]$
is bounded by a constant $C$ independent of $q$, for $q\in\mathcal{P}_{M}$.
By Lemma \ref{lem:subgroups are of unbounded index}, $[\underline{G},\underline{G}](\mathbb{F}_{q})=[\underline{G}(\mathbb{F}_{q}),\underline{G}(\mathbb{F}_{q})]$.
Hence, $\rho$ factors through 
\[
\underline{G}(\mathbb{F}_{q})/[\underline{G},\underline{G}](\mathbb{F}_{q})\simeq\left(\underline{G}/[\underline{G},\underline{G}]\right)(\mathbb{F}_{q})\simeq\left(\underline{U}/\underline{U}'\right)(\mathbb{F}_{q}),
\]
where the first isomorphism follows by Lang's theorem \cite[Theorem 6.1]{PR94}
and from the exact sequence of cohomologies \cite[Proposition 12.3.4]{Spr98}.
Note that $\underline{U}/\underline{U}'\simeq\mathbb{A}_{\Q}^{m}$
and $\left(\underline{U}/\underline{U}'\right)(\mathbb{F}_{q})\simeq\mathbb{F}_{q}^{m}$
for some $m\in\N$. Denote by $\Psi_{\rho}\in\Irr(\mathbb{F}_{q}^{m})\simeq\Irr(\left(\underline{U}/\underline{U}'\right)(\mathbb{F}_{q}))$
the character corresponding to $\rho$ and denote by $\widetilde{\varphi}:X\rightarrow\underline{G}/[\underline{G},\underline{G}]\simeq\mathbb{A}_{\Q}^{m}$
the composition of $\varphi$ with the quotient map $\underline{G}\rightarrow\underline{G}/[\underline{G},\underline{G}]$.
By Corollary \ref{cor:character estimates on varieties multivariate},
we have
\begin{align}
\left|A_{\varphi_{*}\mu_{X(\mathbb{F}_{q})},\rho}\right| & =\left|\sum_{g\in\underline{G}(\mathbb{F}_{q})}\rho(g^{-1})\varphi_{*}\mu_{X(\mathbb{F}_{q})}(g)\right|=\left|\frac{1}{\left|X(\mathbb{F}_{q})\right|}\sum_{x\in X(\mathbb{F}_{q})}\overline{\rho}(\varphi(x))\right|\nonumber \\
 & =\left|\frac{1}{\left|X(\mathbb{F}_{q})\right|}\sum_{x\in X(\mathbb{F}_{q})}\Psi_{\overline{\rho}}(\widetilde{\varphi}(x))\right|\leq Cq^{-1/2},\label{eq:niceness for one dimensional}
\end{align}
for $q\in\mathcal{P}_{M}$. By (\ref{eq:niceness for high dimensional})
and (\ref{eq:niceness for one dimensional}), the family $\left\{ \varphi_{*}\mu_{X(\mathbb{F}_{q})}\right\} _{q:p\gg1}$
has uniform spectral decay, as required. 
\end{proof}

\subsection{Some consequences of Theorem \ref{thm:convolutions of morphisms has geometrically irreducible generic fiber}}

In \cite[Thm B]{GH21} and \cite{GH19}, the second author and Yotam
Hendel showed that various singularities properties of morphisms are
improved by taking sufficiently many self-convolutions. In particular,
it was shown that any dominant morphism from a smooth, geometrically
irreducible $\Q$-variety $X$ to an algebraically $\Q$-group $\underline{G}$
becomes flat, with reduced fibers of rational singularities\footnote{A $\Q$-scheme of finite type $X$ has \textit{rational singularities}
if it is normal and for every resolution of singularities $\pi:\widetilde{X}\rightarrow X$,
one has $R^{i}\pi_{*}(O_{\widetilde{X}})=0$ for $i\geq1$.} after sufficiently many self-convolutions. Combining Theorem \ref{thm:convolutions of morphisms has geometrically irreducible generic fiber},
Proposition \ref{prop:gen irred convoluted with flat} and \cite[Thm B]{GH21},
we strengthen these results in the setting when $\underline{G}$ is
an absolutely simply connected algebraic $\Q$-group, by showing that
convolved morphisms eventually have geometrically irreducible fibers
as well. 
\begin{thm}
\label{thm:improved convolution theorem}Let $\varphi:X\rightarrow\underline{G}$
be a dominant $\Q$-morphism from a smooth, geometrically irreducible
$\Q$-variety $X$ to a connected, absolutely simply connected algebraic
$\Q$-group $\underline{G}$. Then: 
\begin{enumerate}
\item $\varphi^{*t}$ is flat with geometrically irreducible and reduced
fibers (and in particular FGI) for all $t\geq\mathrm{dim}\underline{G}+1$. 
\item $\varphi^{*t}$ is flat with geometrically irreducible and normal
fibers for all $t\geq\mathrm{dim}\underline{G}+2$. 
\item $\exists N\in\N$ such that for every $t>N$, the map $\varphi^{*t}$
is flat with geometrically irreducible fibers with rational singularities. 
\end{enumerate}
\end{thm}

The condition of rational singularities in Item (3) is important because
of the following proposition, which is a generalization of Proposition
\ref{prop: flatness and counting points}: 
\begin{prop}[{{\cite[Theorems A and 4.7]{CGH23}}}]
\label{prop:characterization of FRS and FGI}Let $\varphi:X\rightarrow Y$
be a dominant morphism of smooth, geometrically irreducible $\Q$-varieties.
Then: 
\begin{enumerate}
\item $\varphi$ is flat with fibers of rational singularities (abbreviated
\textbf{``FRS''}) if and only if there exists $C>0$ such that for
every prime $p\gg1$, the density of $\varphi_{*}\mu_{X(\Zp)}$ with
respect to $\mu_{Y(\Zp)}$ is bounded by $C$, where $\mu_{X(\Zp)}$
and $\mu_{Y(\Zp)}$ denote the canonical measures on $X(\Zp)$ and
$Y(\Zp)$\footnote{See e.g.~\cite[Lemma 4.2]{CGH23} for the definition of the canonical
measure.}. 
\item $\varphi$ is (FRS) and (FGI) if and only if for every prime $p\gg1$,
\[
\left\Vert \varphi_{*}\mu_{X(\Zp)}-\mu_{Y(\Zp)}\right\Vert _{\infty}<Cp^{-1/2},
\]
\end{enumerate}
\end{prop}

We now turn to the second application which is about Lie algebra word
maps. 
\begin{defn}
A \emph{Lie algebra word} $w(X_{1},...,X_{r})$ is a $\Q$-linear
combination of iterated commutators, or more precisely, an element
in a free $\Q$-Lie algebra $\mathcal{L}_{r}$ on a finite set $\{X_{1},\ldots,X_{r}\}$.
The space $\mathcal{L}_{r}$ has a natural gradation, so one can define
the \textit{degree} of a word $w$ as the maximal grade $d\in\N$
in which the image of $w$ is non-trivial. For example, $w=[X,Y]+2[[X,Y],Y]$
is a Lie algebra word of degree $3$. For each Lie algebra $\g$,
$w\in\mathcal{L}_{r}$ induces a \emph{Lie algebra word map} $w_{\g}:\g^{r}\rightarrow\g$. 
\end{defn}

In \cite[Theorem A]{GHb}, the second author and Hendel have shown
that non-zero word maps $w_{\g}:\g^{r}\rightarrow\g$ on simple Lie
algebra becomes (FRS) after a number $t(w)$ of self-convolutions
which is independent of $\g$. Theorem \ref{thm:convolutions of morphisms has geometrically irreducible generic fiber},
and Propositions \ref{prop:gen irred convoluted with flat} and \ref{prop:characterization of FRS and FGI}
imply the following improvement of \cite[Theorem A]{GHb}, which further
shows that these maps have geometrically irreducible fibers:
\begin{cor}[{{cf.~\cite[Theorem A]{GHb}}}]
\label{Cor:imporved word map theorem}Let $w\in\mathcal{L}_{r}$
be a Lie algebra word of degree $d$. Then there exists $0<C<10^{6}$,
such that for every simple $\Q$-Lie algebra $\mathfrak{g}$, for
which $0\neq w_{\mathfrak{g}}$, one has: 
\begin{enumerate}
\item If $t\geq Cd^{3}$ then $w_{\mathfrak{g}}^{*t}$ is (FGI). 
\item If $t\geq Cd^{4}$ then $w_{\mathfrak{g}}^{*t}$ is (FRS) and (FGI). 
\item There exists a constant $C'=C'(\g,w)$, such that if $t\geq Cd^{4}$,
one has 
\[
\left\Vert (w_{\mathfrak{g}}^{*t})_{*}\mu_{\g(\Zp)}^{rt}-\mu_{\g(\Zp)}\right\Vert _{\infty}<C'p^{-1/2},
\]
where $\mu_{\g(\Zp)}$ is the Haar probability measure on $\g(\Zp)$.
In particular, the $L^{\infty}$-mixing time $t_{\infty}$ of the
random walk on $\g(\Zp)$ induced by the word measure $(w_{\mathfrak{g}})_{*}\mu_{\g(\Zp)}^{r}$,
is smaller than $Cd^{4}$, whenever $p\gg_{w,\g}1$. 
\end{enumerate}
\end{cor}

\bibliographystyle{alpha}
\bibliography{bibfile}

\end{document}